\documentclass[11pt]{amsart}
\usepackage{amssymb, latexsym, mathrsfs,   color }
\usepackage[all]{xy}
\usepackage[colorlinks=true, pdfstartview=FitV,linkcolor=blue,citecolor=blue,urlcolor=blue]{hyperref}
\usepackage{chngcntr}
\counterwithin{figure}{section}

    \advance \topmargin by -\headheight
    \advance \topmargin by -\headsep

    \evensidemargin \oddsidemargin
\newtheorem {theorem}    {Theorem}[section]

\newtheorem {lemma}      [theorem]    {Lemma}
\newtheorem {corollary}  [theorem]    {Corollary}
\newtheorem {proposition}[theorem]    {Proposition}

\newcommand{\bb}{\mathbb}
\renewcommand{\rm}{\mathrm}

\newcommand{\cal}{\mathcal}

\newcommand{\GGL}{\mathrm{GL}}

\renewcommand{\sp}{\mathrm{Sp}}
\renewcommand{\o}{\mathrm{O}}
\newcommand{\so}{\mathrm{SO}}
\newcommand{\Fq}{\bb{F}_q}
\newcommand{\fq}{(\bb{F}_q)}
\theoremstyle{definition}
\newtheorem{definition}[theorem]{Definition}
\newtheorem{remark}[theorem]{Remark}

\newcommand{\E}{\mathcal{E}}
\newcommand{\pl}{\pi_{\Lambda}}

\newcommand{\pll}{\pi_{\Lambda'}}
\newcommand{\prll}{\pi_{\rho,\Lambda,\Lambda'}}

\newcommand{\prwll}{\pi_{\rho,\Lambda,\Lambda'}}

\newcommand{\prllc}{\pi_{\rho,\Lambda_1,\Lambda_1'}}
\newcommand{\prlld}{\pi_{\rho,\Lambda_2,\Lambda_2'}}
\newcommand{\pkh}{\pi_{\rho,k,h}}
\newcommand{\pkhp}{\pi_{\rho',k',h'}}

\newcommand{\pw}{{\pi}}

\newcommand{\wla}{\widetilde{\Lambda}}

\newcommand{\dg}{G^{(1)}(s)}
\newcommand{\ddg}{G^{(2)}(s)}
\newcommand{\ddgg}{(G^{(2)}(s))^*}
\newcommand{\dddg}{G^{(3)}(s)}

\newcommand{\sgn}{\rm{sgn}\cdot}

\newcommand{\df}{G^{\prime(1)}(s')}
\newcommand{\ddf}{G^{\prime(2)}(s')}

\newcommand{\dddf}{G^{\prime(3)}(s')}

\newcommand{\p}{\pi^{(1)}}
\newcommand{\pp}{\pi^{(2)}}
\newcommand{\ppp}{\pi^{(3)}}

\newcommand{\CD}{{\mathcal{D}}}

\newcommand{\UUl}{\underline}
\newcommand{\e}{{\epsilon'}}
\newcommand{\ee}{\epsilon_{-1}}

\numberwithin{equation}{section}

\begin{document}

\title{On the Gan-Gross-Prasad problem for finite classical groups}

\date{\today}

\author[Zhicheng Wang]{Zhicheng Wang}

\address{School of Mathematical Science, Soochow University, Suzhou 310027, Jiangsu, P.R. China}

\email{11735009@zju.edu.cn}
\begin{abstract}
In this paper we study the Gan-Gross-Prasad problem for finite classical groups. Our results provide complete answers for unipotent representations, and we obtain the explicit branching laws for these representations. Moreover, for arbitrary representations, we give a formula to reduce the Gan-Gross-Prasad problem to the restriction problem of Deligne-Lusztig characters.
\end{abstract}

\maketitle

\section{Introduction}

Let $\overline{\mathbb{F}}_q$ be an algebraic closure of a finite field $\mathbb{F}_q$, which is of characteristic $p>2$. Consider a connected reductive algebraic group $G$ defined over $\Fq$, with Frobenius map $F$. Let $Z$ be the center of $G^F$. We will assume that $q$ is large enough such that the main theorem in \cite{S} holds, namely assume that
 \begin{itemize}
 \item $T^F/Z$ has at least two Weyl group orbits of regular characters, for every $F$-stable maximal torus $T$ of $G$.
 \end{itemize}
 Let $H$ be a subgroup of $G$. Let $\pi$ (resp. $\sigma$) be a representation of $G^F$ (resp. $H^F$). We write
\[
\langle \pi,\sigma \rangle_{H^F} = \dim \mathrm{Hom}_{H^F}(\pi,\sigma ).
\]

In this paper, we focus on Gan-Gross-Prasad problem for orthogonal and symplectic groups over finite fields.
Let $V$ be an $\Fq$-vector space endowed with a nondegenerate bilinear form $(,)$ with sign $\epsilon$, i.e. $(v, w)=\epsilon (w,v)$ for any $v, w\in V$. Moreover, suppose that $W\subset V$ is a non-degenerate subspace satisfying:
\begin{itemize}

\item $\epsilon \cdot (-1)^{\rm{dim} W^\bot} = -1$,

\item $W^\bot$ is a split space.
\end{itemize}
According to whether $n-m$ is odd or even, Gan-Gross-Prasad problem is called the Bessel case or Fourier-Jacobi case.

\subsection{Bessel case}

We first consider the Bessel case. Let $V_n$ be an $n$-dimensional space over $\mathbb{F}_{q}$ with a nondegenerate symmetric bilinear form $(,)$, which defines the special orthogonal group $\so(V_n)$. We will consider various pairs of symmetric spaces $V_n\supset V_{n-2\ell}$ and the following partitions of $n$,
\[
\UUl{p}_\ell=[2\ell+1, 1^{n-2\ell-1}],\quad 0\leq \ell< n/2.
\]
Assume that  $V_n$ has a decomposition
\[
V_n=X+V_{n-2\ell}+X^\vee
\]
where $X+X^\vee=V_{n-2\ell}^\perp$ is a polarization. Let $\{e_1,\ldots, e_\ell\}$ be a basis of $X$, $\{e_1',\ldots, e_\ell'\}$ be the dual basis of $X^\vee$, and
let $X_i=\rm{Span}_{\mathbb{F}_{q}}\{e_1,\ldots, e_i\}$, $i=1,\ldots, \ell$. Let $P$ be the parabolic subgroup of $\so(V_n)$ stabilizing the flag
\[
X_1\subset\cdots\subset X_\ell,
\]
so that its Levi component is $M\cong \GGL_1^\ell\times \so(V_{n-2\ell})$. Its unipotent radical can be written in the form
\[
N_{\UUl{p}_\ell}=\left\{n=
\begin{pmatrix}
z & y & x\\
0 & I_{n-2\ell} & y'\\
0 & 0 & z^*
\end{pmatrix}
: z\in U_{\GGL_\ell}
\right\},
\]
where the superscript ${}^*$ denotes the transpose inverse, and $U_{\GGL_\ell}$ is the subgroup of unipotent upper triangular matrices of $\GGL_\ell$.
Fix a nontrivial additive character $\psi$ of $\mathbb{F}_{q}$. Pick up an anisotropic vector $v_0\in V_{n-2\ell}$ and define a generic character $\psi_{\UUl{p}_\ell, v_0}$ of $N_{\UUl{p}_\ell}(\Fq)$ by
\[
\psi_{\UUl{p}_\ell, v_0}(n)=\psi\left(\sum^{\ell-1}_{i=1}z_{i,i+1}+(y_\ell, v_0)\right), \quad n\in N_{\UUl{p}_\ell}(\Fq),
\]
where $y_\ell$ is the last row of $y$. The identity component of the stabilizer of $\psi_{\UUl{p}_\ell, v_0}$ in $M(\Fq)$ is the special orthogonal group $\rm{SO}(W)$, where $W$ is the orthogonal complement of $v_0$ in $V_{n-2\ell}$.
Put
\begin{equation}\label{hnu}
H=\so(W)\ltimes N_{\UUl{p}_\ell},\quad \nu=\psi_{\UUl{p}_\ell, v_0}.
\end{equation}
Let $\pi$ and $\pi'$ be two irreducible cuspidal representations of $\so(V_{n})$ and $\so(W)$ respectively. The Gan-Gross-Prasad problem is concerned with the multiplicity:
\begin{equation}\label{ggpb}
m(\pi,\pi'):=\dim\rm{Hom}_{H(\Fq)}(\pi, \pi'\otimes\nu).
\end{equation}

Note that $m(\pi,\pi')$ depends on the choice of $v_0$. Let $Q$ be the quadratic form associated to $(,)$.
Pick up two anisotropic vectors $v_0$, $v_0'\in V_{n-2\ell}$ such that $Q(v_0)/Q(v_0')$ is a non-square in $\Fq$. The identity component of the stabilizer of $\psi_{\UUl{p}_\ell, v_0'}$ in $M(\Fq)$ is the special orthogonal group $\so(W')$ of the orthogonal complement $W'$ of $v_0'$ in $V_{n-2\ell}$. If $n-2\ell$ is even, then $\so(W)\cong \so(W')$, but the groups $\so(W)$, $\so(W')$ are not conjugate in $\so(V_{n-2\ell})$.
If $n-2\ell$ is odd, then there are two choices of anisotropic vectors $v_0$, $v_0'\in V_{n-2\ell}$ such that $W$ is split but $W'$ not. Thus we get $\so(W)\ncong \so(W')$ in this case. In general, we have
\begin{equation}\label{disc}
\rm{disc} \ V= (-1)^{n-1} \cdot Q(v_0) \cdot \rm{disc} \ W,
\end{equation}
where both sides are regarded as square classes in $\Fq^\times/ (\Fq^\times)^2\cong\{\pm1\}$. Here the discriminant is normalized by
\[
\rm{disc} \ V=(-1)^{n(n-1)/2}\det V \in \Fq^\times/ (\Fq^\times)^2,
\]
such that when $\dim V$ is even, $\rm{disc} \ V=+1$ if and only if $\rm{SO}(V)$ is split.

The above discussions are valid for full orthogonal groups as well. In this paper, we will focus on the full orthogonal groups case. If $\rm{dim}V-\rm{dim}W=1$, then
 \begin{equation}\label{1o}
m(\pi,\pi')=\langle\pi, \pi'\rangle_{G(W)}.
\end{equation}
We call the above multiplicity the basic case for Bessel case.

\subsection{Fourier-Jacobi case}

We next turn to the Fourier-Jacobi case. Let $W_{2n}$ be a symplectic space of dimension $2n$ over $\Fq$, which gives the symplectic group $\sp_{2n}(\Fq)$. Consider pairs of symplectic spaces $W_{2n}\supset W_{2n-2\ell}$ and partitions
\[
\UUl{p}'_\ell=[2\ell, 1^{2n-2\ell}],\quad 0\leq \ell\leq n.
\]
We use similar notations for various subspaces and subgroups as in the Bessel case. Note that if we let $P_\ell$ be the parabolic subgroup of $\rm{Sp}_{2n}$ stabilizing $X_\ell$ and let $N_\ell$ be its unipotent radical, then $N_{\UUl{p}_\ell}=U_{\GGL_\ell}\ltimes N_\ell$. Let $\omega_\psi$ be the Weil representation (see \cite{Ger}) of $\sp_{2(n-\ell)}(\Fq)\ltimes \mathcal{H}_{2n-2\ell}$ depending on $\psi$, where $\mathcal{H}_{2n-2\ell}$ is the Heisenberg group of $W_{2n-2\ell}$. Roughly speaking, there is a natural homomorphism $N_\ell(\Fq)\to \mathcal{H}_{2n-2\ell}$ invariant under the conjugation action of $U_{\GGL_\ell}(\Fq)$ on $N_\ell(\Fq)$, which enables us to view $\omega_\psi$ as a representation of $\sp_{2(n-\ell)}(\Fq)\ltimes N_{\UUl{p}_\ell}(\Fq)$. Let $\psi_\ell$ be the character of $U_{\GGL_\ell}(\Fq)$ given by
\[
\psi_\ell(z)=\psi\left(\sum^{\ell-1}_{i=1}z_{i,i+1}\right),\quad z\in U_{\GGL_\ell}(\Fq).
\]
For the Fourier-Jacobi case, put
\begin{equation}\label{hnu'}
H=\sp_{2(n-\ell)}\ltimes N_{\UUl{p}_\ell},\quad \nu=\omega_\psi\otimes\psi_\ell.
\end{equation}
Similar to the Bessel case, the Gan-Gross-Prasad problem is concerned with the multiplicity:
\begin{equation}\label{ggpfj}
m_\psi(\pi,\pi'):=\rm{dim}\rm{Hom}_{H(\Fq)}(\pi, \pi'\otimes\nu).
\end{equation}
If $\ell=0$, then
   \begin{equation}\label{1sp}
m_\psi(\pi,\pi')=\langle\pi\otimes\overline{\omega_\psi}, \pi'\rangle_{G(V)}.
\end{equation}
We call the above multiplicity the basic case for Fourier-Jacobi case.

 In the $p$-adic case, the local Gan-Gross-Prasad conjecture \cite{GP1, GP2, GGP1} provides explicit answers. To be a little more precise, let $G$ be a classical group defined over a local field and $\pi$ belongs to a generic Vogan $L$-packet. The multiplicity one property holds for this situation, namely
\[
m(\pi, \sigma) \ (\textrm{resp. }m_{\psi}(\pi, \sigma))\leq 1,
\]
and the invariants attached to $\pi$ and $\sigma$ that detect the multiplicity $m(\pi,\sigma)$ is the local root number associated to their Langlands parameters.  In the $p$-adic case, the local Gan-Gross-Prasad conjecture has been resolved by J.-L. Waldspurger and C. M\oe glin and J.-L. Waldspurger \cite{W2, W3, W4, MW} for orthogonal groups, by R. Beuzart-Plessis \cite{BP1, BP2} and W. T. Gan and A. Ichino \cite{GI} for unitary groups, and by H. Atobe \cite{Ato} for symplectic-metaplectic groups. On the other hand, D. Jiang and L. Zhang \cite{JZ1} study the local descents for $p$-adic orthogonal groups, whose results can be viewed as a refinement of the local Gan-Gross-Prasad conjecture, and the descent method has important applications towards the global problem (see \cite{JZ2}).

There are also some multiplicity one results over finite fields, proved via known multiplicity one result for local field. However, we can not get the the multiplicity one in the Gan-Gross-Prasad problem for arbitrary representations directly in this way. In previous works \cite{LW2,LW3,LW4}, we have studied the Gan-Gross-Prasad problem of unipotent representations of finite unitary groups and the descent problem of unipotent cuspidal representations of finite orthogonal groups and finite symplectic groups. In this paper, we generalize our previous results and our main tool is the theta correspondence over finite fields.  To apply the theta correspondence, we first show that the parabolic induction preserves the multiplicity (\ref{ggpb}) and (\ref{ggpfj}), and thereby make a reduction to the basic case as follows. For the Bessel case, let $l=\frac{n+1-m}{2}$, and let $P$ be an $F$-stable maximal parabolic subgroup of $\rm{O}(V_{n+1})$ with Levi factor $\GGL_{l} \times \rm{O}(W)$. There exists a cuspidal representation $\tau$ satisfying some technical conditions such that
\begin{equation}\label{2o}
m(\pi,\sigma)=\langle\pi,I^{\rm{O}(V_{n+1})}_{P}(\tau\otimes\sigma)\rangle_{\rm{O}(V_{n})}
\end{equation}
where $I^{\rm{O}({V_{n+1}})}_{P}(\tau\otimes\sigma)$ is the parabolic induction. For Fourier-Jacobi case, let $l=n-m$, and let $P$ be an $F$-stable maximal parabolic subgroup of $\sp_{2n}$ with Levi factor $\GGL_{l} \times \sp_{2m}$. There exists a cuspidal representation $\tau$ as before such that
\begin{equation}\label{2sp}
m_{\psi}(\pi, \sigma)=\langle\pi,I^{\sp_{2n}}_{P}(\tau\otimes\sigma)\otimes\omega_{n,\psi}\rangle_{\sp_{2n}\fq}
\end{equation}
where $I^{\sp_{2n}}_{P}(\tau\otimes\sigma)$ is the parabolic induction. Then we compute the right side of (\ref{2o}) and (\ref{2sp}) by the standard arguments of theta correspondence and see-saw dual pairs, which are used in the proof of local Gan-Gross-Prasad conjecture (see \cite{GI,Ato}). We can conclude that each cases can be reduced to the multiplicity (\ref{1o}).

 For an $F$-stable maximal torus $T$ of $G$ and a character $\theta$ of $T^F$,  let $R_{T,\theta}^G$ be the virtual character of $G^F$ defined by P. Deligne and G. Lusztig in \cite{DL}. We say a complex irreducible representation is uniform if it is a linear combination of the Deligne-Lusztig characters. In \cite{R}, Reeder consider the multiplicity (\ref{1o}) for Deligne-Lusztig characters on the special orthogonal groups, and he gives a explicit formula.  In other words, if both $\pi$ and $\sigma$ are uniform, then (\ref{1o}) can be calculated by Reeder's formula.

 Our main result is to give a formula to reduce the the multiplicity in above three problems to the uniform case. Although in general, explicit calculation with Reeder's formula is still quite involved, we can give some explicit results for some interesting cases. For example, we can give the multiplicity one for unipotent representations, and have branching laws for these representations.

\subsection{Classification of the irreducible representations of finite orthogonal groups and finite symplectic groups}

Let $G$ be a reductive group defined over $\Fq$, and let $\E(G)=\rm{Irr}(G^F)$ be the set of irreducible representations of $G^F$. Let $P$ be an $F$-stable parabolic subgroup of $G$ with Levi decomposition $P=LV$. We focus on classical groups. Assume that the Levi subgroup $L$ is of the form $\GGL_{n}\times G'$ where $G'$ is a classical group with the same type of $G$. For any irreducible cuspidal representation $\sigma \in \cal{E}(G^{\prime })$, let
\[
\cal{E}(G,\sigma)=\{\pi\in \cal{E}(G)|\langle\pi,I_{\GGL_{n}\times G'}^G(\rho\otimes\sigma)\rangle_{G^F}\ne 0\textrm{ for some }\rho\in \E(\GGL_{n})\}.
\]
It is easily seen that for every irreducible $\pi\in\cal{E}(G)$, there exists exactly one of pair $(G^{\prime },\sigma)$ such that $\pi\in\cal{E}(G,\sigma)$.

Let $G^*$ be the dual group of $G$. We still denote the Frobenius endomorphism of $G^*$ by $F$. Then there is a natural bijection between the set of $G^F$-conjugacy classes of $(T, \theta)$ and the set of $G^{*F}$-conjugacy classes of $(T^*, s)$ where $T^*$ is a $F$-stable maximal torus in $G^*$ and $s \in   T^{*F}$. We will also denote $R_{T,\theta}^G$  by $R_{T^*,s}^G$ if $(T, \theta)$ corresponds to $(T^*, s)$.
For a semisimple element $s \in G^{*F}$, define Lusztig series as follows:
\[
\mathcal{E}(G,s) = \{ \pi \in \mathcal{E}(G)  :  \langle \pi, R_{T^*,s}^G\rangle \ne 0\textrm{ for some }T^*\textrm{ containing }s \}.
\]
And
\[
\mathcal{E}(G)=\coprod_{(s)}\mathcal{E}(G,s)
\]
where $(s)$ runs over the conjugacy classes of semisimple elements. Moreover, there is a bijection
\[
\mathcal{L}_s:\mathcal{E}(G,s)\to \mathcal{E}(C_{G^{*}}(s),1),
\]
extended by linearity to a map between virtual characters satisfying that
\[
\mathcal{L}_s(\varepsilon_G R^G_{T^*,s})=\varepsilon_{C_{G^{*}}(s)} R^{C_{G^{*}}(s)}_{T^*,1}.
\]
In particular, Lusztig correspondence sends cuspidal representations to cuspidal representations and sends uniform representations to uniform representations. We say an irreducible representation $\pi\in\cal{E}(G,s)$ is unipotent (resp. quadratic unipotent) if $s=1$ (resp. $s^2=1$).

Let $G$ be a symplectic group or orthogonal group. We have a modified Lusztig
corespondence with three groups $\dg$, $\ddg$ and $\dddg$ (c.f. \cite{P4} and section \ref{3.2} for details). Our notation is slightly different from that
of \cite{P4}: the group $G^{(2)}(s)$ always associates with eigenvalue 1. Let $\cal{L}'_s$ be the modified Lusztig
corespondence defined in Subsection \ref{mlus}. The modified Lusztig
corespondence is equal to the Lusztig correspondence if $G$ is an orthogonal group. We have
  \[
\cal{L}'_s: \cal{E}(G,s)\to
\left\{
  \begin{aligned}
 &\cal{E}(\dg\times\ddg\times\dddg,1)\times\{\pm\}&\textrm{ if }G\textrm{ is odd orthogonal};\\
  &\cal{E}(\dg\times\ddg\times\dddg,1)&\textrm{ otherwise}.
\end{aligned}
\right.
 \]
where
\begin{itemize}

\item $\dg$ is a product of general linear groups and unitary groups;

\item  If $G=\o_{2l+1}$, then $\ddg$ and $\dddg$ are symplectic groups.

\item  If $G=\o^\pm_{2l}$, then $\ddg$ and $\dddg$ are even orthogonal groups.

\item  If $G=\sp_{2l}$, then $\ddg$ is a symplectic group and $\dddg$ is an even orthogonal group.
\end{itemize}

We now review some results on the classification of the irreducible unipotent representations by Lusztig in \cite{L1,L2,L3}. We follow the notation of \cite{P3}, which is slightly different from that of \cite{L1}.

A symbol is an array of the form
\[
\Lambda
=
\begin{pmatrix}
a_1,a_2,\cdots,a_{m_1}\\
b_1,b_2,\cdots,b_{m_2}
\end{pmatrix}.
\]
We always assume
that $a_i>a_{i+1}$ and $b_i>b_{i+1}$. Let
\[
\begin{aligned}
&\rm{rank}(\Lambda)=\sum_{a_i\in A}a_i+\sum_{b_i\in B}b_i-\left\lfloor\left(\frac{|A|+|B|-1}{2}\right)^2\right\rfloor, \\
&\rm{def}(\Lambda)=|A|-|B|.
\end{aligned}
\]
Note that the definition of $\rm{def}(\Lambda)$ differs from that of \cite[p.133]{L1}. For a symbol $\Lambda=\begin{pmatrix}
A\\
B
\end{pmatrix}$, let $\Lambda^*$ (resp. $\Lambda_*$) denote the first row (resp. second row) of $\Lambda$, and let
$\Lambda^t=\begin{pmatrix}
B\\
A
\end{pmatrix}$. It is clear that $\rm{rank}(\Lambda^t)=\rm{rank}(\Lambda)$ and $\rm{def}(\Lambda^t)=-\rm{def}(\Lambda)$. Then Lusztig gives a bijection between the unipotent representations of these groups to equivalence classes of symbols as follow:
\[
\left\{
\begin{aligned}
&\cal{E}(\sp_{2n},1)\\
&\cal{E}(\o_{2n+1},1)\\
&\cal{E}(\o^+_{2n},1)\\
&\cal{E}(\o^-_{2n},1)
\end{aligned}\right.
\longrightarrow
\left\{
\begin{aligned}
&\cal{S}_n:=\big\{\Lambda|\rm{rank}(\Lambda)=n, \rm{def}(\Lambda)=1\ (\textrm{mod }4)\big\};\\
&\cal{S}_n\times\{\pm\};\\
&\cal{S}^+_n:=\big\{\Lambda|\rm{rank}(\Lambda)=n, \rm{def}(\Lambda)=0\ (\textrm{mod }4)\big\};\\
&\cal{S}^-_n:=\big\{\Lambda|\rm{rank}(\Lambda)=n, \rm{def}(\Lambda)=2\ (\textrm{mod }4)\big\}.
\end{aligned}\right.
\]
If $G$ is an even (resp. odd) orthogonal group, it is known that $\pi_{\Lambda^t} = \rm{sgn}\cdot\pi_{\Lambda}$ (resp. $\pi_{\Lambda,\epsilon}=\sgn\pi_{\Lambda,-\epsilon}$) where $\pi_{\Lambda}$ (resp. $\pi_{\Lambda,\epsilon}$) means the irreducible representation parametrized by $\Lambda$ (resp. $(\Lambda,\epsilon)$) and $\rm{sgn}$ denotes the sign character. Here we distinguish $\pi_{\Lambda,\pm}$ by decreeing that $\pi_{\Lambda,\pm}(-1)=\pm\rm{Id}$.

Let $\pi$ be an irreducible representation of $\sp_{2n}\fq$, $\o^\epsilon_{2n}\fq$ or $\o_{2n+1}\fq$. Suppose that
\[
\cal{L}'_s(\pi)=\p\otimes\pp\otimes\ppp=\rho\otimes\pl\otimes\pll,\textrm{ (resp. }\rho\otimes\pl\otimes\pll\otimes\epsilon\textrm{)}.
\]
where $\cal{L}'_s$ is the modified Lusztig correspondence. Then we denote $\pi$ by $\prll$ (resp. $\pi_{\rho,\Lambda,\Lambda',\epsilon}$). If $\dg$ is trivial, then we denote $\prll$ by $\pi_{-,\Lambda,\Lambda'}$. Similar notation applies for $\ddg$ and $\dddg$. If $\pi=\pi_{\rho,-,-}$, then we denote it briefly by $\pi_{\rho}$. If $\pi=\pi_{-,\Lambda,-}$ (resp. $\pi_{-,\Lambda,-,\epsilon}$), then it is unipotent and $\pi=\pl$ (resp. $\pi_{\Lambda,\epsilon}$).

It is worth pointing out that there is not a canonical choice of modified Lusztig correspondences. If we fix a choice of modified Lusztig correspondences for every Lusztig series, then we fix a parametrization of irreducible representations. For any $\pi\in\cal{E}(G)$, let
\[
\cal{L}'_G(\pi):=\cal{L}'_s(\pi)\textrm{ if }\pi\in\cal{E}(G,s).
\]
From now on, for every $G$, we fix a choice of $\cal{L}'_G$ satisfying some technical conditions in Subsection \ref{4.4} , and parameterize irreducible representations by this $\cal{L}'_G$. We emphasize that our result does not depend the choice of $\cal{L}'_G$.

According to Lusztig's results \cite{L1}, let $\pi_{\sp_{2k(k+1)}}$, $\pi_{\rm{SO}_{2k(k+1)+1}}$ and $\pi_{\rm{SO}^\epsilon_{2k^2}}$ be the unique unipotent cuspidal representations of the corresponding groups. It follows easily that there are two irreducible unipotent cuspidal representations $\pi$ and $\pi'$ of $\o^\pm_{2k^2}\fq$ (resp. $\o_{2k(k+1)+1}\fq$), and $\pi=\sgn\pi'$. Moreover, if $\pl$ (resp. $\pi_{\Lambda,\epsilon}$) is a unipotent cuspidal representation, then we have
\[
k=\left\{
\begin{array}{ll}
\frac{|\rm{def}(\Lambda)|-1}{2},&\textrm{ if }\pl\in\cal{E}(\sp_{2k(k+1)});\\
\frac{|\rm{def}(\Lambda)|-1}{2},&\textrm{ if }\pi_{\Lambda,\epsilon}\in\cal{E}(\o_{2k(k+1)+1});\\
\frac{|\rm{def}(\Lambda)|}{2},&\textrm{ if }\pl\in\cal{E}(\o^\e_{2k^2});\\
\end{array}
\right.
\]
(c.f. \cite[section 3]{P3} for details).

Let $\pi_{\rho,\Lambda,\Lambda',\epsilon}$ (resp. $\prll$) be an irreducible representation of $\sp_{2n}\fq$, $\o^\pm_{2n}\fq$ or $\o_{2n+1}\fq$. Assume that $\Lambda$ and $\Lambda'$ correspond to unipotent cuspidal representations of $\ddg$ and $\dddg$, respectively. Let
\[
k=\left\{
\begin{aligned}
&\frac{|\rm{def}(\Lambda)|-1}{2}&\textrm{ if } \Lambda\in\cal{S}_m;\\
&\frac{\rm{def}(\Lambda)}{2}&\textrm{ if } \Lambda\in\cal{S}^\pm_m.
\end{aligned}\right.
\]
and
\[
h=\left\{
\begin{aligned}
&\frac{|\rm{def}(\Lambda')|-1}{2}&\textrm{ if } \Lambda'\in\cal{S}_{m'};\\
&\frac{\rm{def}(\Lambda')}{2}&\textrm{ if } \Lambda'\in\cal{S}^\pm_{m'}.
\end{aligned}\right.
\]
For abbreviation, we write $\pkh$ (resp. $\pi_{\rho,k,h,\epsilon}$) instead of $\prll$ (resp. $\pi_{\rho,\Lambda,\Lambda',\epsilon}$).
We emphasize that $\pkh$ (resp. $\pi_{\rho,k,h,\epsilon}$) is \emph{not} necessarily cuspidal.

\subsection{The main result}
 As is standard, denote by $\so^\epsilon_n$ and $\o^\epsilon_n$, $ \epsilon=\pm$, the (special) orthogonal groups of an $n$ -dimensional quadratic space with discriminant $\epsilon\ 1  \in \Fq^{\times} /(\Fq^\times)^2.$ For convenience, by abuse of notation we also write $\epsilon=\epsilon\ 1$ for the sign of the corresponding discriminant. Denote by $\epsilon_{-1}$ the square class of $-1$.

Fix a character $\psi$ of $\Fq$. For the dual pair  $(\sp_{2n},\o^{\epsilon}_{2n'})$, we write $\omega^\epsilon_{n,n'}$ for the restriction of $\omega_{\rm{Sp}_{2N}}$ to $\sp_{2n}\fq\times\o^{\epsilon}_{2n'}\fq$. Similar notation applies for $(\sp_{2n},\o^{\epsilon}_{2n'+1})$. When the context of dual pairs is clear, abbreviate by $\Theta^\epsilon_{n,n'}$ the theta lifting from $G_n$ to $G'_{n'}$.

By abuse of notation, for $\pi=\pi_{\rho}\in\cal{E}(\sp_{2n})$ and $\pi'=\pi_{\rho'}\in\cal{E}(\sp_{2m})$, we write
\[
m_\psi(\pi,\pi')=\left\{
\begin{array}{ll}
m_\psi(\pi,\pi'),&\textrm{ if }n\ge m;\\
m_\psi(\pi',\pi),&\textrm{ if }n< m.
\end{array}
\right.
\]
For $n=m$, by Proposition \ref{rho}, it is well defined.
If $n=0$ (resp. $m=0$), then we set
\[
m_\psi(-,\pi')\textrm{ (resp. $m_\psi(\pi,-)$)}=\left\{
\begin{array}{ll}
1,&\textrm{ if }\pi' \textrm{ (resp. $\pi$) } \textrm{ is regular (see subsection \ref{3.2})};\\
0,&\textrm{ otherwise}.
\end{array}
\right.
\]
Similarly, for $\pi\in\cal{E}(\o^\epsilon_{n})$ and $\pi'\in\cal{E}(\o^\e_{m})$, we write
\[
m(\pi,\pi')=\left\{
\begin{array}{ll}
m(\pi,\pi'),&\textrm{ if }n> m;\\
m(\pi',\pi),&\textrm{ if }n< m
\end{array}
\right.
\]
and for $n=0$ (resp. $m=0$), we set
\[
m(-,\pi')\textrm{ (resp. $m(\pi,-)$)}=\left\{
\begin{array}{ll}
1,&\textrm{ if }\pi' \textrm{ (resp. $\pi$) } \textrm{ is regular (see subsection \ref{3.2})};\\
0,&\textrm{ otherwise}.
\end{array}
\right.
\]

For a pair of irreducible representations $(\pi,\pi')$, whether the multiplicity (\ref{ggpb}) and (\ref{ggpfj}) vanishes depends on the behavior of the pair in the see-saw. For instance, let $\pi$ and $\sgn\pi$ be two unipotent cuspidal representations of $\o^+_{5}\fq$. By \cite[Theorem 3.12]{LW1}, there exists a representation $\pi_0\in\{\pi,\sgn\pi\}$ such that $\Theta^+_{2,1}(\pi_0)=\pi_1$ where $\pi_1$ is a cuspidal representation of $\sp_2\fq$. Let $\pi'\in\cal{E}(\o^+_4,s)$ where $s$ has no eigenvalues $\pm1$. Then by Theoren \ref{p1}, the first occurrence index of $\pi'$ is $2$, and $\Theta^+_{2,2}(\pi')=\pi'_1$ where $\pi'_1$ is an irreducible representation of $\sp_4\fq$. Consider the see-saw diagram
\[
\setlength{\unitlength}{0.8cm}
\begin{picture}(20,5)
\thicklines
\put(6.8,4){$\sp_{2}\times \sp_{2}$}
\put(7.2,1){$\sp_{2}$}
\put(12.4,4){$\o^+_{5}$}
\put(11.9,1){$\o^+_{4}\times \o^{+}_1$}
\put(7.7,1.5){\line(0,1){2.1}}
\put(12.8,1.5){\line(0,1){2.1}}
\put(8,1.5){\line(2,1){4.2}}
\put(8,3.7){\line(2,-1){4.2}}
\end{picture}
\]
We have
\[
m(\pi_0,\pi')=\langle\pi_0,\pi'\rangle_{\o^+_{4}\fq}\le \langle\Theta^+_{1,2}(\pi_1),\pi'\rangle_{\o^+_{4}\fq}=\langle\pi_1,\Theta^+_{2,1}(\pi')\otimes\omega_1^+\rangle_{\sp_{2}\fq}=0.
\]
So we need to pick some good pairs of representations such that the situation does not happen. We call these pairs of irreducible representations of orthogonal groups (resp. symplectic groups) strongly relevant (resp. $(\psi,\epsilon_{-1})$-strongly relevant). See subsection \ref{6.3} for the explicitly definition of strongly relevant pair.

Our first main result is the following.
\begin{theorem}\label{main1}
Assume that $q$ is large enough such that the main theorem in \cite{S} holds.

 (i) Let $n\ge m$. Let $\pkh\in\cal{E}(\sp_{2n})$, and let $\pkhp$ be an irreducible representation of and $\sp_{2m}\fq$. Then
\[
m_\psi(\pkh,\pkhp)=\left\{
\begin{array}{ll}
m_\psi(\pw_{\rho} ,\pw_{\rho'})  &\textrm{if }(\pkh,\pkhp)\textrm{ is $(\psi,\ee)$-strongly relevant};\\
0&\textrm{otherwise }\
\end{array}\right.
\]
where $m_\psi(\pw_{\rho} ,\pw_{\rho'})$ does not depend on $\psi$.

(ii) Let $\pi_{\rho,h,k,\epsilon''}$ be an irreducible representation of $\o^\epsilon_{2n+1}\fq$, and let $\pkhp$ be an irreducible representation of and $\o^\e_{2m}\fq$. Then
\[
m(\pi_{\rho,h,k,\epsilon''},\pkhp)=\left\{
\begin{array}{ll}
m_\psi(\pw_{\rho} ,\pw_{\rho'})  &\textrm{if }(\pi_{\rho,h,k,\epsilon''},\pkhp)\textrm{ is strongly relevant};\\
0&\textrm{otherwise}
\end{array}\right.
\]
where $\pw_{\rho}$ and $\pw_{\rho'}$ are representations of symplectic groups, and $m_\psi(\pw_{\rho} ,\pw_{\rho'})$ is the same one in (i).

\end{theorem}

\begin{remark}
Theorem \ref{main1} does not depend on the choice of modified Lusztig correspondences. In fact, the parametrization
of irreducible representations is not involved in the definition of strongly relevant pair (resp. $(\psi,\epsilon_{-1})$-strongly relevant pair). And for two different choices of modified Lusztig correspondences, $\pi_{\rho}$ and $\pi_{\rho'}$ will not change.
\end{remark}

Note that $\pi_\rho$ and $\pi_{\rho'}$ are uniform and the theta lifting of them are very simple (see Theorem \ref{p1} and Theorem \ref{p2}). So we can reduce $m_\psi(\pw_{\rho} ,\pw_{\rho'})$ to the Bessel case by the standard arguments of theta correspondence and see-saw dual pairs, and calculate the Bessel case by Reeder's formula.

By Corollary \ref{strongly relevant} and Corollary \ref{strongly relevant1}, we have following result.
\begin{corollary}\label{1.3}
Assume that $q$ is large enough such that the main theorem in \cite{S} holds.

(i)  Keep the assumptions in Theorem \ref{main1} (i). Assume that $m_\psi(\pw_{\rho} ,\pw_{\rho'})\ne 0$. If $k\notin\{|h'|,|h'|-1\}$ or $k'\notin\{|h|,|h|-1\}$, then $m_\psi(\pkh,\pkhp)=0.$ If $k\in\{|h'|,|h'|-1\}$ and $k'\in\{|h|,|h|-1\}$, then there are $\epsilon_0$, $\epsilon_0'\in\{\pm\}$ such that
\[
m_\psi(\pi_{\rho,k,\epsilon^1h},\pi_{\rho',k',\epsilon^2h'})=\left\{
\begin{array}{ll}
m_\psi(\pw_{\rho} ,\pw_{\rho'})  &\textrm{if }(\epsilon^1,\epsilon^2)=(\epsilon_0,\epsilon_0');\\
0&\textrm{otherwise }\
\end{array}\right.
\]

(ii)  Keep the assumptions in Theorem \ref{main1} (ii). Assume that $m_\psi(\pw_{\rho} ,\pw_{\rho'})\ne 0$. If $k\notin\{|k'|,|k'|-1\}$ or $h\notin\{|h'|,|h'|-1\}$, then $m(\pi_{\rho,h,k,\epsilon''},\pkhp)=0.$ If $k\in\{|k'|,|k'|-1\}$ and $h\in\{|h'|,|h'|-1\}$, then there are $\epsilon_0$, $\epsilon_0'\in\{\pm\}$ such that
\[
m(\pi_{\rho,h,k,\epsilon''},\pi_{\rho',\epsilon^1k',\epsilon^2h'})=\left\{
\begin{array}{ll}
m_\psi(\pw_{\rho} ,\pw_{\rho'})  &\textrm{if }(\epsilon^1,\epsilon^2)=(\epsilon_0,\epsilon_0');\\
0&\textrm{otherwise }\
\end{array}\right.
\]

\end{corollary}
In remark \ref{sgn}, we know that two different choices of modified Lusztig correspondences are equal up to sgn. So Corollary \ref{1.3} does not depend on the choice of modified Lusztig correspondences.

 In \cite{P3,P4}, Pan determines the theta correspondence for finite symplectic and orthogonal pairs. A complete understanding of theta correspondence should extend our above results to more general representations.
Let
\[
\begin{aligned}
&\mathcal{G}^{\rm{even},+}_{n,m}:=\left\{(\Lambda,\Lambda')|
\Upsilon(\Lambda')_*\preccurlyeq\Upsilon(\Lambda)_*,\Upsilon(\Lambda)^*\preccurlyeq\Upsilon(\Lambda')^*,\rm{def}(\Lambda)>0,\rm{def}(\Lambda')=\rm{def}(\Lambda)-1\right\}; \\
&\mathcal{G}^{\rm{even},-}_{n,m}:=\left\{(\Lambda,\Lambda')|\Upsilon(\Lambda')_*\preccurlyeq\Upsilon(\Lambda)^*,
\Upsilon(\Lambda)_*\preccurlyeq\Upsilon(\Lambda')^*,\rm{def}(\Lambda)>0,\rm{def}(\Lambda')=-\rm{def}(\Lambda)-1\right\}; \\
&\mathcal{G}^{\rm{odd},-}_{n,m}:=\left\{(\Lambda,\Lambda')|\Upsilon(\Lambda')^*\preccurlyeq \Upsilon(\Lambda)^*,
\Upsilon(\Lambda)_*\preccurlyeq\Upsilon(\Lambda')_*,\rm{def}(\Lambda)<0,\rm{def}(\Lambda')=\rm{def}(\Lambda)+1\right\};\\
&\mathcal{G}^{\rm{odd},+}_{n,m}:=\left\{(\Lambda,\Lambda')|\Upsilon(\Lambda')^*\preccurlyeq \Upsilon(\Lambda)_*,
\Upsilon(\Lambda)^*\preccurlyeq\Upsilon(\Lambda')_*,\rm{def}(\Lambda)<0,\rm{def}(\Lambda')=-\rm{def}(\Lambda)+1\right\}
\end{aligned}
\]
be subsets of $\cal{S}_n\times\cal{S}_m^{\pm}$ where $\Upsilon(\Lambda)^*$ and $\Upsilon(\Lambda)_*$ are defined in subsection \ref{sec4.2}. Let
\[
\mathcal{G}=\bigcup_{n,m}\left(\mathcal{G}^{\rm{even},+}_{n,m}\bigcup\mathcal{G}^{\rm{even},-}_{n,m}
\bigcup\mathcal{G}^{\rm{odd},-}_{n,m}\bigcup\mathcal{G}^{\rm{odd},+}_{n,m}\right).
\]

Our second main result is the following.
\begin{theorem}\label{main2}
Assume that $q$ is large enough such that the main theorem in \cite{S} holds.

(i) Let $n\ge m$. Let $\prll\in\cal{E}(\sp_{2n})$ and $\pi_{\rho_1,\Lambda_1,\Lambda_1'}\in\cal{E}(\sp_{2m})$. Then we have
\[
\begin{aligned}
m_\psi(\prll,\pi_{\rho_1,\Lambda_1,\Lambda_1'})
=\left\{
\begin{array}{ll}
m_\psi(\pw_{\rho} ,\pw_{\rho_1}),  &\textrm{if }(\prll,\pi_{\rho_1,\Lambda_1,\Lambda_1'})\textrm{ is $(\psi,\ee)$-strongly relevant, and there}\\
&\textrm {are }
\widetilde{\Lambda_1'}\in\{\Lambda_1',\Lambda_1^{\prime t}\}
\textrm{ and }\widetilde{\Lambda'}\in\{\Lambda',\Lambda^{\prime t}\}\textrm{ such that }(\Lambda,\widetilde{\Lambda_1'})\\
&\textrm{and }(\Lambda_1,\wla')\in \mathcal{G};\\
0,&\textrm{otherwise}.
\end{array}\right.
\end{aligned}
\]
Moreover, if $m=n$, then we have
\[
\prll\otimes\omega_{n,\psi}^{\epsilon_{-1}}=\bigoplus m_\psi(\pw_{\rho} ,\pw_{\rho_1})\pi_{\rho_1,\Lambda_1,\Lambda_1'}
\]
where the sum runs over the irreducible representations as above.

(ii) Let $\pi_{\rho,\Lambda,\Lambda',\epsilon''}\in\cal{E}(\rm{O}^\epsilon_{2n+1})$ and $\pi_{\rho_1,\Lambda_1,\Lambda_1'}\in\cal{E}(\rm{O}^{\epsilon'}_{2m})$.
Then we have
\[
\begin{aligned}
m(\pi_{\rho,\Lambda,\Lambda',\epsilon''},\pi_{\rho_1,\Lambda_1,\Lambda_1'})=\left\{
\begin{array}{ll}
m_\psi(\pw_{\rho} ,\pw_{\rho_1}),  &\textrm{if }(\pi_{\rho,\Lambda,\Lambda',\epsilon''},\pi_{\rho_1,\Lambda_1,\Lambda_1'})\textrm{ is strongly relevant, and there }\\
&\textrm{are } \widetilde{\Lambda_1}\in\{\Lambda_1,\Lambda^{ t}_1\} \textrm{ and } \widetilde{\Lambda_1'}\in\{\Lambda_1',\Lambda_1^{\prime t}\}\textrm{ such that }(\Lambda,\widetilde{\Lambda_1})\\
&\textrm{and }(\Lambda',\widetilde{\Lambda_1'})\in \mathcal{G};\\
0,&\textrm{otherwise}
\end{array}\right.
\end{aligned}
\]
where $m_\psi(\pw_{\rho} ,\pw_{\rho_1})$ is the same one as above.
Moreover, if $m=n$, then we have
\[
\pi_{\rho,\Lambda,\Lambda',\epsilon''}|_{\rm{O}^{\epsilon'}_{2m}}=\bigoplus m_\psi(\pw_{\rho} ,\pw_{\rho_1}) \pi_{\rho_1,\Lambda_1,\Lambda_1'}
\]
where the sum runs over the irreducible representations as above.
\end{theorem}

\begin{remark}
Theorem \ref{main2} does not depend on the choice of modified Lusztig correspondences. In fact, for (i), we already know that the condition of $(\psi,\ee)$-strongly relevant does not depend on modified Lusztig correspondences. Let $\sigma_{\rho,\Lambda,\Lambda'}$ be the irreducible representation of $\sp_{2n}\fq$ by a different choice of modified Lusztig correspondences, and assume that $\sigma_{\rho,\Lambda,\Lambda'}\ne \pi_{\rho,\Lambda,\Lambda'}$. In remark \ref{sgn}, we know that $\sigma_{\rho,\Lambda,\Lambda'}= \pi_{\rho,\Lambda,\Lambda''}$ with $\Lambda''\in\{\Lambda',\Lambda^{\prime t}\}$, which implies that the conditions does not depend on modified Lusztig correspondences. For orthogonal groups, we have similar arguments.
\end{remark}

\begin{corollary}\label{1.6}
Assume that $q$ is large enough such that the main theorem in \cite{S} holds.

(i)  Keep the assumptions in Theorem \ref{main2} (i). Let $\{\pi^i\}=\{\pi_{\rho,\Lambda,\Lambda'},\pi_{\rho,\Lambda,\Lambda^{\prime t}}\}$, and let $\{\pi^i_1\}=\{\pi_{\rho_1,\Lambda_1,\Lambda_1'},\pi_{\rho_1,\Lambda_1,\Lambda^{\prime t}_1}\}$ with $i=1,2$.
Assume that $m_\psi(\pw_{\rho} ,\pw_{\rho'})\ne 0$.  If there are
$\widetilde{\Lambda_1'}\in\{\Lambda_1',\Lambda_1^{\prime t}\}$ and $\widetilde{\Lambda'}\in\{\Lambda',\Lambda^{\prime t}\}$ such that $(\Lambda,\widetilde{\Lambda_1'})$ and $(\Lambda_1,\wla')\in \mathcal{G}$, then there are $i_0$, $j_0$ such that
\[
m_\psi(\pi^i,\pi_1^j)=\left\{
\begin{array}{ll}
m_\psi(\pw_{\rho} ,\pw_{\rho'})  &\textrm{if }(i,j)=(i_0,j_0);\\
0&\textrm{otherwise. }\
\end{array}\right.
\]
If not, then $m_\psi(\pi^i,\pi^j_1)=0$ for every $i$ and $j$.

(ii)  Keep the assumptions in Theorem \ref{main2} (ii).  Let $\{\pi^i\}=\{\pi_{\rho_1,\Lambda_1,\Lambda'_1},\pi_{\rho_1,\Lambda_1^t,\Lambda'_1},\pi_{\rho_1,\Lambda_1,\Lambda^{\prime t}_1},\pi_{\rho_1,\Lambda_1^t,\Lambda^{\prime t}_1}\}$.
Assume that $m_\psi(\pw_{\rho} ,\pw_{\rho'})\ne 0$.  If there are
$\widetilde{\Lambda_1}\in\{\Lambda_1,\Lambda^{ t}_1\}$ and $\widetilde{\Lambda_1'}\in\{\Lambda_1',\Lambda_1^{\prime t}\}$ such that $(\Lambda,\widetilde{\Lambda_1})$ and $(\Lambda',\widetilde{\Lambda_1'})\in \mathcal{G}$, then there is $i_0$ such that
\[
m_\psi(\pi_{\rho,\Lambda,\Lambda',\epsilon''},\pi_1^i)=\left\{
\begin{array}{ll}
m_\psi(\pw_{\rho} ,\pw_{\rho'})  &\textrm{if }i=i_0;\\
0&\textrm{otherwise. }\
\end{array}\right.
\]
If not, then $m_\psi(\pi_{\rho,\Lambda,\Lambda',\epsilon''},\pi_1^i)=0$ for every $i$.

\end{corollary}

We obtain the following immediate consequence by Theorem \ref{main2}, Proposition \ref{regular} and Theorem 1.1 in \cite{LW3}.
\begin{corollary}[multiplicity one for unipotent representations]\label{main3}
Assume that $q$ is large enough such that the main theorem in \cite{S} holds. Let $n\ge m$.

(i) Let $\pl\in\cal{E}(\sp_{2n},1)$. For an irreducible representation $\prllc$ of $\sp_{2m}\fq$, we have
\[
m_\psi(\pl,\prllc)=\left\{
\begin{array}{ll}
1, &\textrm{if }(\pl,\prllc)\textrm{ is $(\psi,\ee)$-strongly relevant, and there is } \widetilde{\Lambda_1'}\in\{\Lambda_1',\Lambda_1^{\prime t}\}\\
&\textrm{such that }(\Lambda,\widetilde{\Lambda_1'})\in \mathcal{G} ,\textrm{ and } \pi_{\Lambda_1} \textrm{ and $\rho$ are regular};\\
0, & \textrm{otherwise.}
\end{array}\right.
\]

(ii) Let $\pi_{\Lambda}\in\cal{E}(\sp_{2n},\sigma)$, and let $\pi_{-,-,\Lambda'}\in\cal{E}(\sp_{2m},\sigma')$ where $\sigma$ is an irreducible unipotent cuspidal representation of $\sp_{2k(k+1)}\fq$ and $\sigma'$ is an irreducible $\theta$-cuspidal representation of $\sp_{2k^{\prime 2}}\fq$. Then
\[
m_\psi(\pl,\pi_{-,-,\Lambda'})=\left\{
\begin{array}{ll}
1, &\textrm{if either $\CD^\rm{FJ}_{k, \psi}(\sigma)=\sigma'$ or $\CD^\rm{FJ}_{k', \psi}(\sigma')=\sigma$ and there is } \widetilde{\Lambda'}\in\{\Lambda',\Lambda^{\prime t}\}\\
&\textrm{such that }(\Lambda,\widetilde{\Lambda'})\in \mathcal{G};\\
0, & \textrm{otherwise.}
\end{array}\right.
\]
where $\CD^\rm{FJ}_{\ell, \psi}$ is defined as in \cite[(1.9)]{LW3} and $\theta$-representations are defined in section \ref{sec4}.

(iii) Let $\pi_{\Lambda,\epsilon''}$ be an irreducible unipotent representation of $\o^\epsilon_{2n+1}\fq$. For an irreducible representation $\prllc$ of $\o^\e_{2m}\fq$, we have
\[
m(\pi_{\Lambda,\epsilon''},\prllc)=\left\{
\begin{array}{ll}
1, &\textrm{ if }(\pi_{\Lambda,\epsilon''},\prllc)\textrm{ is strongly relevant, and there is } \widetilde{\Lambda_1}\in\{\Lambda_1,\Lambda_1^{t}\}\textrm{ such that }\\
&(\Lambda,\widetilde{\Lambda_1})\in \mathcal{G},\textrm{ and } \pi_{\Lambda_1'}\textrm{ and $\rho$ are regular;}\\
0, & \textrm{otherwise.}
\end{array}\right.
\]

(iv) Let $\pi_{\Lambda,\epsilon''}\in\cal{E}(\o^\epsilon_{2n+1},\sigma)$, and let $\pi_{-,\Lambda',-}\in\cal{E}(\o^{\epsilon'}_{2m},\sigma')$ where $\sigma$ is an irreducible unipotent cuspidal representation of $\o^\epsilon_{2k(k+1)}\fq$ and $\sigma'$ is an irreducible unipotent cuspidal representation of $\o^{\epsilon'}_{2k^{\prime 2}}\fq$. Then
\[
m_\psi(\pi_{\Lambda,\epsilon''},\pi_{-,\Lambda',-})=\left\{
\begin{array}{ll}
1, &\textrm{if either $\CD^\rm{B}_{k, v_0}(\sigma)=\sigma'$ or $\CD^\rm{B}_{k'-1, v_0}(\sigma')=\sigma$ and there is } \widetilde{\Lambda'}\in\{\Lambda',\Lambda^{\prime t}\}\\
&\textrm{such that }(\Lambda,\widetilde{\Lambda'})\in \mathcal{G};\\
0, & \textrm{otherwise.}
\end{array}\right.
\]
where $\CD^\rm{B}_{\ell, v_0}$ is defined as in \cite[(1.5)]{LW3}.
\end{corollary}
\begin{corollary}\label{main4}
Assume that $q$ is large enough such that the main theorem in \cite{S} holds.

(i) Let $n\ge m$. Let $\pl$ be a unipotent representation of $\sp_{2n}\fq$. Then we have
\[
\pl\otimes\omega_{n,\psi}^+=\bigoplus\prllc
\]
where $\prllc$ runs over $\cal{E}(\sp_{2n})$ such that $(\pl,\prllc)$  is $(\psi,\ee)$-strongly relevant, and there is $\widetilde{\Lambda_1'}\in\{\Lambda_1',\Lambda_1^{\prime t}\}$ such that $(\Lambda,\widetilde{\Lambda_1'})\in \mathcal{G}$, and $\pi_{\Lambda_1}$ and $\rho$ are regular.

(ii) Let $\pi_{\Lambda,\epsilon''}$ be a unipotent representation of $\o^\epsilon_{2n+1}\fq$, Then we have
\[
\pi_{\Lambda,\epsilon''}|_{\o^\e_{2n}\fq}=\bigoplus\prllc
\]
where $\prllc$ runs over $\cal{E}(\o^\e_{2n})$ such that $(\pl,\prllc)$  is strongly relevant, and there is $\widetilde{\Lambda_1}\in\{\Lambda_1,\Lambda_1^{t}\}$ such that $(\Lambda,\widetilde{\Lambda_1})\in \mathcal{G}$, and $\pi_{\Lambda_1'}$ and $\rho$ and are regular.
\end{corollary}

The conditions of our main result look very complicate. The following will explain the meaning of these conditions.

\begin{corollary}\label{main5}
Assume that $q$ is large enough such that the main theorem in \cite{S} holds. The multiplicity in Gan-Gross-Prasad problem for unipotent representations can be visualized as a diagram:
\[
\setlength{\unitlength}{0.8cm}
\begin{picture}(20,5)
\thicklines
\put(3,4){$\pi\in\cal{E}(\sp_{2n},1)$}
\put(8.4,4.2){Theta lifting}
\put(2.5,1){$\pi^\star\in\cal{E}(\so_{2n+1},1)$}
\put(13.1,4){$\bigoplus\pi',\  \pi'\in \cal{E}(\o^\epsilon_{2m},1)$}
\put(13,1){$\bigoplus\pi^{\prime \star},\  \pi^{\prime \star}\in \cal{E}(\so^\epsilon_{2m},1)$}
\put(7.3,1.3){unipotent part of GGP}
\put(2.5,2.5){$\cal{L}_1\circ\cal{D}$}
\put(14,2.5){$\cal{R}\circ\cal{D}$}
\put(4.3,3.6){\vector(0,-1){2.1}}
\put(13.6,3.6){\vector(0,-1){2.1}}
\put(6.5,1.1){\vector(2,0){6.2}}
\put(6.5,4){\vector(2,0){6.2}}
\end{picture}
\]
where $\cal{L}_1$ is the Lusztig correspondence and $\cal{R}$ is the restriction form $\o^\epsilon_{2m}$ to $\so^\epsilon_{2m}$, and $\cal{D}$ is the Alvis-Curtis duality which sends $\pi_{\Lambda}$ to $\pi_{\Lambda'}$ with $\rm{def}(\Lambda)=\rm{def}(\Lambda')$,  $\Upsilon(\Lambda')^*={}^t(\Upsilon(\Lambda)_*)$ and $\Upsilon(\Lambda')_*={}^t(\Upsilon(\Lambda)^*)$.
\end{corollary}

 This paper is organized as follows. In Section \ref{sec2}, we recall the notation of Harish-Chandra series. In Section \ref{sec3}, we recall the theory of Deligne-Lusztig characters and Lusztig correspondence. In particular, we focus on the modified Lusztig
corespondence for finite symplectic groups and finite orthogonal groups. Then we provide some results for regular characters which are used in this paper. In Section \ref{sec4}, we recall the classification of quadratic unipotent representations of symplectic groups and orthogonal groups. In Section \ref{sec6}, we recall the result by Pan in \cite{P4} on the Howe correspondence for finite symplectic groups and finite orthogonal groups. Then we discuss the relations between the symbols of representations in the Howe correspondence which play the important roles in the proof of our main results. In Section \ref{sec7}, we recall the and first occurrence index of cuspidal representations of finite orthogonal groups and symplectic groups, and give the definitions of relevant and strongly relevant. In Section \ref{sec8}, we prove the Theorem \ref{main1}. In Section \ref{sec9}, we prove the the Theorem \ref{main2}.

{\bf Acknowledgement.} The author would like to thank Dongwen Liu for many helpful and enlightening discussions on this topic.
\section{Harish-Chandra series}\label{sec2}
Let $G$ be a reductive group defined over $\Fq$, $F$ be the corresponding Frobenius endomorphism, and let $\E(G)=\rm{Irr}(G^F)$ be the set of irreducible representations of $G^F$.
A parabolic subgroup $P$ of $G$ is the normalizer in $G$ of a parabolic subgroup $P^\circ$ of the connected component $G^\circ$ of $G$.
A Levi subgroup $L$ of $P$ is the normalizer in $G$ of the a Levi subgroup $L^\circ$ of $P^\circ$. Then we have a Levi decomposition $P=LV$. If $P$ is $F$-stable, then we have $P^F = L^FV^F.$ Let $\delta$ be a representation of the group $L^F$. We can lift $\delta$ to a character of $P^F$ by making it trivial on $V^F$. We have the parabolic induction
\begin{equation}\label{paraind}
I_L^G(\delta):=I_P^G(\delta)=\mathrm{Ind}_{P^F}^{G^F}\delta.
\end{equation}
It is well-known that the induction in stages holds (see e.g. \cite[Proposition 4.7]{DM}), namely if $Q \subset P$ are two parabolic subgroups of $G$ and $M \subset L$
are the corresponding Levi subgroups, then
\[
I^G_L\circ I^L_M=I^G_M.
\]
We say that a pair $(L, \delta)$ is cuspidal if $\delta$ is cuspidal.

\begin{theorem} For $\pi \in \E(G)$, there is  a unique cuspidal pair $(L, \delta)$ up to $G^F$-conjugacy such that $\langle\pi,I_{L}^G(\delta)\rangle_{G^F}\ne 0$
\end{theorem}

Thus we get a partition of $\E(G_n)$ into series parametrized by $G_n^F$-conjugacy classes of cuspidal pairs $(L, \delta)$. The Harish-Chandra series of $(L, \delta)$ is the set of irreducible representation of $G_n^F$ appearing in $I_{L}^G(\delta)$.
We focus on classical groups, and let $L$ be an $F$-stable standard Levi subgroup of $G_n:=\sp_{2n}$, $\o^\pm_{2n}$ or $\o_{2n+1}$. Then $L^F$ has a standard form
\[
L^F=\GGL_{n_1}\fq\times \GGL_{n_2}\fq\times\cdots\times \GGL_{n_r}\fq\times G_m^F
\]
where $G_m=\sp_{2m}$, $\o^\pm_{2m}$ or $\o_{2m+1}$, and $n_1+\cdots+n_r +m=n$. For a cuspidal pair $(L,\delta)$, one has
\[
\delta=\rho_1\otimes\cdots\otimes\rho_r\otimes\sigma
\]
where $\rho_i$ and $\sigma$ are cuspidal representations of $\GGL_{n_i}\fq$ and $G_m^F$, respectively.

By induction in stages, for any irreducible component $\pi$ of $I_L^G(\delta)$, there exists $\rho\in \E(\GGL_{n-m})$ such that
$\pi \subset I_{\GGL_{n-m}\times G_m}^{G_n}(\rho\otimes\sigma).
$
Let
\[
\cal{E}(G_n,\sigma)=\{\pi\in \cal{E}(G_n)|\langle\pi,I_{\GGL_{n-m}\times G_m}^G(\rho\otimes\sigma)\rangle_{G^F}\ne 0\textrm{ for some }\rho\in \E(\GGL_{n-m})\}.
\]
Then we have  a disjoint union
\[
\E(G_n)=\bigcup_{\sigma}\cal{E}(G_n,\sigma),
\]
where $\sigma$ runs over all irreducible cuspidal representations of $G^{F}_m$, $m=0,1,\cdots,n$.

\section{Deligne-Lusztig characters and Lusztig correspondence} \label{sec3}
Let $G$ be a connected reductive algebraic
group over $\mathbb{F}_q$. In \cite{DL}, P. Deligne and G. Lusztig defined a virtual character $R^{G}_{T,\theta}$ of $G^F$, associated to an $F$-stable maximal torus $T$ of $G$ and a character $\theta$ of $T^F$. We review some standard facts on these characters and Lusztig correspondence (cf. \cite[Chapter 7, 12]{C}), which will be used in this paper. In the last part of this section, we compute the the multiplicity (\ref{1o}) for regular characters by Reeder's formula in \cite{R}.

\subsection{Centralizer of a semisimple element }

 Let $H$ be a symplectic group or orthogonal group. Let $s$ be a semisimple element in the connected component of $H$. Let $C_{H(\overline{\bb{F}}_q)}(s)$ be the centralizer in $H(\overline{\bb{F}}_q)$ of a semisimple element $s \in H^0(\overline{\bb{F}}_q)$. In \cite[subsection 1.B]{AMR}, A.-M. Aubert, J. Michel and R. Rouquier described $C_{H(\overline{\bb{F}}_q)}(s)$ as follows. Let $T(\overline{\bb{F}}_q) \cong \overline{\bb{F}}^\times_q\times\cdots\times\overline{\bb{F}}_q^\times$ be a rational maximal torus of $H(\overline{\bb{F}}_q)$, and let $s=(\lambda_1,\cdots,\lambda_l)\in T^F$. Let $\nu_{\lambda}(s):=\#\{i|\lambda_i=\lambda\}$, and let $\langle\lambda\rangle$ denote the set of all roots in $\overline{\bb{F}}_q$ of the irreducible polynomial of $\lambda$ over $\overline{\bb{F}}_q$. The group $C_{H(\overline{\bb{F}}_q)}(s)$ has a natural decomposition with the eigenvalues of $s$:
\[
C_{H(\overline{\bb{F}}_q)}(s)=\prod_{\langle\lambda\rangle\subset\{\lambda_1,\cdots,\lambda_l\}}H_{[\lambda]}(s)(\overline{\bb{F}}_q)
\]
where $H_{[\lambda]}(s)(\overline{\bb{F}}_q)$ is a reductive quasi-simple group of rank equal to $|\langle\lambda\rangle|\nu_{\lambda}(s)$.

\subsection{Modified Lusztig correspondence for symplectic groups and orthogonal groups}\label{mlus}
Let $G^*$ be the dual group of $G$. We still denote the Frobenius endomorphism of $G^*$ by $F$. Then there is a natural bijection between the set of $G^F$-conjugacy classes of $(T, \theta)$ and the set of $G^{*F}$-conjugacy classes of $(T^*, s)$ where $T^*$ is a $F$-stable maximal torus in $G^*$ and $s \in   T^{*F}$. We will also denote $R_{T,\theta}^G$  by $R_{T^*,s}^G$ if $(T, \theta)$ corresponds to $(T^*, s)$.
For a semisimple element $s \in G^{*F}$, define
\[
\mathcal{E}(G,s) = \{ \chi \in \mathcal{E}(G)  :  \langle \chi, R_{T^*,s}^G\rangle \ne 0\textrm{ for some }T^*\textrm{ containing }s \}.
\]
The set $\mathcal{E}(G,s)$ is called  the Lusztig series. We can thus define a partition of $\mathcal{E}(G)$ by Lusztig series
i.e.,
\[
\mathcal{E}(G)=\coprod_{(s)}\mathcal{E}(G,s).
\]

\begin{proposition}[Lusztig]\label{Lus}
There is a bijection
\[
\mathcal{L}_s:\mathcal{E}(G,s)\to \mathcal{E}(C_{G^{*}}(s),1),
\]
extended by linearity to a map between virtual characters satisfying that
\[
\mathcal{L}_s(\varepsilon_G R^G_{T^*,s})=\varepsilon_{C_{G^{*}}(s)} R^{C_{G^{*F}}(s)}_{T^*,1}.
\]
Moreover, we have
\[
\rm{dim}(\pi)=\frac{|G^F|_{p'}}{|C_{G^{*F}}(s)|_{p'}}\rm{dim}(\cal{L}_s(\pi))
\]
where $|G|_{p'}$ denotes greatest factor of $|G|$ not divided by $p$, and  $\varepsilon_G:= (-1)^r$ where $r$ is the $\Fq$-rank of $G$.
In particular, Lusztig correspondence send cuspidal representation to cuspidal representation.
\end{proposition}

Note that the correspondence $\cal{L}_s$ is usually not uniquely determined. We now give the explicit results of (modified) Lusztig correspondence for symplectic groups and orthogonal groups (c.f.  \cite[section 6, 7]{P4} for details). Our notation is slightly different from that of \cite{P4}: the group $\ddg$ always associates with eigenvalue $1$.

 (1) Suppose that $G$ is a symplectic group. Then $G^*$ is a special odd orthogonal group. We define
 \begin{itemize}

\item $\dg=\prod_{\langle\lambda\rangle\subset\{\lambda_1,\cdots,\lambda_l\},\lambda\ne\pm1}G^*_{[\lambda]}(s)^F$;

\item  $\ddg=(G^*_{[1]}(s))^{* F}$, the dual group of $G^*_{[1]}(s)^F$;

\item  $\dddg=G^*_{[-1]}(s)^F$.
\end{itemize}
Then we have
\[
C_{G^{*F}}(s)\cong\dg\times\ddgg\times\dddg,
\]
and the \emph {modified} Lusztig correspondence:
 \[
 \cal{L}_s': \cal{E}(G,s)\to\cal{E}(\dg\times\ddg\times\dddg,1)
 \]
 where $\dg$ is a product of finite general linear groups and finite unitary groups, $\ddg$ is a finite symplectic group of rank equal to $\nu_{1}(s)$ and $\dddg$ is a finite even orthogonal group of rank equal to $\nu_{-1}(s)$.  So we can write $\cal{L}'_s(\pi)=\p\otimes\pp\otimes\ppp$. Let $\{\pi_i\}$ denote the image of $\cal{L}_s^{\prime -1}$ of the set
\begin{equation}\label{imega1}
\{\p\otimes\pp\otimes\ppp,\p\otimes\pp\otimes(\rm{sgn}\cdot\ppp)\}.
\end{equation}

(2) Assume that $G$ is an odd orthogonal group. Then $G^*$ is the product of a symplectic group and $\{\pm1\}$. We define
 \begin{itemize}

\item $\dg=\prod_{\langle\lambda\rangle\subset\{\lambda_1,\cdots,\lambda_l\},\lambda\ne\pm1}G^*_{[\lambda]}(s)^F$;

\item  $\ddg=G^*_{[1]}(s)^F$;

\item  $\dddg=G^*_{[-1]}(s)^F$.
\end{itemize}
Now
\[
C_{G^{*F}}(s)\cong\dg\times\ddg\times\dddg\times\{\pm1\},
\]
and the Lusztig correspondence:
 \[
 \cal{L}_s: \cal{E}(G,s)\to
\cal{E}(\dg\times\ddg\times\dddg,1)\times\{\pm\}
 \]
where $\ddg=\sp_{2\nu_{1}(s)}\fq$ and $\dddg=\sp_{2\nu_{-1}(s)}\fq$. Here, by abuse of notation, we denote characters of $\{\pm1\}$ by $\{\pm\}$ instead of $\{1,\rm{sgn}\}$.

(3) Assume that $G$ is an even orthogonal group. Suppose that $G^F\cong \o^{\epsilon_0}_{2n}\fq$. We define
 \begin{itemize}

\item $\dg=\prod_{\langle\lambda\rangle\subset\{\lambda_1,\cdots,\lambda_l\},\lambda\ne\pm1}G^*_{[\lambda]}(s)^F$;

\item
$\ddg=G^*_{[1]}(s)^F;$

\item  $\dddg=G^*_{[-1]}(s)^F$.
\end{itemize}
 Now
\[
C_{G^{*F}}(s)\cong\dg\times\ddg\times\dddg,
\]
where $\ddg\cong\o^\epsilon_{2\nu_{1}(s)}\fq$ and $\dddg\cong\o^\e_{2\nu_{-1}(s)}\fq$ such that $\epsilon\cdot\epsilon'=\ee\cdot\epsilon_0$. Let $\{\pi_i\}$ denote the image of $\cal{L}_s^{-1}$ of the set
\begin{equation}\label{imega2}
\{\p\otimes\pp\otimes\ppp,\p\otimes\pp\otimes(\rm{sgn}\otimes\ppp),\p\otimes(\rm{sgn}\otimes\pp)\otimes\ppp,\p\otimes(\rm{sgn}\otimes\pp)\otimes(\rm{sgn}\otimes\ppp)\}.
\end{equation}

By abuse of notation, we write $ \cal{L}_s'= \cal{L}_s$ if $G$ is an orthogonal group, and call it modified Lusztig correspondence.

\subsection{Regular characters}\label{3.2}
Let $T$ be an $F$-stable maximal torus of $G$ and $W_G(T)$ be the weyl group.
An $F$-stable maximal torus $T$ is said to be minisotropic if $T$ is not contained in any $F$-stable proper parabolic subgroup of $G$. Then a representation $\pi$ of $G^F$ is cuspidal if and only if
\[
\langle \pi,R^G_{T,\theta}\rangle_{G^F}=0
\]
whenever $T$ is not minisotropic, for any character $\theta$ of $T^F$ (see \cite[Theorem 6.25]{S1}). Note that if $G^F=\GGL_n(\Fq)$, then $T$ is said to be minisotropic when $T^F\cong \GGL_1(\mathbb{F}_{q^n})$.

Assume that $\theta\in \cal{E}(T)$, $\theta'\in \cal{E}(T')$ where $T$, $T'$ are $F$-stable maximal tori. The pairs $(T,\theta)$, $(T',\theta')$ are said to be geometrically conjugate if for some $n\ge 1$, there exists $x \in G^{F^n}$ such that
\[
^xT^{F^n} = T ^{\prime F^n}\ \mathrm{and}\ \ ^x(\theta \circ  N^T_n) = \theta'\circ N^{T'}_n
 \]
 where $N_n^T: T^{F^n}\to T^F$ is the norm map. By \cite[p. 378]{C}, for any geometrically conjugate class $\kappa$, there is a unique regular character $\pi^{reg}_\kappa$ appearing in $R^G_{T,\theta}$ for some $(T,\theta)\in \kappa$; and any regular character appears in exactly one geometric conjugacy class. Moreover
\begin{equation}\label{reg}
\pi^{reg}_\kappa=\sum_{(T,\theta)\in\kappa \ \rm{mod} \ G^F}\frac{\varepsilon_G \varepsilon_T R^G_{T,\theta}}{\langle R^G_{T,\theta},R^G_{T,\theta}\rangle _{G^F}}.
\end{equation}
The above equation implies that $\pi_\kappa^{reg}$ appears in $R^G_{T,\theta}$ for every pair $(T,\theta)\in \kappa$. Thus $\pi_\kappa^{reg}$ is cuspidal if and only if $T$ is minisotropic and $\theta$ is regular for every pair $(T,\theta)\in \kappa$.  Here $\theta$ regular means that
\[
^x\theta=\theta, \ x\in W_G(T)^F \ \mathrm{if \ and \ only\ if\ } x=1.
\]
In particular, if $\tau$ is an irreducible cuspidal representation of $\GGL_n(\Fq)$, then there is a pair $(T,\theta)$ with $T$ an $F$-stable minisotropic maximal torus and $\theta$ regular such that
$
\tau=\pm R_{T,\theta}^G.
$

\begin{proposition}\label{regular}
Let $s$ and $s'$ be two semisimple elements of $\so^\epsilon_{n}\fq^*$ and $\so^\e_{n-1}\fq^*$, respectively. Assume that $s$ and $s'$ have no common eigenvalues and $\pm1$ are not eigenvalues of $s$ and $s'$. Let $\tau_1\in\cal{E}(\so^\epsilon_n,s)$ and, $\tau_2\in\cal{E}(\so^\e_{n-1},s')$. Then
\[
\langle\tau_1,\tau_2\rangle_{\so^\e_{n-1}\fq}=\left\{
\begin{array}{ll}
1,&\textrm{ if both }\tau_1\textrm{ and }\tau_2\textrm{ are regular;}\\
0,&\textrm{ otherwise.}
\end{array}
\right.
\]
\end{proposition}
\begin{proof}
By \cite[(9.1)]{R}, for any $F$-stable maximal torus $s\in T\subset\so^\epsilon_n$ and $s'\in S\subset \so^\e_{n-1}$, we have
\begin{equation}\label{reg1}
\langle R^{\so^\epsilon_n}_{T,s},R^{\so^\e_{n-1}}_{S,s'}\rangle_{\so^\e_{n-1}\fq}=\varepsilon_{\so^\epsilon_n} \varepsilon_T \varepsilon_{\so^\e_{n-1}} \varepsilon_S.
\end{equation}

Since $\pm1$ are not eigenvalues of $s$ and $s'$, both $C_{(\so^\epsilon_n)^{*F}}(s)$ and $C_{(\so^\e_{n-1})^{*F}}(s')$ are a product of general linear groups and unitary groups, which implies that $\tau_1$ and $\tau_2$ are uniform, i.e. $\tau_1$ and $\tau_2$ are
linear combination of the Deligne-Lusztig characters.
Suppose that
\[
\tau_1=\sum_{(T,s)\in\kappa \ \rm{mod} \ \so^\epsilon_{n}\fq}C_{T} R^{\so^\epsilon_n}_{T,s}
\]
and
\[
\tau_2=\sum_{(S,s')\in\kappa' \ \rm{mod} \ \so^\e_{n-1}\fq}C_{S} R^{\so^\e_{n-1}}_{S,s'}
\]
where $\kappa$ and $\kappa'$ are geometrically conjugate classes, and $C_{T}$ and $C_{S}\in\bb{Z}$. Then by (\ref{reg1}), we have
\[
\begin{aligned}
\langle \tau_1,R^{\so^\e_{n-1}}_{S,s'}\rangle_{\so^\e_{n-1}\fq}
=&\sum_{(T,s)\in\kappa \ \rm{mod} \ \so^\epsilon_{n}\fq}C_{T} \langle R^{\so^\epsilon_n}_{T,s},R^{\so^\e_{n-1}}_{S,s'}\rangle_{\so^\e_{n-1}\fq}\\
=&\sum_{(T,s)\in\kappa \ \rm{mod} \ \so^\epsilon_{n}\fq}\varepsilon_{\so^\epsilon_n} \varepsilon_T \varepsilon_{\so^\e_{n-1}} \varepsilon_S C_{T}\\
=&\varepsilon_{\so^\e_{n-1}} \varepsilon_S\sum_{(T,s)\in\kappa \ \rm{mod} \ \so^\epsilon_{n}\fq}\varepsilon_{\so^\epsilon_n} \varepsilon_T C_{T}\\
=&\varepsilon_{\so^\e_{n-1}} \varepsilon_S\sum_{(T,s)\in\kappa \ \rm{mod} \ \so^\epsilon_{n}\fq}\frac{\varepsilon_{\so^\epsilon_n} \varepsilon_T C_T \langle R^{\so^\epsilon_n}_{T,s},R^{\so^\epsilon_n}_{T,s}\rangle _{\so^\epsilon_n\fq}}{\langle R^{\so^\epsilon_n}_{T,s},R^{\so^\epsilon_n}_{T,s}\rangle _{\so^\epsilon_n\fq}}\\
\end{aligned}
\]
By \cite[Theorem 7.3.4]{C}, for two pairs $(T,s)$ and $(T',s)\in \kappa$, if $(T,s)\ne (T',s)\ \rm{mod} \ \so^\epsilon_{n}\fq$, then we have
\[
\langle R^{\so^\epsilon_n}_{T,s},R^{\so^\epsilon_n}_{T',s}\rangle _{\so^\epsilon_n\fq}=0.
\]
So by (\ref{reg}),
\[
\begin{aligned}
&\varepsilon_{\so^\e_{n-1}} \varepsilon_S\sum_{(T,s)\in\kappa \ \rm{mod} \ \so^\epsilon_{n}\fq}\frac{\varepsilon_{\so^\epsilon_n} \varepsilon_T C_T \langle R^{\so^\epsilon_n}_{T,s},R^{\so^\epsilon_n}_{T,s}\rangle _{\so^\epsilon_n\fq}}{\langle R^{\so^\epsilon_n}_{T,s},R^{\so^\epsilon_n}_{T,s}\rangle _{\so^\epsilon_n\fq}}\\
=&\varepsilon_{\so^\e_{n-1}} \varepsilon_S
\left\langle \sum_{(T,s)\in\kappa \ \rm{mod} \ \so^\epsilon_{n}\fq}
\frac{\varepsilon_{\so^\epsilon_n} \varepsilon_TR^{\so^\epsilon_n}_{T,s}}{\langle R^{\so^\epsilon_n}_{T,s},R^{\so^\epsilon_n}_{T,s}\rangle _{\so^\epsilon_n\fq}}
,\sum_{(T',s)\in\kappa \ \rm{mod} \ \so^\epsilon_{n}\fq}
C_{T'}  R^{\so^\epsilon_n}_{T',s}
\right\rangle _{\so^\epsilon_n\fq}\\
=&\varepsilon_{\so^\e_{n-1}} \varepsilon_S\langle  \pi^{reg}_\kappa,\tau_1  \rangle_{\so^\epsilon_n\fq}\\
=&\left\{
\begin{array}{ll}
\varepsilon_{\so^\e_{n-1}} \varepsilon_S,\textrm{ if } \tau_1  =\pi^{reg}_\kappa;\\
0, \textrm{ otherwise,}
\end{array}
\right.
\end{aligned}
\]
which implies that
$
\begin{aligned}
\langle \tau_1,\tau_2\rangle_{\so^\e_{n-1}\fq}=0,
\end{aligned}
$
if $\tau_1  \ne\pi^{reg}_\kappa$.

Suppose $\tau_1  =\pi^{reg}_\kappa$. With same argument, we have
\[
\begin{aligned}
\langle \tau_1,\tau_2\rangle_{\so^\e_{n-1}\fq}
=&\sum_{(S,s')\in\kappa' \ \rm{mod} \ \so^\e_{n-1}\fq}C_{S} \langle\tau_1,R^{\so^\e_{n-1}}_{S,s'}\rangle_{\so^\e_{n-1}\fq}\\
=&\sum_{(S,s')\in\kappa' \ \rm{mod} \ \so^\e_{n-1}\fq} \varepsilon_{\so^\e_{n-1}} \varepsilon_S C_S\\
=&\langle \pi^{reg}_{\kappa'},\tau_2\rangle_{\so^\e_{n-1}\fq}\\
=&\left\{
\begin{array}{ll}
1,\textrm{ if } \tau_2  =\pi^{reg}_{\kappa'};\\
0, \textrm{ otherwise.}
\end{array}
\right.
\end{aligned}
\]
\end{proof}
\section{Classification of quadratic unipotent representations}\label{sec4}
In the this section, we first review some results on the classification of the irreducible unipotent representations of symplectic groups and orthogonal groups by Lusztig in \cite{L1, L2, L3}. Then we give a parametrization of irreducible representations.

\subsection{Symbols}
We follow the notation of \cite{P3}. The notation is slightly different from that of \cite{L1}.

A symbol is an array of the form
\[
\Lambda=
\begin{pmatrix}
A\\
B
\end{pmatrix}
=
\begin{pmatrix}
a_1,a_2,\cdots,a_{m_1}\\
b_1,b_2,\cdots,b_{m_2}
\end{pmatrix}
\]
of two finite subsets $A$, $B$ (possibly empty) with $a_i, b_i\ge0$, $a_i>a_{i+1}$ and $b_i>b_{i+1}$.

The rank and defect of a symbol $\Lambda$ are defined by
\[
\begin{aligned}
&\rm{rank}(\Lambda)=\sum_{a_i\in A}a_i+\sum_{b_i\in B}b_i-\left\lfloor\left(\frac{|A|+|B|-1}{2}\right)^2\right\rfloor, \\
&\rm{def}(\Lambda)=|A|-|B|
\end{aligned}
\]
where $|X|$ denotes the cardinality of a finite set $X$. Note that the definition of $\rm{def}(\Lambda)$ differs from that of \cite{L1} p.133.

For a symbol $\Lambda=\begin{pmatrix}
A\\
B
\end{pmatrix}$, let $\Lambda^*$ (resp. $\Lambda_*$) denote the first row (resp. second row) of $\Lambda$, i.e. $\Lambda^*=A$ and $\Lambda_*=B$. For a symbol $\Lambda=\begin{pmatrix}
A\\
B
\end{pmatrix}$, let
$\Lambda^t=\begin{pmatrix}
B\\
A
\end{pmatrix}$.

Define an equivalence relation generated by the rule
\[
\begin{pmatrix}
a_1,a_2,\cdots,a_{m_1}\\
b_1,b_2,\cdots,b_{m_2}
\end{pmatrix}
\sim
\begin{pmatrix}
a_1+1,a_2+1,\cdots,a_{m_1}+1,0\\
b_1+1,b_2+1,\cdots,b_{m_2}+1,0
\end{pmatrix}.
\]
Note that the defect and rank are functions on the set of equivalence classes of symbols.

\subsection{Bi-partitions}\label{sec4.2}
Let $\lambda=[\lambda_1,\lambda_2,\cdots,\lambda_k]$ be a partition. We always assume that $\lambda_i\ge \lambda_{i+1}$. We denote by ${}^t\lambda$ the transpose of $\lambda$. For two partitions $\lambda=[\lambda_1,\lambda_2,\cdots,\lambda_k]$ and $\mu=[\mu_1,\mu_2,\cdots,\mu_l]$, we denote
\[
\lambda\preccurlyeq\mu\quad\textrm{if }\mu_i-1\le \lambda_i \le\mu_i\textrm{ for each }i.
\]

Let $\mathcal{P}_2(n)=\left\{\begin{bmatrix}
\lambda\\
\mu
\end{bmatrix}\right\}$
denote the set of bi-partitions of $n$ where $\lambda$, $\mu$ are partitions
and $|\lambda| + |\mu| = n$. To each symbol we can associate a bi-partition as follows:
\[
\Upsilon: \Lambda=\begin{pmatrix}
a_1,a_2,\dots,a_{m_1}\\
b_1,b_2,\dots,b_{m_2}
\end{pmatrix}
\mapsto
\begin{bmatrix}
a_1-(m_1-1),a_2-(m_1-2),\dots,a_{m_1-1}-1,a_{m_1}\\
b_1-(m_2-1),b_2-(m_2-1),\dots,b_{m_2-1}-1,b_{m_2}
\end{bmatrix}=\begin{bmatrix}
\lambda\\
\mu
\end{bmatrix}.
\]
We write
$\Upsilon(\Lambda)^*=\lambda$ and $\Upsilon(\Lambda)_*=\mu.$
Then we have a bijection:
\[
\Upsilon:\mathcal{S}_{n,\beta}\to \left\{
\begin{array}{ll}
\mathcal{P}_2(n-(\frac{\beta+1}{2})(\frac{\beta-1}{2})), &  \textrm{if }\ \beta\textrm{ is odd},\\
\mathcal{P}_2(n-(\frac{\beta}{2})^2), & \textrm{if }\ \beta\textrm{ is even}.
\end{array}\right.
\]
where $\mathcal{S}_{n,\beta}$ denotes the set of symbols of rank $n$ and defect $\beta$.

\subsection{Classification of unipotent representations}

Now we recall the correspondence on irreducible unipotent representations of symplectic groups and orthogonal groups. If $\pi\in\cal{E}(G,I)$, we say that $\pi$ is a unipotent representation. Lusztig gives a bijection between the unipotent representations of these groups to equivalence classes of symbols as follow:
\[
\left\{
\begin{aligned}
&\cal{E}(\sp_{2n},1)\\
&\cal{E}(\o_{2n+1},1)\\
&\cal{E}(\o^+_{2n},1)\\
&\cal{E}(\o^-_{2n},1)
\end{aligned}\right.
\longrightarrow
\left\{
\begin{aligned}
&\cal{S}_n:=\big\{\Lambda|\rm{rank}(\Lambda)=n, \rm{def}(\Lambda)=1\ (\textrm{mod }4)\big\};\\
&\cal{S}_n\times\{\pm\};\\
&\cal{S}^+_n:=\big\{\Lambda|\rm{rank}(\Lambda)=n, \rm{def}(\Lambda)=0\ (\textrm{mod }4)\big\};\\
&\cal{S}^-_n:=\big\{\Lambda|\rm{rank}(\Lambda)=n, \rm{def}(\Lambda)=2\ (\textrm{mod }4)\big\};
\end{aligned}\right.
\]
If $G$ is an even (resp. odd) orthogonal group, it is known that $\pi_{\Lambda^t} = \rm{sgn}\cdot\pi_{\Lambda}$ (resp. $\pi_{\Lambda,\epsilon}=\sgn\pi_{\Lambda,-\epsilon}$) where $\pi_{\Lambda}$ (resp. $\pi_{\Lambda,\epsilon}$) means the irreducible representation parametrized by $\Lambda$ (resp. $(\Lambda,\epsilon)$) and $\rm{sgn}$ denotes the sign character. Here we distinguish $\pi_{\Lambda,\pm}$ by decreeing that $\pi_{\Lambda,\pm}(-1)=\pm\rm{Id}$.

\subsection{Classification of quadratic unipotent representations}\label{4.4}
\begin{definition}

(i) If $G$ is orthogonal group and $\pi\in\cal{E}(G,-I)$, we say that $\pi$ is a $\theta$-epresentation. For $G^F =\sp_{2n}\fq$ we have $G^{*F} = \so_{2n+1}\fq$. Let $s = (-I, 1)$ with $I$ being the identity in $\so_{2n}^\epsilon\fq \in \so_{2n+1}\fq$. We say that $\pi$ is a $\theta$-epresentation if $\pi\in\cal{E}(G,s)$.

(ii) If $\pi\in\cal{E}(G,s)$ where $s$ satisfies $s^2 =I$, we say that $\pi$ is a quadratic unipotent representation. Let
\[
\rm{Quad}(G):=\{\pi\in\cal{E}(G)|\pi\textrm{ is quadratic unipotent}\}.
\]
\end{definition}
By the work of Lusztig \cite{L1} and Waldspurger \cite{W1}, we have a parametrization of the quadratic unipotent representations by a pair of symbols which generalizes that of the unipotent representations given above. We will give a parametrization of quadratic unipotent representations via (modified) Lusztig correspondence, which is slightly different from that in \cite{W1}. We think this definition here will be more convenient to use the results in \cite{P3, P4}.

By the (modified) Lusztig correspondence, there is a bijection between $\rm{Quad}(G)$ and $\bigcup_{s}\cal{E}(C_{G^{*}}(s),1)$ where $s$ satisfies $s^2=1$.
More explicitly, we have bijection between the quadratic unipotent representations of these groups to equivalence classes of symbols as follow:
\[
\left\{
\begin{aligned}
\rm{Quad}(\sp_{2n})\longrightarrow&\bigcup_{n_1+n_2=n}\cal{E}(\sp_{2n_1+1},1)\times\cal{E}(\o^\pm_{2n_2},1);\\
\rm{Quad}(\o_{2n+1})\longrightarrow&\bigcup_{n_1+n_2=n}\cal{E}(\sp_{2n_1},1)\times\cal{E}(\sp_{2n_2},1)\times\{\pm\};\\
\rm{Quad}(\o_{2n})\longrightarrow&\bigcup_{n_1+n_2=n}\cal{E}(\o^\pm_{2n_1},1)\times\cal{E}(\o^\pm_{2n_2},1);\\
\end{aligned}\right.
\]
where
\[
\rm{Quad}(\o_{2n})=\rm{Quad}(\o^+_{2n})\bigcup\rm{Quad}(\o^-_{2n}).
\]
Based on above bijection and Lusztig's classification of unipotent representations, we obtained in {\it loc. cit.} the following classification of quadratic unipotent representations:
\[
\left\{
\begin{aligned}
\rm{Quad}(\sp_{2n})\longrightarrow&\bigcup_{n_1+n_2=n}\cal{S}_{n_1}\times\cal{S}^\pm_{n_2};\\
\rm{Quad}(\o_{2n+1})\longrightarrow&\bigcup_{n_1+n_2=n}\cal{S}_{n_1}\times\cal{S}_{n_2}\times\{\pm\};\\
\rm{Quad}(\o_{2n})\longrightarrow&\bigcup_{n_1+n_2=n}\cal{S}^\pm_{n_1}\times\cal{S}^\pm_{n_2}.\\
\end{aligned}\right.
\]

Recall that the (modified) Lusztig correspondence is not uniquely determined. The parametrization of quadratic unipotent representations depends on the choice of the (modified) Lusztig correspondence.
Let
\[
\cal{L}_G': \cal{E}(G)\to\left\{
\begin{array}{ll}
\cal{E}(\dg,1)\otimes\cal{E}(\ddg,1)\otimes\cal{E}(\dddg,1)\otimes\{\pm\},&\textrm{if $G$ is an orthogonal group;}\\
\cal{E}(\dg,1)\otimes\cal{E}(\ddg,1)\otimes\cal{E}(\dddg,1),&\textrm{otherwise}
\end{array}\right.
\]
such that for $\pi\in\cal{E}(G,s)$, we have
\[
\cal{L}_G'(\pi)=
\cal{L}_s'(\pi).
\]
We call $\cal{L}_G'$ the modified Lusztig correspondence for $G$. For a fixed $\cal{L}_G'$, let $\pi_{\Lambda,\Lambda'}$ (resp. $\pi_{\Lambda,\Lambda',\epsilon}$) denote the irreducible quadratic unipotent representation
parametrized by the pair of symbols $(\Lambda,\Lambda')$ (resp. $(\Lambda,\Lambda',\epsilon)$) via $\cal{L}_G'$.

Note that $\pi_{\Lambda,-} $ (resp. $\pi_{\Lambda,-,\epsilon} $) is a unipotent representation of symplectic group or even orthogonal group (resp. odd orthogonal group) and $\pi_{\Lambda,-}=\pl $ (resp. $\pi_{\Lambda,-,\epsilon} =\pi_{\Lambda,\epsilon}$ ) where we write blank by $-$. On the other hand, $\pi_{-,\Lambda} $ (resp. $\pi_{-,\Lambda,\epsilon} $) is a $\theta$-epresentation of symplectic group or even orthogonal group (resp. odd orthogonal group). And we have
\[
\cal{L}'_s:\pi_{\Lambda,\Lambda'} \textrm{ (resp. }\pi_{\Lambda,\Lambda',\epsilon}\textrm{)}\to \pl\otimes\pll\textrm{ (resp. }\pl\otimes\pll\otimes\epsilon\textrm{)}.
\]

The following information may be read off of \cite[section 4]{W1} and \cite[section 6.1]{P3}.
\begin{proposition}\label{q1}
Let $G_n$ be $\sp_{2n}$, $\o^\pm_{2n}$ or $\o_{2n+1}$. For every $G_n$, there exists a modified Lusztig correspondence $\cal{L}'_{G_n}$ such that the following hold. Let $\pi_{\Lambda,\Lambda'} $ (resp. $\pi_{\Lambda,\Lambda',\epsilon} $) be a cuspidal quadratic unipotent representation of $G_n^F$.

(i) Let $G_n=\sp_{2n}$ and $G_m=\sp_{2m}$ with $m>n$. Let $\pi_{\Lambda,\Lambda'} $ be a cuspidal quadratic unipotent representation of $G_n^F$, and let $\pi_{\Lambda_1,\Lambda_1'}\in \rm{Quad}(G_m)$. If $\pi_{\Lambda_1,\Lambda_1'}\in \cal{E}(G_m,\pi_{\Lambda,\Lambda'} )$, then
\begin{itemize}

\item $\pi_{\Lambda_1,\Lambda_1^{\prime t}}\in \cal{E}(G_m,\pi_{\Lambda,\Lambda^{\prime t}} )$;

\item $\pi_{\Lambda_1,\Lambda_1'}^c:=\pi_{\Lambda_1,\Lambda_1'}(hgh^{-1})=\pi_{\Lambda_1,\Lambda_1^{\prime t}}$ where $g\in G^F_n$ and $h\in \rm{CSp}^\pm_{2n}\fq$ with $\zeta\circ\lambda(h)=-1$. (Here $\pi_{\Lambda_1,\Lambda_1'}^c$, $\zeta$ and $\lambda$ are defined in \cite{W1}.)

\item $\pi_{\Lambda,\Lambda'}(-I)=\pi_{\Lambda,\Lambda^{\prime t}}(-I)$.
\end{itemize}

(ii)  Let $G_n=\o^\epsilon_{2n}$ and $G_m=\o^\epsilon_{2m}$ with $m>n$. Let $\pi_{\Lambda,\Lambda'} $ be a cuspidal quadratic unipotent representation of $G_n^F$, and let $\pi_{\Lambda_1,\Lambda_1'}\in \rm{Quad}(G_m)$. If $\pi_{\Lambda_1,\Lambda_1'}\in \cal{E}(G_m,\pi_{\Lambda,\Lambda'} )$, then
\begin{itemize}

\item $\chi\otimes\pi_{\Lambda_1,\Lambda_1'}=\pi_{\Lambda_1',\Lambda_1}\in \cal{E}(G_m,\pi_{\Lambda',\Lambda} )$, where $\chi$ is the character $\rm{sp}$ defined in \cite[p10]{W1};

\item $\pi_{\Lambda_1,\Lambda_1^{\prime t}}\in \cal{E}(G_m,\pi_{\Lambda,\Lambda^{\prime t}} )$;

\item $\pi_{\Lambda_1^t,\Lambda_1'}\in \cal{E}(G_m,\pi_{\Lambda^t,\Lambda'} )$;

\item $\rm{sgn}\otimes\pi_{\Lambda_1,\Lambda_1'}=\pi_{\Lambda_1^t,\Lambda_1^{\prime t}}\in \cal{E}(G_m,\pi_{\Lambda^t,\Lambda^{\prime t}} )$

\item $\pi_{\Lambda_1,\Lambda_1'}^c:=\pi_{\Lambda_1,\Lambda_1'}(hgh^{-1})=\pi_{\Lambda_1,\Lambda_1^{\prime t}}$ where $g\in G^F_n$ and $h\in \rm{CO}^\pm_{2n}\fq$ with $\zeta\circ\lambda(h)=-1$ (Here $\pi_{\Lambda_1,\Lambda_1'}^c$, $\zeta$ and $\lambda$ are defined in \cite{W1}).
\end{itemize}

(iii) Let $G_n=\o_{2n+1}$ and $G_m=\o_{2m+1}$ with $m>n$. Let $\pi_{\Lambda,\Lambda'} $ be a cuspidal quadratic unipotent representation of $G_n^F$, and let $\pi_{\Lambda_1,\Lambda_1'}\in \rm{Quad}(G_m)$. If $\pi_{\Lambda_1,\Lambda_1',\epsilon'}\in \cal{E}(H,\pi_{\Lambda,\Lambda',\epsilon} )$, then
\begin{itemize}
\item $\e=\epsilon$.

\item $\chi\otimes\pi_{\Lambda_1,\Lambda_1',\epsilon}=\pi_{\Lambda_1',\Lambda_1,\epsilon}\in \cal{E}(G_m,\pi_{\Lambda',\Lambda,\epsilon} )$, where $\chi$ is the character $\rm{sp}$ defined in \cite[p10]{W1};

\item $\pi_{\Lambda_1,\Lambda_1^{\prime t},\epsilon}\in \cal{E}(G_m,\pi_{\Lambda,\Lambda^{\prime t},\epsilon} )$;

\item $\rm{sgn}\otimes\pi_{\Lambda_1,\Lambda_1',\epsilon}=\otimes\pi_{\Lambda_1,\Lambda_1',-\epsilon}\in \cal{E}(G_m,\pi_{\Lambda,\Lambda',-\epsilon} )$.
\end{itemize}

(iv) In each case, the defects of $(\Lambda,\Lambda')$ are preserved by parabolic induction. In other words, $\rm{def}(\Lambda_1)=\rm{def}(\Lambda)$ and $\rm{def}(\Lambda_1')=\rm{def}(\Lambda')$.
\end{proposition}

More generally, we have a parametrization of irreducible representations via the above choice of modified Lusztig correspondences as follows.
Let $\pi$ be an irreducible representation of $\sp_{2n}\fq$, $\o^\epsilon_{2n}\fq$ or $\o^\epsilon_{2n+1}\fq$. Suppose that
\[
\cal{L}'_G(\pi)=\p\otimes\pp\otimes\ppp=\rho\otimes\pl\otimes\pll,\textrm{ (resp. }\rho\otimes\pl\otimes\pll\otimes\e\textrm{)}
\]
where $\cal{L}'_G$ is the modified Lusztig correspondence. Then we denote $\pi$ by $\prll$ (resp. $\pi_{\rho,\Lambda,\Lambda',\epsilon}$). If $\dg$ is trivial, then $\prll=\pi_{-,\Lambda,\Lambda'}=\pi_{\Lambda,\Lambda'}$.

It is easily seen that there exists a modified Lusztig correspondence $\cal{L}'_{G_n}$ satisfying similar conditions in Proposition \ref{q1}. To be more explicitly, let $\prll\in\cal{E}(G_n)$ be an irreducible representation, and let $\pi_{\rho_1\Lambda_1,\Lambda_1'}\in\cal{E}(G_m,\prll)$. We can substitute $\prll, \pi_{\rho,\Lambda,\Lambda^{\prime t}}$ for $\pi_{\Lambda,\Lambda'}, \pi_{\Lambda,\Lambda^{\prime t}}$ in Proposition \ref{q1} and similar argument applies for other representations by the obvious way.

From now on, we fix a choice of modified Lusztig correspondences  $\cal{L}'_{G_n}$ satisfying the conditions in our discussion above. Thus we fix a parametrization for irreducible representations, and in particular, for quadratic unipotent representations. In what follows, we denote by $\pi_{\rho,\Lambda,\Lambda'}$ (resp. $\pi_{\rho,\Lambda,\Lambda',\epsilon}$) the irreducible representation corresponding $(\rho,\Lambda,\Lambda')$ (resp. $(\rho,\Lambda,\Lambda',\epsilon)$) in this parametrization, and denote it briefly by $\pi_{\Lambda,\Lambda'}$ (resp. $\pi_{\Lambda,\Lambda',\epsilon}$) for quadratic unipotent representations.

\section{Howe correspondence of unipotent representations for finite symplectic groups and
finite orthogonal groups} \label{sec6}
In this section we review the Howe correspondence of irreducible representations for finite symplectic groups and
finite orthogonal groups. We first recall the Howe correspondence for symplectic groups and even orthogonal groups. Then we deduce the $(\sp_{2n},\o_{2n'+1})$ case from the $(\sp_{2n},\o^\epsilon_{2n'})$ case by the modified Lusztig correspondence.

\subsection{Notations}
Let $\omega_{\rm{Sp}_{2N}}$ be the Weil representation or its character (cf. \cite{Ger}) of the finite symplectic group $\rm{Sp}_{2N}(\Fq)$, which depends on the choice of a nontrivial additive character $\psi$ of $\Fq$.
For the dual pair  $(\sp_{2n},\o^{\epsilon}_{2n'})$ where $\epsilon=\pm$, we write $\omega^\epsilon_{n,n'}$ for the restriction of $\omega_{\rm{Sp}_{2N}}$ to $\sp_{2n}\fq\times\o^{\epsilon}_{2n'}\fq$. Similar notation applies for $(\sp_{2n},\o^{\epsilon}_{2n'+1})$. When the context of dual pairs is clear, abbreviate by $\Theta^\epsilon_{n,n'}$ the theta lifting from $G_n$ to $G'_{n'}$.
For an irreducible representation $\pi$ of $G_{n}$, the smallest integer $n^\epsilon(\pi)$ such that $\pi$ occurs in $\omega^\epsilon_{n,n^\epsilon(\pi)}$ is called the {\it first occurrence index} of $\pi$ in the Witt tower $\left\{G'_{n}\right\}$.

Recall the convention that $\o^+_{2n}$ (resp. $\o^-_{2n}$) denotes the isometry group of the split (resp. nonsplit) form of dimension $2n$. For odd orthogonal groups, one has $\o^+_{2n+1}\cong\o^-_{2n+1}$  as abstract groups; however they act on two quadratic spaces with different discriminants. We write ${\bf Sp}$, ${\bf O}^\pm_{\rm{even}}$ and ${\bf O}^\pm_\rm{odd}$ for the Witt tower $\{\sp_{2n}\}_{n\ge0}$, $\{\o^\pm_{2n}\}_{n\ge0}$ and $\{\o^\pm_{2n+1}\}_{n\ge0}$.

\subsection{Pan's result}
  In \cite{AMR}  conjecture 3.11, Aubert, Michel and Rouquier give a explicit description of the theta correspondence of unipotent representations for a dual pair $(\sp_{2n},\o^\epsilon_{2n'})$. In \cite{P3}, Pan proves their conjecture.

 Let
\[
\begin{aligned}
&\mathcal{B}^+_{n,n'}:=\left\{(\Lambda,\Lambda')|{}^t(\Upsilon(\Lambda')_*)\preccurlyeq{}^t(\Upsilon(\Lambda)^*),
{}^t(\Upsilon(\Lambda)_*)\preccurlyeq{}^t(\Upsilon(\Lambda')^*),\rm{def}(\Lambda')=-\rm{def}(\Lambda)+1\right\}; \\
&\mathcal{B}^-_{n,n'}:=\left\{(\Lambda,\Lambda')|{}^t(\Upsilon(\Lambda')^*)\preccurlyeq {}^t(\Upsilon(\Lambda)_*),
{}^t(\Upsilon(\Lambda)^*)\preccurlyeq{}^t(\Upsilon(\Lambda')_*),\rm{def}(\Lambda')=-\rm{def}(\Lambda)-1\right\}
\end{aligned}
\]
be two subsets of $\cal{S}_n\times\cal{S}_{n'}^{+}$ and $\cal{S}_n\times\cal{S}_{n'}^{-}$, respectively.
 \begin{theorem}\label{p}
 Let $\pl\in\cal{E}(\sp_{2n},1)$ and $\pll\in \cal{E}(\o^\epsilon_{2n'},1)$. Then $\pl\otimes\pll$ occurs in $\omega_{n,n'}^\epsilon$ if and only if $(\Lambda,\Lambda')\in \mathcal{B}^\epsilon_{n,n'}$.
 \end{theorem}

 Recall that
  \[
\cal{L}_s': \cal{E}(G,s)\to
\left\{
  \begin{aligned}
 &\cal{E}(\dg,1)\times\cal{E}(\ddg,1)\times\cal{E}(\dddg,1)\times\{\pm\}&\textrm{ if }G\textrm{ is odd orthogonal};\\
  &\cal{E}(\dg,1)\times\cal{E}(\ddg,1)\times\cal{E}(\dddg,1)&\textrm{ otherwise}.
\end{aligned}
\right.
 \]

\begin{theorem}[Pan]\label{p1}
 Let $(G, G') = (\sp_{2n},\o^\epsilon_{2n'})$, and let $\pi\in\cal{E}(G,s)$ and $\pi'\in \cal{E}(G',s')$ for
some semisimple elements $s \in G^*$ and $s'\in (G^{\prime *})^0$. Write $\cal{L}'_s(\pi)=\p\otimes\pp\otimes\ppp$ and
and $\cal{L}_{s'}(\pi')=\pi^{\prime (1)}\otimes\pi^{\prime (2)}\otimes\pi^{\prime (3)}$, and let $\{\pi_i'\}$ be defined in (\ref{imega2}). Suppose that $q$ is large enough so that the main result in \cite{S} holds. Then $\pi \otimes \pi_i '$ (for some $i$) occurs in $\omega_{n,n'}^\epsilon$ if and only if the following conditions hold:
\begin{itemize}

\item $\dg\cong G^{\prime (1)}(s')$, $\p\cong \pi^{\prime (1)}$;

\item either $\pp\otimes\pi^{\prime (2)}$ or $\pp\otimes(\sgn\pi^{\prime (2)})$ occurs in $\omega_{\ddg,\ddf}$;

\item $\dddg\cong\dddf$, $\pp$ is equal to $ \pi^{\prime (3)}$ or $ \sgn\pi^{\prime (3)}$.
\end{itemize}
That is, the following diagram:
\[
\setlength{\unitlength}{0.8cm}
\begin{picture}(20,5)
\thicklines
\put(5.8,4){$\pi$}
\put(5.2,2.6){$\cal{L}'_s$}
\put(13.8,2.6){$\cal{L}_{s'}'$}
\put(9.4,4.2){$\Theta$}
\put(4.3,1){$\p\otimes\pp\otimes\ppp$}
\put(13.4,4){$\pi_i'$}
\put(11.8,1){$\pi^{\prime (1)}\otimes\pi^{\prime (2)}\otimes\pi^{\prime (3)}$}
\put(8.5,1.3){$\rm{id}\otimes\Theta\otimes\rm{id}$}
\put(6,3.6){\vector(0,-1){2.1}}
\put(13.5,3.6){\vector(0,-1){2.1}}
\put(8.2,1.1){\vector(2,0){3}}
\put(8.2,4){\vector(2,0){3}}
\end{picture}
\]
commutes up to a twist of the sgn character.
\end{theorem}

\begin{theorem}[Pan]\label{p2}
Let $(G, G') = (\sp_{2n},\o^\epsilon_{2n'+1})$, and let $\pi \in  \cal{E}(G,s)$ and $\pi'\in \cal{E}(G',s')$ for some semisimple elements $s \in G^*$ and $s'\in (G^{\prime *})^0$. Write $\cal{L}'_s(\pi)=\p\otimes\pp\otimes\ppp$ and
and $\cal{L}_{s'}'(\pi')=\pi^{\prime (1)}\otimes\pi^{\prime (2)}\otimes\pi^{\prime (3)}\otimes\e$, and let $\{\pi_i\}$ be defined as in (\ref{imega1}). Then $\pi_i \otimes \pi '$ occurs in $\omega^\epsilon_{n,n'}$ for some $i$ if and only if the following conditions hold:
\begin{itemize}

\item $\dg\cong\df$, $\p\cong \pi^{\prime (1)}$;

\item $\ddg\cong\dddf$, $\pp\cong\pi^{\prime (3)}$;

\item either $\ppp\otimes\pi^{\prime (2)}$ or $(\sgn \ppp)\otimes\pi^{\prime (2)}$ occurs in $\omega_{\dddg,\ddf}$.
\end{itemize}
That is, the following diagram:
\[
\setlength{\unitlength}{0.8cm}
\begin{picture}(20,5)
\thicklines
\put(5.8,4){$\pi_i$}
\put(5.2,2.6){$\cal{L}'_s$}
\put(13.8,2.6){$\iota\circ\cal{L}_{s'}'$}
\put(9.4,4.2){$\Theta$}
\put(4.3,1){$\p\otimes\pp\otimes\ppp$}
\put(13.4,4){$\pi'$}
\put(11.8,1){$\pi^{\prime (1)}\otimes\pi^{\prime (3)}\otimes\pi^{\prime (2)}$}
\put(8.5,1.3){$\rm{id}\otimes\rm{id}\otimes\Theta$}
\put(6,3.6){\vector(0,-1){2.1}}
\put(13.5,3.6){\vector(0,-1){2.1}}
\put(8.2,1.1){\vector(2,0){3}}
\put(8.2,4){\vector(2,0){3}}
\end{picture}
\]
commutes up to a twist of the sgn character where $\iota(\pi^{\prime (1)}\otimes\pi^{\prime (2)}\otimes\pi^{\prime (3)}\otimes\epsilon')=\pi^{\prime (1)}\otimes\pi^{\prime (3)}\otimes\pi^{\prime (2)}$.

\end{theorem}

\begin{remark}\label{sgn}
In \cite{P4}, above two Theorems hold for any modified Lusztig correspondences, which implies that any two different choices of modified Lusztig correspondences are equal up to sgn.
\end{remark}

Therefore, the description of general Howe correspondence for dual pair of a symplectic group and an orthogonal group is now completely characterized up to sgn. We remark that if $\pi\otimes\pi'$ occurs in $\omega^\epsilon_{n,n'}$, then $\pi(-I)=\pi'(-I)$ (see \cite[Proposition 3.1 (i)]{LW1}). So, keep the assumptions in Theorem \ref{p2}, the sgn of $\pi'$ is uniquely determined.

\subsection{Theta lifting and parabolic induction}

We now shows that theta lifting and parabolic induction are compatible.

\begin{lemma}\label{5.5}
(i) Let $\pi$ be an irreducible representation of $\sp_{2m}\fq$. Let $m>n$, and let $\sigma$ be an irreducible cuspidal representation of $\sp_{2n}\fq$. Let $\{\pi_i\}$ be defined as in (\ref{imega1}). Then there is at most one of $\pi_i$ appearing in $\cal{E}(\sp_{2m},\sigma)$.

(ii) Let $\pi$ be an irreducible representation of $\o^\epsilon_{2m}\fq$. Let $m>n$, and let $\sigma$ be an irreducible cuspidal representation of $\o^\epsilon_{2n}\fq$. Let $\{\pi_i\}$ be defined as in (\ref{imega2}). Then there is at most one of $\pi_i$ appearing in $\cal{E}(\o^\epsilon_{2m},\sigma)$.

(iii) Let $\pi$ be an irreducible representation of $\o^\epsilon_{2m+1}\fq$. Let $m>n$, and let $\sigma$ be an irreducible cuspidal representation of $\o^\epsilon_{2n+1}\fq$. Assume that $\pi\in \cal{E}(\o^\epsilon_{2m+1},\sigma)$. Then $\sgn\pi\notin \cal{E}(\o^\epsilon_{2m+1},\sigma)$.
\end{lemma}
\begin{proof}
It follows immediately from Proposition \ref{q1}.
\end{proof}

In \cite{LW3}, we know the Howe correspondence of representations in the Harish-Chandra series $\cal{E}(G,\sigma)$ for a cuspidal representation $\sigma$.

\begin{proposition} [Proposition 5.8 in \cite{LW3}]\label{o3}
Let  $(G_m, G'_{m'})$ be a dual pair  in the Witt tower  $(\bf{Sp},\bf{O}^\epsilon_{\rm{even}})$ or $(\bf{Sp},\bf{O}^\epsilon_{\rm{odd}})$. Assume that $\pi\in \cal{E}(G_m,\sigma)$, where $\sigma$ is an irreducible cuspidal representation of $G_n^F$, $n\leq m$, $n\equiv m \ \rm{mod} \ 2$. Let $n' = n^\epsilon(\sigma)$ be its first occurrence index, so that $\sigma' := \Theta^\epsilon_{n,n'}(\sigma)$ is an irreducible cuspidal representation of $G_{n'}^{\prime F}$. Then the following hold.

(i) The irreducible constituents of $\Theta^\epsilon_{m,m'}(\pi)$ belong to   $\cal{E}(G_{m'}^{\prime },\sigma')$,

(ii) If $m'-m\ge {n'-n}$, then $\Theta^\epsilon_{m,m'}(\pi)\ne 0$.
\end{proposition}

\begin{corollary} \label{p3}
Let  $(G, G')$ be a dual pair  in the Witt tower  $(\bf{Sp},\bf{O}^\epsilon_{\rm{even}})$ or $(\bf{Sp},\bf{O}^\epsilon_{\rm{odd}})$. Let $\pi\in \cal{E}(\sp_{2n},\sigma)$ and $\pi'\in\cal{E}(\o^\epsilon_{2n'+1},\sigma')$ (resp. $\pi'\in\cal{E}(\o^\epsilon_{2n'},\sigma')$). Let $\pi$ and $\pi'$ satisfying the conditions in Theorem \ref{p1} (resp. Theorem \ref{p2}).

(i) Let $(G,G')= (\sp_{2n},\o^\epsilon_{2n'})$, and let $\{\pi_i'\}$ be defined as in (\ref{imega2}). Assume that $\pi'_i\in\cal{E}(\o^\epsilon_{2n'},\sigma'_i)$. Then
\begin{itemize}

\item There is exactly one of $\pi\otimes\pi_i'$ appearing in $ \omega^\epsilon_{n,n'}$;

\item  $\pi\otimes\pi_i'$ appears in $ \omega^\epsilon_{n,n'}$ if and only if $\sigma\otimes\sigma'_i$ appears in $\omega^\epsilon_{m,m'}$ for some $m,m'$.

\item If $\prllc\in\Theta_{n,n_1'}^\epsilon(\pi)$ and $\prlld\in\Theta_{n,n_2'}^\epsilon(\pi)$, then $\rm{def}(\Lambda_1)=\rm{def}(\Lambda_2)$ and $\Lambda_1'=\Lambda_2'$.
\end{itemize}

(ii)  Let $(G,G')= (\sp_{2n},\o^\epsilon_{2n'+1})$, and let $\{\pi_i\}$ be defined as in (\ref{imega1}). Assume that $\pi_i\in\cal{E}(\sp_{2n},\sigma_i)$ and $\pi'\in\cal{E}(\o^\epsilon_{2n'},\sigma')$. Then
\begin{itemize}

\item There is exactly one of $\pi_i\otimes\pi'$ appearing in $ \omega^\epsilon_{n,n'}$;

\item  $\pi_i\otimes\pi'$ appears in $ \omega^\epsilon_{n,n'}$ if and only if $\sigma_i\otimes\sigma'$ appears in $\omega^\epsilon_{m,m'}$ for some $m,m'$.

\item If $\pi_{\rho,\Lambda_1,\Lambda_1',\epsilon_1}\in\Theta_{n,n_1'}^\epsilon(\pi)$ and $\pi_{\rho,\Lambda_2,\Lambda_2',\epsilon_2}\in\Theta_{n,n_2'}^\epsilon(\pi)$, then $\rm{def}(\Lambda_1)=\rm{def}(\Lambda_2)$, $\epsilon_1=\epsilon_2$ and $\Lambda_1'=\Lambda_2'$.
\end{itemize}

\end{corollary}
\begin{proof}
By Proposition \ref{q1}, for every $i\ne j$, we have $\sigma_i\ne\sigma_j$ (resp. $\sigma'_i\ne\sigma_j'$). Note that the defects are preserved by parabolic induction. Then the Corollary follows immediately from Proposition \ref{o3}, Theorem \ref{p1} and Theorem \ref{p2}.
\end{proof}

The next result shows that the theta lifting and the parabolic induction are compatible.

\begin{proposition}[Proposition 3.1 in \cite{LW3}]\label{w1}
Let $G_n$ and $G_{n+\ell}$ be two classical groups in the same Witt tower, $\ell\geq 0$. Let $\tau$ be an irreducible cuspidal representation of $\GGL_\ell(\mathbb{F}_{q})$, $\pi$ be an irreducible representation of $G_n(\Fq)$, and  $\pi':=\Theta_{n,n'}(\pi)$. Let $\chi_{\GGL_{\ell}}$ be the unique linear character of $\GGL_\ell\fq$ of order $2$. Let $\rho\subset I^{G_{n+\ell}}_{\GGL_\ell\times G_{n}}(\tau \otimes \pi)$ be an irreducible representation of $G_{n+\ell}$ and $\rho' \subset \Theta_{n+\ell, n'+\ell}(\rho)$ be an irreducible representation of $G'_{n'+\ell}$.  Assume that $\tau$ is non-selfdual if $\ell=1$. Then we have
\[
\rho'\subset I^{G'_{n'+\ell}}_{\GGL_\ell\times G'_{n'}}((\chi\otimes\tau)\otimes \pi'),
\]
where
\[
\chi=\left\{
\begin{array}{ll}
\chi_{\GGL_{\ell}}, &  \textrm{if $(G_{n+\ell},G'_{n'+\ell})$ contains an odd orthogonal group,}\\
1, & \textrm{otherwise.}
\end{array}\right.
\]
In particular, if $I^{G_{n+\ell}}_{\GGL_\ell\times G_{n}}(\tau \otimes \pi)$ is irreducible, then
\[
\Theta_{n+\ell, n'+\ell}(I^{G_{n+\ell}}_{\GGL_\ell\times G_{n}}(\tau \otimes \pi))= I^{G'_{n'+\ell}}_{\GGL_\ell\times G'_{n'}}((\chi\otimes\tau)\otimes \pi').
\]
\end{proposition}

Suppose $I^{G_{n+\ell}}_{\GGL_\ell\times G_{n}}(\tau \otimes \pi)$ is not irreducible. We have following result which generalizes above Proposition.
\begin{proposition}\label{w2}
Let $G_n$ and $G_{n+\ell}$ be two classical groups in the same Witt tower of symplectic groups or orthogonal groups, $\ell\geq 0$. Let $\pi\in\cal{E}(G_n,s)$ be an irreducible representation of $G_n(\Fq)$ with $s\in G^*_n(\Fq)$, and  $\pi':=\Theta_{n,n'}(\pi)$. Let $\chi_{\GGL_{\ell}}$ be the unique linear character of $\GGL_\ell\fq$ of order $2$. Let $\tau\in\cal{E}(\GGL_\ell,s_0)$ be an irreducible cuspidal representation of $\GGL_\ell(\mathbb{F}_{q})$ with $s_0\in \GGL_\ell\fq$. Let $ I^{G_{n+\ell}}_{\GGL_\ell\times G_{n}}(\tau \otimes \pi)=\bigoplus_{i}\rho_i$ with $\rho_i$ irreducible.  Assume that $s$ and $s_0$ have no common eigenvalues, and $s_0$ has no eigenvalues $\pm1$. Then we have
\[
\bigoplus_{i}\Theta_{n+\ell, n'+\ell}(\rho_i)=I^{G'_{n'+\ell}}_{\GGL_\ell\times G'_{n'}}((\chi\otimes\tau)\otimes \pi').
\]
where
\[
\chi=\left\{
\begin{array}{ll}
\chi_{\GGL_{\ell}}, &  \textrm{if $(G_{n+\ell},G'_{n'+\ell})$ contains an odd orthogonal group,}\\
1, & \textrm{otherwise.}
\end{array}\right.
\]
Hence, by abuse of notations, we write
\[
\Theta_{n+\ell,n'+\ell}(I^{G_{n+\ell}}_{\GGL_\ell\times G_{n}}(\tau \otimes \pi))=I^{G'_{n'+\ell}}_{\GGL_\ell\times G'_{n'}}((\chi\otimes\tau)\otimes \pi').
\]
\end{proposition}
\begin{proof}
If $\tau$ is not self dual, then $I^{G_{n+\ell}}_{\GGL_\ell\times G_{n}}(\tau \otimes \pi)$ is irreducible. Then the proposition follows from Proposition \ref{w1}.

Suppose that $\tau$ is self dual. Let
\[
\cal{L}'_s(\rho_i)=\rho_i^{(1)}\otimes\rho_i^{(2)}\otimes\rho_i^{(3)}.
\]
Since $s_0$ has no eigenvalues $\pm1$, we have $\rho_i^{(1)}\ne \rho_j^{(1)}$ for $i\ne j$. Then by Theorem \ref{p1} and Theorem \ref{p2}, for any irreducible representation $\rho'$ of $G'_{n'+\ell}$, at most one of $\rho_i\otimes\rho'$ appears in $\omega^\epsilon_{n+\ell,n'+\ell}$. In other words, if $\rho'$ appears in $\Theta_{n+\ell,n'+\ell}(\rho_i)$, then we have
\[
\langle \omega^\epsilon_{n+\ell,n'+\ell}, \rho_i\otimes\rho'\rangle=\langle \omega^\epsilon_{n+\ell,n'+\ell}, I^{G_{n+\ell}}_{\GGL_\ell\times G_{n}}(\tau \otimes \pi)\otimes\rho'\rangle.
\]
By the proof of Proposition 3.1 in \cite{LW3}, we have
\[
\langle \omega^\epsilon_{n+\ell,n'+\ell}, I^{G_{n+\ell}}_{\GGL_\ell\times G_{n}}(\tau \otimes \pi)\otimes\rho'\rangle
=\langle  I^{ G'_{n'+\ell}}_{\GGL_{\ell}\times G'_{n'}}((\chi\otimes\tau)\otimes\pi') ,\rho' \rangle,
\]
which implies
\[
\bigoplus_{i}\Theta_{n+\ell, n'+\ell}(\rho_i)=I^{G'_{n'+\ell}}_{\GGL_\ell\times G'_{n'}}((\chi\otimes\tau)\otimes \pi').
\]
\end{proof}
\subsection{Symbol $\wla$}

Consider $(G,G')=(\o^\pm_{2n},\sp_{2n'})$.

Let $\prll$ be an irreducible representation of $\o^\epsilon_{2n}\fq$. Define
\[
\rm{def}(\Theta(\rho,\Lambda,\Lambda',\o^\epsilon)):=\{\rm{def}(\Lambda_1)|\prll\otimes\pi_{\rho,\Lambda_1,\Lambda_1'}\in \omega^\epsilon_{n,n'}\}.
\]
It is well defined by Theorem \ref{p1} and Corollary \ref{p3}.
\begin{definition}
For $\prll$ with $\rm{def}(\Lambda)\ne 0$, define
\[
\widetilde{\Lambda}\in\{\Lambda,\Lambda^{ t}\}
\]
satisfying
\[
\rm{def}(\widetilde{\Lambda})=\left\{
\begin{aligned}
&-\rm{def}(\Theta(\rho,\Lambda,\Lambda',\o^\epsilon))+1\textrm{ if }\epsilon=+;\\
&-\rm{def}(\Theta(\rho,\Lambda,\Lambda',\o^\epsilon))-1\textrm{ if }\epsilon=-.\\
\end{aligned}
\right.
\]

\end{definition}
By Theorem \ref{p1} and Corollary \ref{p3}, for two triples $(\rho,\Lambda,\Lambda')$ and $(\rho,\Lambda^t,\Lambda')$ with $\rm{def}(\Lambda)\ne 0$, we have $\widetilde{\Lambda}^t=\widetilde{\Lambda^t}$. Then we have a bijection:
\[
\begin{aligned}
\prwll \longleftrightarrow(\rho, \wla, \Lambda').
\end{aligned}
\]
Note that $\widetilde{\Lambda}$ depends on $(\rho,\Lambda,\Lambda')$, i.e. for two triples $(\rho_1,\Lambda,\Lambda_1')$ and $(\rho_2,\Lambda,\Lambda_2')$, we may pick different $\widetilde{\Lambda}$.
\begin{corollary} \label{ctheta2}
Let $\prll\in \cal{E}(\o^\epsilon_{2n})$ with $\rm{def}(\Lambda)\ne 0$.

(i) If $\prll\otimes\prllc$ appears in $\omega_{n,n'}^\epsilon$, then $(\Lambda_1,\wla)\in \mathcal{B}^{\epsilon_{\Lambda}}_{m_1,m}$ where $\epsilon_{\Lambda}=(-1)^{\frac{\rm{def}(\Lambda)}{2}}$, $m=\rm{rank}(\Lambda)$ and $m_1=\rm{rank}(\Lambda_1)$. In paritcular, $\pi_{\widetilde{\Lambda}}\otimes\pi_{\Lambda_1}$ appears in $\omega_{m,m_1}^{\epsilon_{\Lambda}}$.

(ii) If $(\rho,\Lambda,\Lambda')=(-,\Lambda,-)$, then $\wla=\Lambda$.
\end{corollary}
\begin{proof}
(ii) clearly follows from the definition of $\widetilde{\Lambda}$ and Theorem \ref{p}. We will only prove (i) for $\epsilon_\Lambda=+$, and the proof for $\epsilon_\Lambda=-$ is similar.

Suppose that $\prll\in\cal{E}(\o^\epsilon_{2n},s)$. Recall that
\[
C_{G^{*F}}(s)\cong\dg\times\ddg\times\dddg,
\]
where $\ddg\cong\o^{\epsilon_\Lambda}_{2\nu_{1}(s)}\fq$ and $\dddg\cong\o^{\epsilon_{\Lambda'}}_{2\nu_{-1}(s)}\fq$.
By Theorem \ref{p1}, either $(\Lambda_1,\Lambda)\in \mathcal{B}^+_{m_1,m}$ or $(\Lambda_1,\Lambda^{ t})\in \mathcal{B}^+_{m_1,m}$. On the other hand, for any $\Lambda^*\in\{\Lambda,\Lambda^t\}$, $(\Lambda_1,\Lambda^*)\in \mathcal{B}^+_{m_1,m}$ if and only if $\rm{def}(\Lambda')=-\rm{def}(\Lambda^*)+1$. Hence, (i) follows from the definition of $\widetilde{\Lambda}$.
\end{proof}

 Now consider $(G,G')=(\sp_{2n},\o^\pm_{2n'+1})$. In the same manner, by Theorem \ref{p} and Theorem \ref{p2}, we have
\begin{corollary} \label{ctheta3}
Let $\prll\in\cal{E}(\sp_{2n})$, and let $\pi_{\rho_1,\Lambda_1,\Lambda'_1}\in\cal{E}(\o^\epsilon_{2n'+1})$. Let $\epsilon_{\Lambda'}=(-1)^{\frac{\rm{def}(\Lambda')}{2}}$. If $\prll\otimes\pi_{\rho_1,\Lambda_1,\Lambda'_1}$ appears in $\omega^\epsilon_{m',m}$ for some $m$ and $m'$, then there exist a symbol $\widetilde{\Lambda'}\in\{\Lambda',\Lambda^{\prime t}\}$ such that $(\Lambda_1,\widetilde{\Lambda'})\in \cal{B}^{\epsilon_{\Lambda'}}_{n,n'}$ where $\widetilde{\Lambda'}$ depends on $\epsilon$, $\rho$ and $\Lambda$.
\end{corollary}

\subsection{See-saw pairs}
Recall the general formalism of see-saw dual pairs. Let $(G, G')$ and $(H, H')$ be two reductive dual pairs in a symplectic group $\rm{Sp}(W)$ such that $H \subset G$ and $G' \subset H'$.  Then there is a see-saw diagram
\[
\setlength{\unitlength}{0.8cm}
\begin{picture}(20,5)
\thicklines
\put(6.8,4){$G$}
\put(6.8,1){$H$}
\put(12.1,4){$H'$}
\put(12,1){$G'$}
\put(7,1.5){\line(0,1){2.1}}
\put(12.3,1.5){\line(0,1){2.1}}
\put(7.5,1.5){\line(2,1){4.2}}
\put(7.5,3.7){\line(2,-1){4.2}}
\end{picture}
\]
and the associated see-saw identity
\[
\langle \Theta_{G',G}(\pi_{G'}),\pi_H\rangle_H =\langle \pi_{G'},\Theta_{H,H'}(\pi_H)\rangle_{G'},
\]
where $\pi_H$ and $\pi_{G'}$ are representations of $H$ and $G'$ respectively.

In this paper, we consider the following two cases.
(1) Consider the case that
\[
G\cong\o^\epsilon_{2n}\fq,\ H\cong\o_{2n-1}^{\epsilon'}\fq\times\o_1^{\epsilon''}\fq,\ H'\cong\sp_{2n'}\fq,\times\sp_{2n'}\fq\textrm{ and }G'\cong\sp_{2n'}\fq,
\]
where $\epsilon\cdot\ee=\epsilon'\cdot\epsilon''$ so that $H$ is embedded into $G$ by \cite{LW3} (1.3), and $G'$ is embedded into $H'$ diagonally.

(2) Consider the case that
 \[
 G\cong\o_{2n+1}^\epsilon\fq,\ H\cong\o^\e_{2n}\fq\times\o_1^{\epsilon''}\fq,\ H'\cong\sp_{2n'}\fq\times\sp_{2n'}\fq\textrm{ and }G'\cong\sp_{2n'}\fq,
 \]
 where $\epsilon=\e\cdot\epsilon''$.
For fixed $\epsilon$, so that $H$ is embedded into $G$ again by \cite{LW3} (1.3).
\section{First occurrence index of cuspidal representations}\label{sec7}

For a dual pair $(G_n,G'_{n'})$ and an irreducible cuspidal representation $\pi$ of $G_n$, by \cite[Theorem 2.2]{AM} , there is a cuspidal representation appearing in the theta lifting if and only if $n'$ is the first occurrence index of $\pi$. Moreover, the theta lifting of $\pi$ is an irreducible cuspidal representation.

\subsection{Symbol of unipotent cuspidal representations}
In \cite{L1}, we know that $\sp_{2n}$ (resp,  $\rm{SO}_{2n+1}$, $\rm{SO}^{\epsilon}_{2n}$) has a unique irreducible unipotent cuspidal representation if and only if $n=k(k+1)$ (resp. $n=k(k+1)$, $n=k^2$).

For $\sp_{2k(k+1)}$, the unique unipotent cuspidal representation $\pi_{\Lambda}$ is associated to the symbol:
\begin{equation}\label{cussp}
\Lambda=\left\{
\begin{aligned}
&\begin{pmatrix}
2k,2k-1,\cdots,1,0\\
-
\end{pmatrix}\textrm{ if }k\textrm{ is even};\\
&\begin{pmatrix}
-\\
2k,2k-1,\cdots,1,0
\end{pmatrix}\textrm{ if }k\textrm{ is odd};\\
\end{aligned}\right.
\end{equation}
of defect$(-1)^k(2k + 1)$.

The trivial character of the trivial group $O^+_0$ is regarded as unipotent cuspidal and is associated to the symbol $\begin{pmatrix}
-\\
-
\end{pmatrix}$. For $\o^{\epsilon}_{2k^2}$ with $\epsilon=(-1)^k$, there are two unipotent cuspidal representations $\pi_{\Lambda}$ and $\sgn\pl=\pi_{\Lambda^t}$ where \[\Lambda=\begin{pmatrix}
2k-1,2k-2,\cdots,1,0\\
-
\end{pmatrix}.\]
For $\o_{2k(k+1)+1}$, there are two unipotent cuspidal representations $\pi_{\Lambda,+}$ and $\sgn\pi_{\Lambda,+}=\pi_{\Lambda,-}$ where \begin{equation}\label{cusso}
\Lambda=\left\{
\begin{aligned}
&\begin{pmatrix}
2k,2k-1,\cdots,1,0\\
-
\end{pmatrix}\textrm{ if }k\textrm{ is even};\\
&\begin{pmatrix}
-\\
2k,2k-1,\cdots,1,0
\end{pmatrix}\textrm{ if }k\textrm{ is odd}.\\
\end{aligned}\right.
\end{equation}

\subsection{First occurrence index}\label{6.2}
\begin{theorem}[\cite{AM}, Theorem 5.2] \label{even}
The theta correspondence for dual pairs $(\sp_{2n}, \o^\epsilon_{2n'})$ takes unipotent cuspidal representations to unipotent cuspidal representations as follows :

(i)  $(\sp_{2k(k+ 1)}, \o^\epsilon _{2k^2})$, $\epsilon = \rm{sgn}(-1)^k$,
\[
\Theta^{\epsilon}_{k(k+1), k^2}: \left\{
\begin{aligned}
&\pl,\ \Lambda=\begin{pmatrix}
2k,\cdots,1,0\\
-
\end{pmatrix}
\to
\pll,\ \Lambda'=\begin{pmatrix}
-\\
2k-1,\cdots,1,0\\
\end{pmatrix}
\textrm{ if }k \textrm{ is even;}\\
&\pl,\ \Lambda=\begin{pmatrix}
-\\
2k,\cdots,1,0
\end{pmatrix}
\to
\pll,\ \Lambda'=\begin{pmatrix}
2k-1,\cdots,1,0\\
-
\end{pmatrix}
\textrm{ if }k \textrm{ is odd}.
\end{aligned}\right.
 \]

(ii)  $(\sp_{2k(k+ 1)}, \o^\epsilon _{2(k+1)^2})$, $\epsilon = \rm{sgn}(-1)^{k+1}$,
\[
\Theta^{\epsilon}_{k(k+1), (k+1)^2}: \left\{
\begin{aligned}
&\pl,\ \Lambda=\begin{pmatrix}
2k,\cdots,1,0\\
-
\end{pmatrix}
\to
\pll,\ \Lambda'=\begin{pmatrix}
-\\
2k+1,\cdots,1,0\\
\end{pmatrix}
\textrm{ if }k \textrm{ is even;}\\
&\pl,\ \Lambda=\begin{pmatrix}
-\\
2k,\cdots,1,0
\end{pmatrix}
\to
\pll,\ \Lambda'=\begin{pmatrix}
2k+1,\cdots,1,0\\
-
\end{pmatrix}
\textrm{ if }k \textrm{ is odd}.
\end{aligned}\right.
 \]
\end{theorem}

Let $\prll$ (resp. $\pi_{\rho,\Lambda,\Lambda',\epsilon}$) be an irreducible representation of $\sp_{2n}\fq$, $\o^\pm_{2n}\fq$ or $\o_{2n+1}\fq$. Assume that $\Lambda$ and $\Lambda'$ correspond to unipotent cuspidal representations of $\ddg$ and $\dddg$, respectively. Let
\[
k=\left\{
\begin{aligned}
&\frac{|\rm{def}(\Lambda)|-1}{2}&\textrm{ if } \Lambda\in\cal{S}_m;\\
&\frac{\rm{def}(\Lambda)}{2}&\textrm{ if } \Lambda\in\cal{S}^\pm_m
\end{aligned}\right.
\]
and
\[
h=\left\{
\begin{aligned}
&\frac{|\rm{def}(\Lambda')|-1}{2}&\textrm{ if } \Lambda'\in\cal{S}_{m'};\\
&\frac{\rm{def}(\Lambda')}{2}&\textrm{ if } \Lambda'\in\cal{S}^\pm_{m'}.
\end{aligned}\right.
\]
For abbreviation, we write $\pkh$ (resp. $\pi_{\rho,k,h,\epsilon}$) instead of $\prll$ (resp. $\pi_{\rho,\Lambda,\Lambda',\epsilon}$).
We emphasize that $\pkh$ (resp. $\pi_{\rho,k,h,\epsilon}$) is \emph{not} necessarily cuspidal.

\begin{proposition}\label{cus1}
Let $\pkh$ be an irreducible representation of $\sp_{2n}\fq$.

(i) Let $n^\epsilon$ be the first occurrence index of $\pkh$ in the Witt tower ${\bf O}^\epsilon_\rm{even}$. Then either
\[
\left\{
\begin{array}{ll}
n^+=n- k;\\
n^-=n+k+1
\end{array}\right.
\textrm{ or }
\left\{
\begin{array}{ll}
n^+=n+ k+1;\\
n^-=n-k.
\end{array}\right.
\]
Moreover, if $n^\epsilon=n-k$, then $\Theta_{n,n^\epsilon}^\epsilon(\pkh)=\pi_{\rho,k_1,h_1}$ is irreducible with $k_1\in\{\pm k\}$ and $h_1\in\{\pm h\}$, and $\Theta_{n^\epsilon,n}^\epsilon(\pi_{\rho,k_1,h_1})=\pkh$. The first occurrence index of $\pi_{\rho,k,-h}$ in the Witt tower ${\bf O}^\epsilon_\rm{even}$ is also $n^\epsilon$, and $\Theta_{n,n^\epsilon}^\epsilon(\pi_{\rho,k,-h})\ne\pi_{\rho,k_1,h_1}$.

(ii) Let $n_1$ be the first occurrence index of $\pkh$ in the Witt tower ${\bf O}^\epsilon_\rm{odd}$. Then either
\[
\left\{
\begin{array}{ll}
n^+=n- h;\\
n^-=n+h
\end{array}\right.
\textrm{ or }
\left\{
\begin{array}{ll}
n^+=n+ h;\\
n^-=n-h.
\end{array}\right.
\]
Moreover, if $n^\epsilon=n-|h|$, then $\Theta_{n,n^\epsilon}^\epsilon(\pkh)=\pi_{\rho,k_1,h_1,\e}$ is irreducible with $k_1= |h|-1$ and $h_1= k$, and $\Theta_{n^\epsilon,n}^\epsilon(\pi_{\rho,k_1,h_1,\e})=\pkh$. The first occurrence index of $\pi_{\rho,k,-h}$ in the Witt tower ${\bf O}^{\epsilon}_\rm{odd}$ is $n+|h|$.

\end{proposition}
\begin{proof}
 The Proposition follows immediately from Theorem \ref{p1}, Theorem \ref{p2}, Corollary \ref{p3} and Theorem \ref{even} no matter whether $\pkh$ is cuspidal or not.
\end{proof}

\begin{proposition}\label{cus2}

(i) Let $\pkh$ be an irreducible representation of $\o^\epsilon_{2n}\fq$.  Let $n'$ be the first occurrence index of $\pkh$ in the Witt tower ${\bf Sp}$. Then $n'=n\pm k$.
Moreover, the following hold.
\begin{itemize}
\item
 If $n'=n+ |k|$, then $\Theta_{n,n'}^\epsilon(\pkh)=\pi_{\rho,k_1,h_1}$ with $k_1=|k|$ and $h_1\in\{\pm h\}$, and $\Theta_{n',n}^\epsilon(\pi_{\rho,k_1,h_1})=\pkh$. The first occurrence index of $\sgn\pi_{\rho,k,h}$ is $n-|k|$.
\item
 If $n'=n- |k|$, then $\Theta_{n,n'}^\epsilon(\pkh)=\pi_{\rho,k_1,h_1}$ with $k_1=|k|-1$ and $h_1\in\{\pm h\}$, and $\Theta_{n',n}^\epsilon(\pi_{\rho,k_1,h_1})=\pkh$. The first occurrence index of $\sgn\pi_{\rho,k,h}$ is $n+|k|$.
 \item
 Let $n_1$ and $n_2$ be the first occurrence index of $\pi_{\rho,-k,h}$ and $\pi_{\rho,k,-h}$. Then $n_2=n'$ and $n_1+n'=2n$, and $\Theta^\epsilon_{n,n'}(\pi_{\rho,k,-h})=\pi_{\rho,k_1,-h_1}$ and $\Theta^\epsilon_{n',n}(\pi_{\rho,k_1,-h_1})=\pi_{\rho,k,-h}$.
\end{itemize}

(ii) Let $\pi_{\rho,k,h,\e}$ be an irreducible representation of $\o^\epsilon_{2n+1}$.  Let $n'$ be the first occurrence index of $\pi_{\rho,k,h,\e}$ in the Witt tower ${\bf Sp}$. Then either $n'=n+ k+1$ or $n'=n-k$.
Moreover, the following hold.
\begin{itemize}
\item
 If $n'=n+ k+1$, then $\Theta_{n,n'}^\epsilon(\pi_{\rho,k,h,\e})=\pi_{\rho,k_1,h_1}$ with $k_1= h$ and $h_1\in\{\pm(k+1)\}$, and $\Theta_{n',n}^\epsilon(\pi_{\rho,k_1,h_1})=\pi_{\rho,k,h,\e}$. The first occurrence index of $\sgn\pi_{\rho,k,h,\e}$ is $n-k$.
\item
 If $n'=n- k$, then $\Theta_{n,n'}^\epsilon(\pi_{\rho,k,h,\e})=\pi_{\rho,k_1,h_1}$ with $k_1=h$ and $h_1\in\{\pm k\}$, and $\Theta_{n',n}^\epsilon(\pi_{\rho,k_1,h_1})=\pi_{\rho,k,h,\e}$. The first occurrence index of $\sgn\pi_{\rho,k,h,\e}$ is $n+k+1$.
\end{itemize}

\end{proposition}
\begin{proof}
We only prove that last part of (i). The rest follows immediately from Theorem \ref{p1}, Theorem \ref{p2}, Corollary \ref{p3} and Theorem \ref{even}, and the conservation relation for cuspidal representations given in \cite[Theorem 12.3]{P1} as above Proposition.

By the same argument of the proof of Proposition 4.3 in \cite{W1}, we have $\pi_{\rho,k,h}^c=\pi_{\rho,k,-h}$ for the irreducible cuspidal representation of finite symplectic groups and finite even orthogonal groups. Let
\[
\omega_{n,n'}^{\epsilon,c}(g):=\omega_{n,n'}^{\epsilon}(xgx^{-1})
\]
where $g\in \o^\epsilon_{2n}\fq\times\sp_{2n'}\fq$ and $x=x_1\times x_2\in \rm{CO}^\epsilon_{2n}\fq\times\rm{CSp}_{2n'}\fq$ with $\zeta\circ\lambda_{x_i}=-1$ (see Proposition \ref{q1}). Since there is only one Weil representation for dual pair $(\o^\epsilon_{2n}\fq,\sp_{2n'}\fq)$, we conclude that $\omega_{n,n'}^{\epsilon,c}=\omega_{n,n'}^{\epsilon}$. So
\[
\begin{aligned}
\langle \pi_{\rho,k,h}\otimes\pi_{\rho,k_1,h_1},\omega^{\epsilon}_{n,n'}\rangle_{\o^\epsilon_{2n}\fq\times\sp_{2n'}\fq}=&\langle \pi_{\rho,k,h}^c\otimes\pi_{\rho,k_1,h_1}^c,\omega^{\epsilon,c}_{n,n'}\rangle_{\o^\epsilon_{2n}\fq\times\sp_{2n'}\fq}\\
=&\langle \pi_{\rho,k,-h}\otimes\pi_{\rho,k_1,-h_1},\omega^{\epsilon}_{n,n'}\rangle_{\o^\epsilon_{2n}\fq\times\sp_{2n'}\fq}.
\end{aligned}
\]
By the conservation relation and Proposition \ref{q1} (ii), we have $n_1+n'=2n$.
\end{proof}

\subsection{Strongly relevant pair of representations}\label{6.3}

Denote by $\epsilon_{-1}$, the square class of $-1$.
\begin{definition}\label{strongly relevant0}
Let $\psi$ be a fixed nontrivial additive character of $\Fq$.

(i) Let $\pi$ be an irreducible cuspidal representation of $\sp_{2n}\fq$, and let $\pi'$ be an irreducible cuspidal representation of $\sp_{2m}\fq$. Let $n^\epsilon$ be the first occurrence index of $\pi$ in the Witt tower ${\bf O}^\epsilon_\rm{even}$, and let $m^{\epsilon'}$ be the first occurrence index of $\pi'$ in the Witt tower ${\bf O}^{\epsilon'}_\rm{odd}$. Pick $\epsilon\in\{\pm\}$ such that $n^\epsilon\le n$. We say the pair of representations $(\pi,\pi')$ is $(\psi,\epsilon_0)$-relevant if $n-n^\epsilon=m-m^{\epsilon\cdot\epsilon_0}-1$ or $n-n^\epsilon= m-m^{\epsilon\cdot\epsilon_0}$. We say the pair of representations $(\pi,\pi')$ is $(\psi,\epsilon_0)$-strongly relevant if $(\pi,\pi')$ is $(\psi,\epsilon_0)$-relevant and $(\pi',\pi)$ is $(\psi,\ee\cdot\epsilon_0)$-relevant. It is easily to see $(\pi,\pi')$ is $(\psi,\epsilon_0)$-strongly relevant if and only if $(\pi',\pi)$ is $(\psi,\ee\cdot\epsilon_0)$-strongly relevant.

(ii) Let $\pi_{\rho_{},\Lambda_{},\Lambda_{}'}\in\cal{E}(\sp_{2n},\sigma_{})$, and let $\pi_{\rho_1,\Lambda_1,\Lambda_1'}\in\cal{E}(\sp_{2m},\sigma_1)$. Assume that either $\rm{def}(\Lambda_1')\ne 0$ or $\Lambda_1'=\begin{pmatrix}
-\\
-
\end{pmatrix}$. We say the pair of representations $(\pi_{\rho_{},\Lambda_{},\Lambda_{}'},\pi_{\rho_1,\Lambda_1,\Lambda_1'})$ is $(\psi,\epsilon_0)$-relevant if $(\sigma_{},\sigma_1)$ is $(\psi,\epsilon_0)$-relevant.

 If $\rm{def}(\Lambda_1')= 0$ and $\Lambda_1'\ne\begin{pmatrix}
-\\
-
\end{pmatrix}$, then we define $(\psi,\epsilon_0)$-relevant pair by induction on $|\Upsilon(\Lambda_1')^*|+|\Upsilon(\Lambda_1')_*|$.
 Let $n^\epsilon$ be the first occurrence index of $\pi_{\rho_{},\Lambda_{},\Lambda_{}'}$ in the Witt tower ${\bf O}^\epsilon_\rm{even}$, and let $m^{\epsilon'}$ be the first occurrence index of $\pi_{\rho_1,\Lambda_1,\Lambda_1'}$ in the Witt tower ${\bf O}^{\epsilon'}_\rm{odd}$. Pick $\epsilon\in\{\pm\}$ such that $n^\epsilon\le n$.
 Let $\pi_{}^\star=\Theta^\epsilon_{n,n^\epsilon}(\pi_{\rho_{},\Lambda_{},\Lambda_{}'})$ and $\pi_1^\star=\Theta^{\epsilon\cdot\epsilon_0}_{m,m^{\epsilon\cdot\epsilon_0}}(\pi_{\rho_1,\Lambda_1,\Lambda_1'})$. By Theorem \ref{p1} and Theorem \ref{p2}, we know that $\pi^\star_{}$ and $\pi^\star_1$ are irreducible.  Let $n^\star$ and $m^\star$ be the first occurrence index of $\pi^\star_{}$ and $\pi^\star_1$ in the Witt tower ${\bf Sp}$, respectively. With same argument, we know that $\Theta^\epsilon_{n^\epsilon,n^\star}(\pi_{}^\star)$ and $\Theta^{\epsilon\cdot\epsilon_0}_{m^{\epsilon\cdot\epsilon_0},m^\star}(\pi_1^\star)$ are irreducible, and we denote them by $\pi_{\rho_{},\Psi_{},\Psi_{}'}$ and  $\pi_{\rho_1,\Psi_1,\Psi_1'}$. Note that $|\Upsilon(\Psi_1')^*|+|\Upsilon(\Psi_1')_*|<|\Upsilon(\Lambda_1')^*|+|\Upsilon(\Lambda_1')_*|$ and $\rm{def}(\Psi_1)=0$.

 We say the pair of representations $(\pi_{\rho_{},\Lambda_{},\Lambda_{}'},\pi_{\rho_1,\Lambda_1,\Lambda_1'})$ is $(\psi,\epsilon_0)$-relevant if the following hold:
 \begin{itemize}
\item $n-n^\epsilon=m-m^{\epsilon\cdot\epsilon_0}-1$ or $n-n^\epsilon= m-m^{\epsilon\cdot\epsilon_0}$;
\item $n^\epsilon-n^\star-1=m^{\epsilon\cdot\epsilon_0}-m^\star$ or $n^\epsilon-n^\star= m^{\epsilon\cdot\epsilon_0}-m^\star$;
\item $(\pi_{\rho_{},\Psi_{},\Psi_{}'},\pi_{\rho_1,\Psi_1,\Psi_1'})$ is $(\psi,\epsilon_0)$-\emph{strongly} relevant.
\end{itemize}
 We say the pair of representations $(\pi_{\rho_{},\Lambda_{},\Lambda_{}'},\pi_{\rho_1,\Lambda_1,\Lambda_1'})$ is $(\psi,\epsilon_0)$-strongly relevant if $(\pi_{\rho_{},\Lambda_{},\Lambda_{}'},\pi_{\rho_1,\Lambda_1,\Lambda_1'})$ is $(\psi,\epsilon_0)$-relevant and $(\pi_{\rho_1,\Lambda_1,\Lambda_1'},\pi_{\rho_{},\Lambda_{},\Lambda_{}'})$ is $(\psi,\ee\cdot\epsilon_0)$-relevant.

(iii) We will write $(\psi,\epsilon_0)$-relevant (resp. $(\psi,\epsilon_0)$-strongly relevant) simply $\epsilon_0$-relevant (resp. $\epsilon_0$-strongly relevant) when no confusion can arise.
\end{definition}

For orthogonal groups, the first occurrence index does not depend on $\psi$. In fact, for even orthogonal groups, the Weil representation is the same for different choices of $\psi$. For odd orthogonal groups, let $\omega_{\rm{Sp}_{2N},\psi}$ and $\omega_{\rm{Sp}_{2N},\psi'}$  be the Weil representations of the finite symplectic group $\rm{Sp}_{2N}(\Fq)$ corresponding to $\psi$ and $\psi'$ respectively. Note that restricted to the dual pairs $\sp_{2n'}\fq\times\o_{2n+1}^\epsilon\fq$ with $N=n'(2n+1)$, one has
\begin{equation}\label{omega}
\omega^\epsilon_{n',n,\psi}\cong \omega^{-\epsilon}_{n',n,\psi'}
\end{equation}
via the isomorphism $\o_{2n+1}^\epsilon\cong \o_{2n+1}^{-\epsilon}$.
Let $\pi\in \cal{E}(\o_{2n+1}^\epsilon,s)$ be an irreducible cuspidal representation. Assume that the first occurrence index of $\pi$ (resp. $\rm{sgn}\cdot\pi$) is $n^\epsilon_1$ (resp. $n^\epsilon_2$) and $\Theta^\epsilon_{n,n^\epsilon_1}(\pi)=\prll$ (resp. $\Theta^\epsilon_{n,n^\epsilon_2}(\rm{sgn}\cdot\pi)=\prllc$). Then $\prll$ (resp. $\prllc$) is cuspidal and so is $\pi_{\rho,\Lambda,\Lambda^{\prime t}}$ (resp. $\pi_{\rho,\Lambda_1,\Lambda_1^{\prime t}}$).
By Theorem \ref{p2}, if $1$ is not a eigenvalue of $s$, then $n^\pm_i=n$ do not depend on $\psi$.
Assume that $s$ has a eigenvalue $1$. By Proposition \ref{q1}, we have
\[
\prll(-I)=\pi_{\rho,\Lambda,\Lambda^{\prime t}}(-I)\textrm{ and }\prllc(-I)=\pi_{\rho,\Lambda_1,\Lambda_1^{\prime t}}(-I).
 \]
 Recall that for any $\pi_1\in\cal{E}(\o^\epsilon_{2n+1})$ and $\pi_2\in\cal{E}(\sp_{2n^\epsilon_1})$, if $\pi_1\otimes\pi_2$ appears in $\omega^\epsilon_{n,n^\epsilon_1}$, then $\pi_1(-I)=\pi_2(-I)$. Then $\pi_{\rho,\Lambda_1,\Lambda_1^{\prime t}}(-I)=\sgn\pi(-I)\ne \pi(-I)$, therefore $\pi\otimes\pi_{\rho,\Lambda_1,\Lambda_1^{\prime t}}$ does not appear in $\omega^{-\epsilon}_{n,n^\epsilon_2}$. On the other hand, by Theorem \ref{p2} and the conservation relation, either
$\pi\otimes\pi_{\rho,\Lambda_1,\Lambda_1^{\prime t}}$ or $\pi\otimes\pi_{\rho,\Lambda,\Lambda^{\prime t}}$ appears in $\omega^{-\epsilon}_{n,m}$ for some $m$. So the first occurrence index of $\pi$ is $n^\epsilon_1$ in $\o_{2n+1}^{-\epsilon}$.

\begin{definition}\label{strongly relevant2}

(i) Let $\pi$ be an irreducible cuspidal representation of $\rm{O}^\epsilon_{2n+1}(\Fq)$, and let $\pi'$ be an irreducible cuspidal representation of $\rm{O}^{\epsilon'}_{2m}(\Fq)$. For $\chi_0\in\{1,\rm{sgn}\}$, let $n_0^\epsilon$ (resp. $m_0^{\epsilon'}$) be the first occurrence index of $\chi_0\otimes\pi$ (resp. $\chi_0\otimes\pi'$).
Pick $\chi_0$ such that $n_0^\epsilon\le n$. We say the pair of representations $(\pi,\pi')$ is relevant if
$n-n_0^\epsilon=m-m_0^{\epsilon'}-1\textrm{ or }n-n_0^\epsilon= m-m_0^{\epsilon'}$.
 We say the pair of representations $(\pi,\pi')$ is strongly relevant if both $(\pi,\pi')$ and $(\chi\otimes\pi,\chi\otimes\pi')$ are relevant where $\chi$ is defined as Proposition \ref{q1}.

(ii) Let $\pi_{\rho_{},\Lambda_{},\Lambda_{}',\epsilon_{}}\in\cal{E}(\rm{O}^\epsilon_{2n+1},\sigma_{})$, and let $\pi_{\rho_1,\Lambda_1,\Lambda_1'}\in\cal{E}(\rm{O}^{\epsilon'}_{2m},\sigma_1)$. Assume that either $\rm{def}(\Lambda_1)\ne 0$ or $\Lambda_1=\begin{pmatrix}
-\\
-
\end{pmatrix}$.
We say the pair of representations $(\pi_{\rho_{},\Lambda_{},\Lambda_{}',\epsilon_{}},\pi_{\rho_1,\Lambda_1,\Lambda_1'})$ is relevant if $(\sigma_{},\sigma_1)$ is relevant.

 If $\rm{def}(\Lambda_1)= 0$ and $\Lambda_1\ne\begin{pmatrix}
-\\
-
\end{pmatrix}$, then we define relevant pair by induction on $|\Upsilon(\Lambda_1)^*|+|\Upsilon(\Lambda_1)_*|$. Fix a nontrivial additive character $\psi$ of $\Fq$. For $\chi_0\in\{1,\rm{sgn}\}$, let $n_0^\epsilon$ (resp. $m_0^{\epsilon'}$) be the first occurrence index of $\chi_0\otimes\pi_{\rho_{},\Lambda_{},\Lambda_{}',\epsilon_{}}$ (resp. $\chi_0\otimes\pi_{\rho_1,\Lambda_1,\Lambda_1'}$).
Pick $\chi_0$ such that $n_0^\epsilon\le n$. Let $\pi_{\psi,1}^\star=\Theta^\epsilon_{n,n_0^\epsilon}(\pi_{\rho_{},\Lambda_{},\Lambda_{}',\epsilon_{}})$ and $\pi_{\psi,2}^\star=\Theta^{\epsilon'}_{m,m_0^{\epsilon'}}(\pi_{\rho_1,\Lambda_1,\Lambda_1'})$. By Theorem \ref{p1} and Theorem \ref{p2}, we know that $\pi^\star_{\psi,1}$ and $\pi^\star_{\psi,2}$ are irreducible.  Let $n^\star$ and $m^\star$ be the first occurrence index of $\pi^\star_{\psi,1}$ and $\pi^\star_{\psi,2}$ in the Witt tower ${\bf O}^{\epsilon}_\rm{odd}$ and ${\bf O}^{\epsilon'}_\rm{even}$, respectively. With same argument, we know that $\Theta^\epsilon_{n_0^\epsilon,n^\star}(\pi_{\psi,1}^\star)$ and $\Theta^{\epsilon'}_{m_0^{\epsilon'},m^\star}(\pi_{\psi,2}^\star)$ are irreducible, and we denote them by $\pi_{\rho_{},\Psi_{},\Psi_{}',\epsilon_{}}$ and  $\pi_{\rho_1,\Psi_1,\Psi_1'}$. Note that $|\Upsilon(\Psi_1)^*|+|\Upsilon(\Psi_1)_*|<|\Upsilon(\Lambda_1)^*|+|\Upsilon(\Lambda_1)_*|$  and $\rm{def}(\Psi_1)=0$.

 We say the pair of representations $(\pi_{\rho_{},\Lambda_{},\Lambda_{}',\epsilon_{}},\pi_{\rho_1,\Lambda_1,\Lambda_1'})$ is relevant if the following hold:
 \begin{itemize}
\item $n-n_0^\epsilon=m-m_0^{\epsilon'}-1\textrm{ or }n-n_0^\epsilon= m-m_0^{\epsilon'}$;
\item $n_0^\epsilon-n^\star-1=m_0^{\epsilon'}-m^\star$ or $n_0^\epsilon-n^\star= m_0^{\epsilon'}-m^\star$;
\item $(\pi_{\rho_{},\Psi_{},\Psi_{}',\epsilon_{}},\pi_{\rho_1,\Psi_1,\Psi_1'})$ is \emph{strongly} relevant.
\end{itemize}

 We say the pair of representations $(\pi_{\rho_{},\Lambda_{},\Lambda_{}',\epsilon_{}},\pi_{\rho_1,\Lambda_1,\Lambda_1'})$ is strongly relevant if both $(\pi_{\rho_{},\Lambda_{},\Lambda_{}',\epsilon_{}},\pi_{\rho_1,\Lambda_1,\Lambda_1'})$ and $(\chi\otimes\pi_{\rho_{},\Lambda_{},\Lambda_{}',\epsilon_{}},\chi\otimes\pi_{\rho_1,\Lambda_1,\Lambda_1'})$ are relevant where $\chi$ is defined as Proposition \ref{q1}.

Recall that the first occurrence index $n_0^\epsilon$ and $m_0^{\epsilon'}$ do not depend on $\psi$. By conservation relation in \cite[section 9]{P4}, the first occurrence index $n^\star$ and $m^\star$ also do not depend on $\psi$. By Theorem \ref{p1}, Theorem \ref{p2} and Lemma \ref{5.5}, the different choice of $\psi$ does not change the representations $\pi_{\rho_{},\Psi_{},\Psi_{}',\epsilon_{}}$ and $\pi_{\rho_1,\Psi_1,\Psi_1'}$. So the definition of relevant pair and strongly relevant does not depond on the choice of $\psi$.

\end{definition}

\begin{corollary}\label{strongly relevant}
Let $\pkh$ be an irreducible representation of $\sp_{2n}\fq$, and let $\pkhp$ be an irreducible representation of and $\sp_{2m}\fq$.
For any $\psi$ and $\epsilon_0$,  the following hold.

(i) If $(\pkh,\pkhp)$ is $(\psi,\epsilon_0)$-relevant, then $k=|h'|$ or $k=|h'|-1$.

(ii) If $(\pkh,\pkhp)$ is $(\psi,\epsilon_0)$-relevant, then $(\pkh,\pi_{\rho,k',-h'})$ is not.

(iii) If $(\pkh,\pkhp)$ is not $(\psi,\epsilon_0)$-relevant and $k=|h'|$ or $k=|h'|-1$, then $(\pkh,\pi_{\rho,k',-h'})$ is $(\psi,\epsilon_0)$-relevant.
\end{corollary}

\begin{proof}
It follows immediately from Proposition \ref{cus1}.
\end{proof}

For orthogonal groups, we have follow result. It follows immediately from Proposition \ref{cus2}.
\begin{corollary}\label{strongly relevant1}
(i) Let $\pkh$ be an irreducible representation of $\o^\epsilon_{2n}\fq$, and let $\pi_{\rho',k',h',\epsilon''}$ be an irreducible representation of $\o^\e_{2m+1}\fq$. If $(\pkh,\pi_{\rho',k',h',\epsilon''})$ is relevant, then $|k|=k'$ or $|k|=k'-1$.

(ii) Let $\pi_{\rho,k,h,\epsilon''}$ be an irreducible representation of $\o^\epsilon_{2n+1}\fq$, and let $\pkhp$ be an irreducible representation of and $\o^\e_{2m}\fq$. If $(\pi_{\rho,k,h,\epsilon''},\pkhp)$ is relevant, then $|k'|=k$ or $|k'|=k+1$.
\end{corollary}

\section{The Gan-Gross-Prasad problem: cuspidal case}\label{sec8}

From now on, we fix a character $\psi$ of $\Fq$. We write $\omega_{n,\psi}$ simply $\omega_n$ when no confusion can arise. Let $\pi\in\cal{E}(\sp_{2n},s)$ be an irreducible representation of $\sp_{2n}\fq$. Recall that
\[
\cal{L}'_s(\pi)=\p\otimes\pp\otimes\ppp.
\]
where $\cal{L}'_s$ is the modified Lusztig correspondence. In this section we study the Gan-Gross-Prasad problem for representation $\pi$ such that $\pp$ and $\ppp$ are cuspidal, i.e. we consider the representations which is of the form $\pkh$.

By abuse of notation, for $\pi=\pi_{\rho,-,-}\in\cal{E}(\sp_{2n})$ and $\pi'=\pi_{\rho',-,-}\in\cal{E}(\sp_{2m})$, we write
\begin{equation}\label{8}
m_\psi(\pi,\pi')=\left\{
\begin{array}{ll}
m_\psi(\pi,\pi'),&\textrm{ if }n\ge m;\\
m_\psi(\pi',\pi),&\textrm{ if }n< m.
\end{array}
\right.
\end{equation}
If $n=m$, then by Proposition \ref{rho}, we know that (\ref{8}) is well defined.
For any irreducible representations, $\pi\in\cal{E}(\o^\epsilon_{n})$ and $\pi'\in\cal{E}(\o^\e_{m})$, we write
\[
m(\pi,\pi')=\left\{
\begin{array}{ll}
m(\pi,\pi'),&\textrm{ if }n> m;\\
m(\pi',\pi),&\textrm{ if }n< m.
\end{array}
\right.
\]

We will prove the following result, which is the  Fourier-Jacobi case of Theorem \ref{main1}.
\begin{theorem}\label{8.1}
Let $n\ge m$. Let $\pkh$ be an irreducible representation of $\sp_{2n}\fq$, and let $\pkhp$ be an irreducible representation of $\sp_{2m}\fq$. Then
\[
m_\psi(\pkh,\pkhp)=\left\{
\begin{array}{ll}
m_\psi(\pw_{\rho} ,\pw_{\rho'})  &\textrm{ if }(\pkh,\pkhp)\textrm{ is $(\psi,\ee)$-strongly relevant};\\
0&\textrm{ otherwise}
\end{array}\right.
\]
where $m_\psi(\pw_{\rho} ,\pw_{\rho'})$ does not depend on $\psi.$
\end{theorem}

For the Bessel case, we have the similar result, and we will only give a sketch of the proof.
\begin{theorem}\label{8.2}
Let $\pi_{\rho,h,k,\epsilon''}$ be an irreducible representation of $\o^\epsilon_{2n+1}\fq$, and let $\pkhp$ be an irreducible representation of and $\o^\e_{2m}\fq$. Then
\[
m(\pi_{\rho,h,k,\epsilon''},\pkhp)=\left\{
\begin{array}{ll}
m_\psi(\pw_{\rho} ,\pw_{\rho'})  &\textrm{ if }(\pi_{\rho,h,k,\epsilon''},\pkhp)\textrm{ is strongly relevant};\\
0&\textrm{ otherwise}
\end{array}\right.
\]
where $m_\psi(\pw_{\rho} ,\pw_{\rho'})$ is the same thing as in Theorem \ref{8.1}.
\end{theorem}

\subsection{Reduction to the basic case}
We first show that parabolic induction preserves multiplicities, and thereby make a reduction to the basic case. We need Proposition
\ref{w1} and the following result which generalizes \cite[Proposition 6.1]{LW3}. Similar to \cite[Proposition 6.1]{LW3}, the proof of Proposition \ref{so2} is an adaptation of that of \cite[Theorem 16.1]{GGP1}. Recall that $ \overline{\omega_n^\epsilon}= \omega_n^{\ee\cdot\epsilon}$.
\begin{proposition}\label{so2}
Let $s$ be a semisimple element of $\sp_{2n}(\Fq)^*=\so_{2n+1}\fq$, and $s'$ be a semisimple element of $\sp_{2m}(\Fq)^*=\so_{2m+1}\fq$. Let $\pi\in\cal{E}(\sp_{2n},s)$ be an irreducible representation of $\sp_{2n}(\Fq)$, and let $\pi'\in\cal{E}(\sp_{2m},s')$ be an irreducible representation of $\sp_{2m}$ with $n \ge m$. Let $P$ be an $F$-stable maximal parabolic subgroup of $\sp_{2n}$ with Levi factor $\GGL_{n-m} \times \sp_{2m}$. Let $s_0$ be a semisimple element of $\GGL_{n-m}\fq$ and let $\tau\in\cal{E}(\GGL_{n-m},s_0)$ be an irreducible cuspidal representation of $\GGL_{n-m}\fq$ which is nontrivial  if $n-m=1$. Assume that $s_0$ has no common eigenvalues with $s$ and $s'$. Then we have
\begin{equation}\label{sp}
m_\psi(\pi, \pi')=\langle \pi\otimes\bar{\nu},\pi'\rangle_{H(\Fq)}=\langle  \pi\otimes \overline{\omega_n^+}, I_{P}^{\sp_{2n}}(\tau\otimes\pi')\rangle _{\sp_{2n}(\Fq)}=\langle  \pi\otimes \omega_n^{\ee}, I_{P}^{\sp_{2n}}(\tau\otimes\pi')\rangle _{\sp_{2n}(\Fq)},
\end{equation}
where the data $(H,\nu)$ is given by \cite[(1.2)]{LW3}.
\end{proposition}
\begin{proof}
It can be proved in the same way as \cite[Theorem 16.1]{GGP1}. The cuspidality assumption of $\pi$ in \cite[Theorem 16.1]{GGP1} was used to obtain the following statement: for an $F$-stable maximal parabolic subgroup $P'$ of $\sp_{2n}$ with Levi factor $\GGL_{n-m} \times \sp_{2m}$,
\[
\langle I^{\sp_{2n}}_{P'}\left(\tau\otimes(\pi'\otimes\omega_m^+\right),\pi\rangle _{\sp_{2n}(\Fq)}=0.
\]
Since in our case $s_0$ has no common eigenvalues with $s$ and $s'$, this  multiplicity is zero.
The rest of the proof is the same as that of \cite[Theorem 16.1]{GGP1}.
\end{proof}

We also have similar result for Bessel case which generalizes Proposition 5.3 and Corollary 5.4 in \cite{LW3}.
\begin{proposition} \label{7.21}
Let $s$ be a semisimple element of $\rm{SO}^\epsilon_n(\Fq)^*$, and $s'$ be a semisimple element of $\rm{SO}^{\epsilon'}_m(\Fq)^*$. Let $\pi\in\cal{E}(\rm{SO}^\epsilon_n,s)$ be an irreducible representation of $\rm{SO}^\epsilon_n(\Fq)$, and let $\pi'\in\cal{E}(\rm{SO}^{\epsilon'}_m,s')$ be an irreducible representation of $\rm{SO}^{\epsilon'}_m(\Fq)$ with $n > m$, $n\equiv m+1$ mod $2$. Let $P$ be an $F$-stable maximal parabolic subgroup of $\rm{SO}^{\epsilon'}_{n+1}$ with Levi factor $\GGL_{\ell} \times \rm{SO}^{\epsilon'}_{m}$, $\ell=(n+1-m)/2$. Let $s_0$ be a semisimple element of $\GGL_{\ell}(\bb{F}_{q})$. Let $\tau_1$ (resp. $\tau_2$) be an irreducible cuspidal representations of $\GGL_{\ell'}(\bb{F}_{q})$ (resp. $\GGL_{\ell-\ell'}(\bb{F}_{q}))$, $\ell'\leq \ell$, which is nontrivial if $\ell'=1$ (resp.  $\ell-\ell'=1$), and
\[
\tau=I_{\GGL_{\ell'}\times  \GGL_{\ell- \ell'}}^{\GGL_\ell}(\tau_1\times\tau_2).
\]
Assume that $\tau\in\cal{E}(\GGL_\ell,s_0)$, and $s_0$ has no common eigenvalues with $s$ and $s'$. Then we have
\begin{equation}\label{7.22}
m(\pi, \pi')=\langle \pi\otimes \bar{\nu}, \pi'\rangle _{H(\Fq)}=\langle I^{\so^\e_{n+1}}_{P}(\tau\otimes\pi'),\pi\rangle _{\so^\epsilon_{n}(\Fq)},
\end{equation}
where the data $(H,\nu)$ is given by \cite[(1.2)]{LW3}.
\end{proposition}

\begin{corollary}\label{o4}
Keep the assumptions in Proposition \ref{7.21}. Then
\begin{equation} \label{eq-cor}
m(\pi, \pi')=\langle I^{\so^\e_{n+1}}_{P}(\tau\otimes\pi'),\pi\rangle _{\so^\epsilon_{n}(\Fq)}=m\left(I^{\so^\e_{n+1-2\ell'}}_{\GGL_{\ell-\ell'}\times \so^\e_{m}}(\tau_2\otimes\pi'),\pi\right).
\end{equation}
\end{corollary}

\begin{remark}
Recall that we assume that the order $q$ of finite filed $\Fq$ is large enough such that the main theorem in \cite{S} holds. For any irreducible representation $\pi$ and $\pi'$, there is always a $\tau$ satisfying the conditions in Proposition \ref{7.21}. More precisely, $s$ and $s'$ have at most $n+m<2n$ different eigenvalues $s[i]\in \bb{F}_{q^{d_i}}$ and $s'[j]\in\bb{F}_{q^{d_j}}$. Since $\tau$ is cuspidal, the
eigenvalues of $s_0$ appear in $\bb{F}_{q^{\ell}}$. By our assumption of $q$, we have $|\Fq|\ge 2n$. So we can always pick a $s_0$ which has no common eigenvalues with $s$ and $s'$.
\end{remark}

In order to apply the theta correspondence we will work with orthogonal groups instead of special orthogonal groups. In Proposition \ref{7.21}, for $m=0$, assume that $\tau\in\cal{E}(\GGL_\ell,s_0)$ such that $\pm1$ are not eigenvalues of $s_0$.
Let $\bf{1}$ denotes the trivial representation of trivial group. By Proposition \ref{regular}, we set
\[
m(\pi,{\bf{1}}):=\langle I^{\o^\e_{n+1}}_{P}(\tau),\pi\rangle _{\o^\epsilon_{n}(\Fq)}=\left\{
\begin{array}{ll}
1,&\textrm{ if }\pi \textrm{ is regular };\\
0,&\textrm{ otherwise}.
\end{array}
\right.
\]
By the standard arguments of theta correspondence and see-saw dual pairs, we set
\[
m_\psi(\pi,{\bf{1}}):=\langle  \pi\otimes \omega_n^{\ee}, I_{P}^{\sp_{2n}}(\tau)\rangle _{\sp_{2n}(\Fq)}=\left\{
\begin{array}{ll}
1,&\textrm{ if }\pi \textrm{ is regular };\\
0,&\textrm{ otherwise'}
\end{array}
\right.
\]

To prove Theorem \ref{8.1} and Theorem \ref{8.2}, by Proposition \ref{so2} and Proposition \ref{7.21}, it suffices to calculate RHS of (\ref{sp}) and (\ref{7.22}).
\subsection{Reformulation}
We now prove the Fourier-Jacobi case. It is not hard to see that Theorem \ref{8.1} readily follows from Theorem \ref{8.6} below.

\begin{theorem}\label{8.6}
Let $s$ be a semisimple element of $\sp_{2n}(\Fq)^*=\so_{2n+1}\fq$, and $s'$ be a semisimple element of $\sp_{2m}(\Fq)^*=\so_{2m+1}\fq$. Let $\pkh\in\cal{E}(\sp_{2n},s)$ be an irreducible representation of $\sp_{2n}(\Fq)$, and let $\pkhp\in\cal{E}(\sp_{2m},s')$ be an irreducible representation of $\sp_{2m}\fq$. Assume that $n\ge m$, and let $\ell=n-m$. Let $P$ be an $F$-stable maximal parabolic subgroup of $\sp_{2n}$ with Levi factor $\GGL_{\ell} \times \sp_{2m}$. Let $s_0$ be a semisimple element of $\GGL_{\ell}\fq$ and let $\tau\in\cal{E}(\GGL_{\ell},s_0)$ be an irreducible cuspidal representation of $\GGL_{\ell}\fq$ which is nontrivial if $\ell=1$. Assume that $s_0$ has no common eigenvalues with $s$ and $s'$. Then we have
\[
\langle  \pkh\otimes \omega_n^{\epsilon_0}, I_{P}^{\sp_{2n}}(\tau\otimes\pkhp)\rangle _{\sp_{2n}(\Fq)}=\left\{
\begin{array}{ll}
m_\psi(\pi_{\rho} ,\pi_{\rho'})  &\textrm{ if }(\pkh,\pkhp)\textrm{ is $\epsilon_0$-strongly relevant};\\
0&\textrm{ otherwise}
\end{array}\right.
\]
where $m_\psi(\pw_{\rho} ,\pw_{\rho'})$ does not depend on $\psi.$

\end{theorem}
We now turn to prove $m_\psi(\pw_{\rho} ,\pw_{\rho'})$ does not depend on $\psi.$
\begin{proposition}\label{rho}
Keep the assumptions in Theorem \ref{8.6}. Let $k=h=k'=k'=0$. Then $m_\psi(\pi_{\rho} ,\pi_{\rho'})=m_{\psi'}(\pi_{\rho} ,\pi_{\rho'})$ where Let $\psi'$ be another nontrivial additive character of $\Fq$ not in the square class of $\psi$. Moreover, if $n=m$, then $m_\psi(\pi_{\rho} ,\pi_{\rho'})=m_\psi(\pi_{\rho'} ,\pi_{\rho})$.
\end{proposition}
\begin{proof}
Assume $\pi_\rho$ and $\pi_{\rho'}$ are irreducible representations of $\sp_{2n}\fq$ and $\sp_{2m}\fq$, respectively. By Proposition \ref{so2}, we only need to prove
\[
\langle  \pi_{\rho}\otimes \omega_{n,\psi}^{\ee}, I_{P}^{\sp_{2n}}(\tau\otimes\pi_{\rho'})\rangle _{\sp_{2n}(\Fq)}=\langle  \pi_{\rho}\otimes \omega_{n,\psi'}^{\ee}, I_{P}^{\sp_{2n}}(\tau\otimes\pi_{\rho'})\rangle _{\sp_{2n}(\Fq)}
\]
Consider the see-saw diagram
\[
\setlength{\unitlength}{0.8cm}
\begin{picture}(20,5)
\thicklines
\put(6.5,4){$\sp_{2n}\times \sp_{2n}$}
\put(7.3,1){$\sp_{2n}$}
\put(12.3,4){$\o^{\ee}_{2n+1}$}
\put(11.6,1){$\o^{+}_{2n}\times \o^{\ee}_1$}
\put(7.7,1.5){\line(0,1){2.1}}
\put(12.8,1.5){\line(0,1){2.1}}
\put(8,1.5){\line(2,1){4.2}}
\put(8,3.7){\line(2,-1){4.2}}
\end{picture}
\]

By Theorem \ref{p1}, we have
\[
\Theta^{+}_{n,n,\psi}(\pi_\rho)=\Theta^{+}_{n,n,\psi'}(\pi_\rho)=\pi_{\rho}'
\]
where $\pi_{\rho}'\in\cal{E}(\o^{+}_{2n})$.
By Theorem \ref{p2} and \cite[p.14]{LW1} we have
\[
\Theta^{\ee}_{n,n,\psi}(\pi_{\rho'})=\pi_{\rho',\epsilon_1}'\textrm{ and }
\Theta^{\ee}_{n,n,\psi'}(\pi_{\rho'})=\pi_{\rho',\epsilon_2}'
\]
where $\pi_{\rho',\epsilon_1}'$ and $\pi_{\rho',\epsilon_2}'\in\cal{E}(\o^{\ee}_{2n+1})$. Since $\pi_{\rho',\epsilon_1}'(-I)=\pi_{\rho'}(-I)=\pi_{\rho',\epsilon_2}'(-I)$, we have $\epsilon_1=\epsilon_2$.

On the other hand,
\[
\Theta^{+}_{n,n,\psi}(\pi_\rho')=\Theta^{+}_{n,n,\psi'}(\pi_\rho')=\pi_{\rho}
\]
and
\[
\Theta^{\ee}_{n,n,\psi}(\pi_{\rho',\epsilon_1}')=\Theta^{\ee}_{n,n,\psi'}(\pi_{\rho',\epsilon_2}')=\pi_{\rho'}.
\]
Then by Proposition \ref{w2}, we have
\begin{equation}\label{cusnv1}
\begin{aligned}
&\langle  \pi_{\rho}\otimes \omega_{n,\psi}^{\ee}, I_{P}^{\sp_{2n}}(\tau\otimes\pi_{\rho'})\rangle _{\sp_{2n}(\Fq)}\\
=&\langle  \Theta^{+}_{n,n,\psi}(\pi_\rho')\otimes \omega^{\ee}_{n,\psi}, I^{\sp_{2n}}_P(\tau\otimes\pi_{\rho'})\rangle _{\sp_{2n}(\Fq)}\\
=&\langle  \pi_\rho', \Theta_{n,n}^{\ee} (I^{\sp_{2n}}_P(\tau\otimes\pi_{\rho'}))\rangle_{\o^\epsilon_{2n}(\Fq)}\\
=&\langle  \pi_\rho',  I^{\o^{\ee}_{2n+1}}_{P'}(\tau\otimes\pi_{\rho',\epsilon_1}')\rangle_{\o^\epsilon_{2n}(\Fq)}\\
\end{aligned}
\end{equation}
Similarly, we have
\[
\langle  \pi_{\rho}\otimes \omega_{n,\psi'}^{\ee}, I_{P}^{\sp_{2n}}(\tau\otimes\pi_{\rho'})\rangle _{\sp_{2n}(\Fq)}=\langle  \pi_\rho',  I_{P'}^{\o^{\ee}_{2n+1}}(\tau\otimes\pi_{\rho',\epsilon_1}')\rangle_{\o^\epsilon_{2n}(\Fq)},
\]
which implies
\[
\langle  \pi_{\rho}\otimes \omega_{n,\psi'}^{\ee}, I_{P}^{\sp_{2n}}(\tau\otimes\pi_{\rho'})\rangle _{\sp_{2n}(\Fq)}=\langle  \pi_{\rho}\otimes \omega_{n,\psi}^{\ee}, I_{P}^{\sp_{2n}}(\tau\otimes\pi_{\rho'})\rangle _{\sp_{2n}(\Fq)}.
\]
Suppose that $n=m$. Then
\[
m_\psi(\pi_{\rho},\pi_{\rho'})=\langle  \pi_{\rho}\otimes \omega_{n,\psi}^{\ee}, \pi_{\rho'}\rangle _{\sp_{2n}(\Fq)}=\langle  \pi_{\rho}, \pi_{\rho'}\otimes \overline{\omega_{n,\psi}^{\ee}}\rangle _{\sp_{2n}(\Fq)}=\langle  \pi_{\rho}, \pi_{\rho'}\otimes \omega_{n,\psi}^{+}\rangle _{\sp_{2n}(\Fq)}
\]
Recall that $\omega_{n,\psi}^{\epsilon}=\omega_{n,\psi'}^{-\epsilon}$. Hence
\[
\begin{aligned}
\langle  \pi_{\rho}, \pi_{\rho'}\otimes \omega_{n,\psi}^{+}\rangle _{\sp_{2n}(\Fq)}
=&\left\{
\begin{array}{ll}
\langle \pi_{\rho}, \pi_{\rho'}\otimes \omega_{n,\psi}^{\ee}\rangle _{\sp_{2n}(\Fq)},&\textrm{if $\ee=+$ ;}\\
\langle \pi_{\rho}, \pi_{\rho'}\otimes \omega_{n,\psi'}^{\ee}\rangle _{\sp_{2n}(\Fq)},&\textrm{if $\ee=-$ ;}
\end{array}\right.\\
=&\left\{
\begin{array}{ll}
m_\psi(\pi_{\rho'},\pi_{\rho}),&\textrm{if $\ee=+$ ;}\\
m_{\psi'}(\pi_{\rho'},\pi_{\rho}),&\textrm{if $\ee=-$ ;}
\end{array}\right.\\
=&m_\psi(\pi_{\rho'},\pi_{\rho}).
\end{aligned}
\]
\end{proof}
The rest of this section is devoted to the proof of Theorem \ref{8.6}, which will be divided into two parts.

\subsection{Vanishing result}
As a first step towards the proof, we establish the cases where the multiplicity in Theorem \ref{8.6} vanishes.

\begin{proposition}\label{cusvan}
Keep the assumptions in Theorem \ref{8.6}. Assume that $n\ge m$. If $(\pkh,\pkhp)$ is not $\epsilon_0$-strongly relevant, then we have
\[
\langle  \pkh\otimes\omega_n^{\epsilon_0}, I_{P}^{\sp_{2n}}(\tau\otimes\pkhp)\rangle _{\sp_{2n}(\Fq)}=0.
\]

\end{proposition}
\begin{proof}
It follows immediately from \cite[Proposition 5.6]{LW3}, Proposition \ref{cus1} and the standard arguments of theta correspondence and
see-saw dual pairs.
\end{proof}

\subsection{Non-vanishing result}
To finish the proof of  Theorem \ref{8.6}, it remains to prove the following result.

\begin{proposition}\label{cusnv}
Keep the assumptions in Theorem \ref{8.6}. Assume that $n\ge m$. If $(\pkh,\pkhp)$ is $\epsilon_0$-strongly relevant, then we have
\[
\langle  \pkh\otimes \omega_n^{\epsilon_0}, I_{P}^{\sp_{2n}}(\tau\otimes\pkhp)\rangle _{\sp_{2n}(\Fq)}=m_\psi(\pi_{\rho} ,\pi_{\rho'}).
\]

\end{proposition}
\begin{proof}
 It is trivial if $k=k'=h=h'=0$.
We prove the proposition by induction on $|k|+|k'|+|h|+|h'|$. Assume that this proposition holds for $|k|+|k'|+|h|+|h'|<N$. We only prove on $|k|+|k'|+|h|+|h'|=N$ for $\epsilon_0=\ee$. The proof of $\epsilon_0=-\ee$ is similar and will be left to the reader. To ease notations we suppress various Levi subgroups from the parabolic induction in the sequel, which should be clear from the context.

Since $(\pkh,\pkhp)$ is $\ee$-strongly relevant, by Corollary \ref{strongly relevant}, we have $k=|h'|$ or $k=|h'|-1$ and $k'=|h|$ or $k'=|h|-1$, which implies if $|k|+|k'|+|h|+|h'|>0$, then $|k|+|k'|>0$.
So there are two cases as follows:
\begin{itemize}
\item[]
Case (A): $k>0$
\item[]
Case (B): $k=0$. In this case, we have $k'>0$.
\end{itemize}
We now prove the Case (A).

(1) Suppose that $k=|h'|$.

 Put $n_1=n-k$. By Proposition \ref{cus1} (i), we can pick $\epsilon\in\{\pm\}$ such that the first occurrence index of $\pkh$ in the Witt tower ${\bf O}^\epsilon_\rm{even}$ is $n_1$ and $\Theta_{n,n_1}^\epsilon(\pkh)=\pi_{\rho_1,k_1,h_1}$ is an irreducible representation of $\o^\epsilon_{2n_1}\fq$ with $k_1\in\{\pm k\}$ and $h_1\in\{\pm h\}$, and $\Theta^\epsilon_{n_1,n}(\pi_{\rho_1,k_1,h_1})=\pkh$. Since $(\pkh,\pkhp)$ is $\ee$-strongly relevant and $k=|h'|$, the first occurrence index of $\pkhp$ in the Witt tower ${\bf O}^{\ee\cdot\epsilon}_\rm{odd}$ is $m-k$ and by Proposition \ref{cus1} (ii), $\Theta^{\ee\cdot\epsilon}_{m,m-k}(\pkhp)=\pi_{\rho'_1,k_1',h_1',\e}$ with $k_1'=|h'|-1=k-1=|k_1|-1$ and $h_1'=k'$.

Consider the see-saw diagram
\[
\setlength{\unitlength}{0.8cm}
\begin{picture}(20,5)
\thicklines
\put(6.5,4){$\sp_{2n}\times \sp_{2n}$}
\put(7.3,1){$\sp_{2n}$}
\put(12.3,4){$\o^{\ee\cdot\epsilon}_{2n_1+1}$}
\put(11.6,1){$\o^{\epsilon}_{2n_1}\times \o^{\ee}_1$}
\put(7.7,1.5){\line(0,1){2.1}}
\put(12.8,1.5){\line(0,1){2.1}}
\put(8,1.5){\line(2,1){4.2}}
\put(8,3.7){\line(2,-1){4.2}}
\end{picture}
\]
By Proposition \ref{w2} and Proposition \ref{cus1}, one has,
\begin{equation}\label{cusnv1}
\begin{aligned}
&\langle  \pkh\otimes \omega^{\ee}_n, I^{\sp_{2n}}(\tau\otimes\pkhp)\rangle _{\sp_{2n}(\Fq)}\\
=&\langle  \Theta_{n_1,n}^{\epsilon}(\pi_{\rho_1,k_1,h_1})\otimes \omega^{\ee}_n, I^{\sp_{2n}}(\tau\otimes\pkhp)\rangle _{\sp_{2n}(\Fq)}\\
=&\langle  \pi_{\rho_1,k_1,h_1}, \Theta_{n,n_1}^{\ee\cdot\epsilon} (I^{\sp_{2n}}(\tau\otimes\pkhp))\rangle_{\o^\epsilon_{2n_1}(\Fq)}\\
=&\langle  \pi_{\rho_1,k_1,h_1},  I^{\o^{\ee\cdot\epsilon}_{2n_1+1}}((\chi\otimes\tau)\otimes\pi_{\rho'_1,k_1',h_1',\e})\rangle _{\o^\epsilon_{2n_1}(\Fq)}\\
\end{aligned}
\end{equation}

$\bullet$ Suppose that $n>m$. By Corollary \ref{o4}, let $\ell'=1$, one has
\[
\begin{aligned}
&\langle  \pi_{\rho_1,k_1,h_1},  I^{\o^{\ee\cdot\epsilon}_{2n_1+1}}((\chi\otimes\tau)\otimes\pi_{\rho'_1,k_1',h_1',\e})\rangle _{\o^\epsilon_{2n_1}(\Fq)}\\
=&m_\psi(\pi_{\rho_1,k_1,h_1},I^{\o^{\ee\cdot\epsilon}_{2n_1-1}}(\tau'\otimes\pi_{\rho'_1,k_1',h_1',\e}))\\
=&\langle  \pi_{\rho_1,k_1,h_1},  I^{\o^{\ee\cdot\epsilon}_{2n_1-1}}(\tau'\otimes\pi_{\rho'_1,k_1',h_1',\e})\rangle _{\o^{\ee\cdot\epsilon}_{2n_1-1}(\Fq)}.
\end{aligned}
\]
where $\tau'$ is $\tau_2$ in Corollary \ref{o4}.

Let $n_2=n_1-1-|k_1'|=n_1-|k_1|$. Then by Proposition \ref{cus2} (i), we have $\Theta^\epsilon_{n_1,n_2}(\sgn \pi_{\rho_1,k_1,h_1})= \pi_{\rho,k_2,h_2}$
 is an irreducible representation of $\sp_{2n_2}\fq$ with $k_2=k-1$ and $h_2\in\{\pm h\}$ and $\Theta^\epsilon_{n_2,n_1}( \pi_{\rho,k_2,h_2})= \sgn \pi_{\rho_1,k_1,h_1}$.
Consider the see-saw diagram
\[
\setlength{\unitlength}{0.8cm}
\begin{picture}(20,5)
\thicklines
\put(6.3,4){$\sp_{2n_2}\times \sp_{2n_2}$}
\put(7.2,1){$\sp_{2n_2}$}
\put(12.3,4){$\o^\epsilon_{2n_1}$}
\put(11.6,1){$\o^{\ee\cdot\epsilon}_{2n_1-1}\times \o^{+}_1$}
\put(7.7,1.5){\line(0,1){2.1}}
\put(12.8,1.5){\line(0,1){2.1}}
\put(8,1.5){\line(2,1){4.2}}
\put(8,3.7){\line(2,-1){4.2}}
\end{picture}
\]
By Proposition \ref{w2} and Proposition \ref{cus2}, one has,
\[
\begin{aligned}
&\langle  \pi_{\rho_1,k_1,h_1},  I^{\o^{\ee\cdot\epsilon}_{2n_1-1}}(\tau'\otimes\pi_{\rho'_1,k_1',h_1',\e})\rangle _{\o^{\ee\cdot\epsilon}_{2n_1-1}(\Fq)}\\
=&\langle \sgn \pi_{\rho_1,k_1,h_1}, \sgn I^{\o^{\ee\cdot\epsilon}_{2n_1-1}}(\tau'\otimes\pi_{\rho'_1,k_1',h_1',\e})\rangle _{\o^{\ee\cdot\epsilon}_{2n_1-1}(\Fq)}\\
=&\langle \sgn \pi_{\rho_1,k_1,h_1},  I^{\o^{\ee\cdot\epsilon}_{2n_1-1}}(\tau'\otimes(\sgn\pi_{\rho'_1,k_1',h_1',\e}))\rangle _{\o^{\ee\cdot\epsilon}_{2n_1-1}(\Fq)}\\
=&\langle  \Theta_{n_2,n_1}^\epsilon(\pi_{\rho,k_2,h_2}),  I^{\o^{\ee\cdot\epsilon}_{2n_1-1}}(\tau'\otimes(\sgn\pi_{\rho'_1,k_1',h_1',\e}))\rangle _{\o^{\ee\cdot\epsilon}_{2n_1-1}(\Fq)}\\
=&\langle  \pi_{\rho,k_2,h_2},  \Theta^{\ee\cdot\epsilon}_{n_1-1,n_2}(I^{\o^{\ee\cdot\epsilon}_{2n_1-1}}(\tau'\otimes(\sgn\pi_{\rho'_1,k_1',h_1',\e})))\otimes \omega^+_{n_2}\rangle _{\sp_{2n_2}(\Fq)}\\
=&\langle  \pi_{\rho,k_2,h_2},  I^{\sp_{2n_2}}((\chi\otimes\tau')\otimes\pi_{\rho',k_2',h_2'})\otimes \omega^+_{n_2}\rangle _{\sp_{2n_2}(\Fq)}\\
=&\langle  \pi_{\rho,k_2,h_2}\otimes \omega^{\ee}_{n_2},  I^{\sp_{2n_2}}((\chi\otimes\tau')\otimes\pi_{\rho',k_2',h_2'})\rangle _{\sp_{2n_2}(\Fq)}\\
\end{aligned}
\]
where $k_2'=|h_1'|=k'$ and $h_2'\in\{\pm k_1'\}=\{|h'|-1,-|h'|+1\}$.

$\bullet$ Suppose that $n=m$. Then $\ell=0$, and $\tau$ dose not appear. Similarly, by Corollary \ref{o4}, one has
\[
\begin{aligned}
&\langle  \pi_{\rho_1,k_1,h_1}, \pi_{\rho'_1,k_1',h_1',\e}\rangle _{\o^\epsilon_{2n_1}(\Fq)}
=\langle  I^{\o^\epsilon_{2(n_1+1)}}(\tau'\otimes\pi_{\rho_1,k_1,h_1}), \pi_{\rho'_1,k_1',h_1',\e}\rangle _{\o^{\ee\cdot\epsilon}_{2n_1+1}(\Fq)}.
\end{aligned}
\]
where $\tau'$ is $\tau_2$ in Corollary \ref{o4}.

Let $n_2=n_1-|k_1'|=n_1-|k_1|+1$. Then by Proposition \ref{w2} and Proposition \ref{cus2} (ii), we have
\[
\Theta_{n_1,n_2}^\epsilon( I^{\o^\epsilon_{2(n_1+1)}}(\tau'\otimes(\sgn\pi_{\rho_1,k_1,h_1}))=  I^{\sp_{2n_2}}(\tau'\otimes\pi_{\rho,k_2,h_2})
\]
with $k_2=k-1$ and $h_2\in\{\pm h\}$.
Consider the see-saw diagram
\[
\setlength{\unitlength}{0.8cm}
\begin{picture}(20,5)
\thicklines
\put(6.3,4){$\sp_{2n_2}\times \sp_{2n_2}$}
\put(7.2,1){$\sp_{2n_2}$}
\put(12.1,4){$\o^\epsilon_{2(n_1+1)}$}
\put(11.5,1){$\o^{\ee\cdot\epsilon}_{2n_1+1}\times \o^{+}_1$}
\put(7.7,1.5){\line(0,1){2.1}}
\put(12.8,1.5){\line(0,1){2.1}}
\put(8,1.5){\line(2,1){4.2}}
\put(8,3.7){\line(2,-1){4.2}}
\end{picture}
\]
By Proposition \ref{w2} and Proposition \ref{cus2}, one has,
\[
\begin{aligned}
&\langle  I^{\o^\epsilon_{2(n_1+1)}}(\tau'\otimes\pi_{\rho_1,k_1,h_1}), \pi_{\rho'_1,k_1',h_1',\e}\rangle _{\o^{\ee\cdot\epsilon}_{2n_1+1}(\Fq)}\\
=&\langle  \sgn I^{\o^\epsilon_{2(n_1+1)}}(\tau'\otimes\pi_{\rho_1,k_1,h_1}), \sgn \pi_{\rho'_1,k_1',h_1',\e}\rangle _{\o^{\ee\cdot\epsilon}_{2n_1+1}(\Fq)}\\
=&\langle  I^{\o^\epsilon_{2(n_1+1)}}(\tau'\otimes(\sgn\pi_{\rho_1,k_1,h_1})),\sgn \pi_{\rho'_1,k_1',h_1',\e}\rangle _{\o^{\ee\cdot\epsilon}_{2n_1+1}(\Fq)}\\
=&\langle \Theta^\epsilon_{n_2,n_1+1}(I^{\sp_{2n_2}}(\tau'\otimes\pi_{\rho,k_2,h_2})),\sgn \pi_{\rho'_1,k_1',h_1',\e}\rangle _{\o^{\ee\cdot\epsilon}_{2n_1+1}(\Fq)}\\
=&\langle  I^{\sp_{2n_2}}(\tau'\otimes\pi_{\rho,k_2,h_2}),  \Theta^\epsilon_{n_1,n_2}(\sgn \pi_{\rho'_1,k_1',h_1',\e})\otimes \omega^+_{n_2}\rangle _{\sp_{2n_2}(\Fq)}\\
=&\langle  I^{\sp_{2n_2}}(\tau'\otimes\pi_{\rho,k_2,h_2}),  \pi_{\rho',k_2',h_2'}\otimes \omega^+_{n_2}\rangle _{\sp_{2n_2}(\Fq)}\\
\end{aligned}
\]
where $k_2'=|h_1'|=k'$ and $h_2'\in\{\pm k_1'\}=\{|h'|-1,-|h'|+1\}$.

Summarizing, by induction hypothesis, we have
\[
\begin{aligned}
&\langle  \pkh\otimes \omega_n^{\ee}, I^{\sp_{2n}}(\tau\otimes\pkhp)\rangle _{\sp_{2n}(\Fq)}\\
=&\left\{
\begin{array}{ll}
\langle  \pi_{\rho,k_2,h_2}\otimes\omega^{\ee}_{n_2},  I^{\sp_{2n_2}}(\tau'\otimes\pi_{\rho',k_2',h_2'})\rangle _{\sp_{2n_2}(\Fq)},&\textrm{ if }n>m;\\
\langle  I^{\sp_{2n_2}}(\tau'\otimes\pi_{\rho,k_2,h_2}),  \pi_{\rho',k_2',h_2'}\otimes \omega^+_{n_2}\rangle _{\sp_{2n_2}(\Fq)},&\textrm{ if }n=m.\\
\end{array}\right.\\
=&\left\{
\begin{array}{ll}
m_\psi(\pw_{\rho} ,\pw_{\rho'}),&\textrm{ if $n>m$ and }(\pi_{\rho,k_2,h_2},\pi_{\rho',k_2',h_2'})\textrm{ is $\ee$-strongly relevant};\\
m_{\psi}(\pw_{\rho} ,\pw_{\rho'}),&\textrm{ if $n=m$ and }(\pi_{\rho',k_2',h_2'},\pi_{\rho,k_2,h_2})\textrm{ is $+$-strongly relevant};\\
0,&\textrm{ otherwise. }\\
\end{array}\right.
\end{aligned}
\]

Note that $(\pi_{\rho,k_2,h_2},\pi_{\rho',k_2',h_2'})\textrm{ is $\ee$-strongly relevant}$ if and only if $(\pi_{\rho',k_2',h_2'},\pi_{\rho,k_2,h_2})$ is $+$-strongly relevant. If $(\pi_{\rho,k_2,h_2},\pi_{\rho',k_2',h_2'})$ is $\ee$-strongly relevant or $m_\psi(\pi_{\rho} ,\pi_{\rho'})=0$, then the Proposition holds. Assume that $m_\psi(\pi_{\rho} ,\pi_{\rho'})\ne0$. It remains to prove that $(\pi_{\rho,k_2,h_2},\pi_{\rho',k_2',h_2'})$ must be $\ee$-strongly relevant.

Assume that $(\pi_{\rho,k_2,h_2},\pi_{\rho',k_2',h_2'})$ is not $\ee$-strongly relevant. Let $n_2^{\epsilon}$ and $m_2^{\ee\cdot\epsilon}$ be the first occurrence index of $\pi_{\rho,k_2,h_2}$ and $\pi_{\rho',k_2',h_2'}$ in the Witt tower ${\bf O}^{\epsilon}_\rm{even}$ and ${\bf O}^{\ee\cdot\epsilon}_\rm{odd}$, respectively. Recall that $\pi_{\rho,k_2,h_2}$ and $\pi_{\rho',k_2',h_2'}$ are irreducible representations of $\sp_{2(n-2k)}\fq=\sp_{2n_2^*}\fq$ and $\sp_{2(m-2|h'|+1)}\fq=\sp_{2m_2^*}\fq$, respectively.
By above see-saw argument, we have
 \[
 n_2^*\le n_2^{\epsilon}=n-k\textrm{ and } m_2^*\le m_2^{\ee\cdot\epsilon}=m-|h'|.
 \]
 By Proposition \ref{cus1}, we have
 \[
 n_2^*\ge n_2^{-\epsilon}=n-2k-k_2\textrm{ and } m_2^*\ge m_2^{-\ee\cdot\epsilon}=m-2|h'|+1-|h_2'|,
 \]
 where $n_2^{-\epsilon}$ and $m_2^{-\ee\cdot\epsilon}$ be the first occurrence index of $\pi_{\rho,k_2,h_2}$ and $\pi_{\rho',k_2',h_2'}$ in the Witt tower ${\bf O}^{-\epsilon}_\rm{even}$ and ${\bf O}^{-\ee\cdot\epsilon}_\rm{odd}$, respectively.
 Recall that $k_2=k-1$, $|h'|=k$ and $|h_2'|=|h'|-1=k-1$, which implies that
 \[
 n_2^*-n_2^{-\epsilon}=m_2^*- m_2^{-\ee\cdot\epsilon}=k-1.
 \]
 Then $(\pi_{\rho,k_2,h_2},\pi_{\rho',k_2',h_2'})$ is $\ee$-relevant. By our assumption, $(\pi_{\rho,k_2,h_2},\pi_{\rho',k_2',h_2'})$ is not $\ee$-strongly relevant, which implies $(\pi_{\rho',k_2',h_2'},\pi_{\rho,k_2,h_2})$ is not $+$-relevant. Note that by Corollary \ref{strongly relevant} (iii), $(\pi_{\rho',k_2',h_2'},\pi_{\rho,k_2,-h_2})$ is $+$-relevant, and by Proposition \ref{cus1} (i),  $(\pi_{\rho,k_2,-h_2},\pi_{\rho',k_2',h_2'})$ is $\ee$-relevant.
 Therefore $(\pi_{\rho,k_2,-h_2},\pi_{\rho',k_2',h_2'})$ is $\ee$-strongly relevant, and by induction on $|k|+|k'|+|h|+|h'|$, we have
\[
\left\{
\begin{array}{ll}
\langle  \pi_{\rho,k_2,-h_2}\otimes \omega^{\ee}_{n_2},  I^{\sp_{2n_2}}(\tau'\otimes\pi_{\rho',k_2',h_2'})\rangle _{\sp_{2n_2}(\Fq)}=m_\psi(\pw_{\rho} ,\pw_{\rho'})\ne 0,&\textrm{ if }n>m;\\
\langle  I^{\sp_{2n_2}}(\tau'\otimes\pi_{\rho,k_2,-h_2}),  \pi_{\rho',k_2',h_2'}\otimes \omega^+_{n_2}\rangle _{\sp_{2n_2}(\Fq)}=m_{\psi}(\pw_{\rho} ,\pw_{\rho'})\ne 0,&\textrm{ if }n=m.\\
\end{array}\right.
\]

  By Proposition \ref{cus1} (i), the first occurrence index of $\pi_{\rho,k,-h}$ in the Witt tower ${\bf O}^\epsilon_\rm{even}$ is also $n_1$. Then $\Theta_{n,n_1}^\epsilon(\pi_{\rho,k,-h})=\pi_{\rho_1,k_1^*,h_1^*}$ is an irreducible representation of $\o^\epsilon_{2n_1}\fq$ where $k_1^*\in\{\pm k\}$ and $h_1^*\in\{\pm h\}$ and $\pi_{\rho_1,k_1^*,h_1^*}\ne \pi_{\rho_1,k_1,h_1}$. By Proposition \ref{cus2} (i), the first occurrence index of $\sgn\pi_{\rho_1,k_1^*,h_1^*}$ (resp. $I^{\o^\epsilon_{2(n_1+1)}}(\tau'\otimes(\sgn\pi_{\rho_1,k_1,h_1})$) is also $n_2$, and $\Theta_{n_1,n_2}^\epsilon(\pi_{\rho_1,k_1^*,h_1^*})=\pi_{\rho,k_2^*,h_2^*}$ (resp. $\Theta_{n_1,n_2}^\epsilon( I^{\o^\epsilon_{2(n_1+1)}}(\tau'\otimes(\sgn\pi_{\rho_1,k_1^*,h_1^*}))=  I^{\sp_{2n_2}}(\tau'\otimes\pi_{\rho,k_2^*,h_2^*})$) is an irreducible representation of $\sp_{2n_2}\fq$ where $k_2^*=k-1=k_2$ and $h_2^*\in\{\pm h\}$. Note that $\pi_{\rho,k_2,h_2}\notin \Theta_{n_1,n_2}^\epsilon(\pi_{\rho_1,k_1^*,h_1^*})$, which implies $\pi_{\rho,k_2^*,h_2^*}=\pi_{\rho,k_2,-h_2}$.
With same see-saw argument above, we have
 \begin{equation}\label{notstrongly relevant}
 \begin{aligned}
 &\langle  \pi_{\rho,k,-h}\otimes \omega^{\ee}_n, I^{\sp_{2n}}(\tau\otimes\pkhp)\rangle _{\sp_{2n}(\Fq)}\\
=&\langle  \pi_{\rho_1,k_1^*,h_1^*},  I^{\o^{\ee\cdot\epsilon}_{2n_1+1}}(\tau\otimes\pi_{\rho'_1,k_1',h_1'})\rangle _{\o^\epsilon_{2n_1}(\Fq)}\\
=&\left\{
\begin{array}{ll}
\langle  \pi_{\rho,k_2,-h_2}\otimes \omega_{n_2}^{\ee},  I^{\sp_{2n_2}}(\tau'\otimes\pi_{\rho',k_2',h_2'})\rangle _{\sp_{2n_2}(\Fq)},&\textrm{ if }n>m;\\
\langle  I^{\sp_{2n_2}}(\tau'\otimes\pi_{\rho,k_2,-h_2}),  \pi_{\rho',k_2',h_2'}\otimes \omega_{n_2}^+\rangle _{\sp_{2n_2}(\Fq)},&\textrm{ if }n=m.\\
\end{array}\right.\\
=&m_\psi(\pw_{\rho} ,\pw_{\rho'})\\
\ne&0.
\end{aligned}
 \end{equation}
 Since $(\pi_{\rho,k,h},\pkhp)$ is $\ee$-strongly relevant, by Corollary \ref{strongly relevant} (ii), $(\pi_{\rho,k,-h},\pkhp)$ is not $\ee$-strongly relevant, which contradicts with (\ref{notstrongly relevant}) by Proposition \ref{cusvan}.

 (2) Suppose $k=|h'|-1$. One has
\[
\langle  \pkh\otimes \omega_n^{\ee}, I^{\sp_{2n}}(\tau\otimes\pkhp)\rangle _{\sp_{2n}(\Fq)}=\langle  \pkh, I^{\sp_{2n}}(\tau\otimes\pkhp)\otimes \omega_n^+\rangle _{\sp_{2n}(\Fq)}
\]
Pick $n_1$ and $\epsilon$ as before.
Consider the see-saw diagram
\[
\setlength{\unitlength}{0.8cm}
\begin{picture}(20,5)
\thicklines
\put(6.5,4){$\sp_{2n}\times \sp_{2n}$}
\put(7.3,1){$\sp_{2n}$}
\put(12.3,4){$\o^\epsilon_{2n_1}$}
\put(11.6,1){$\o^{\ee\cdot\epsilon}_{2n_1-1}\times \o^{+}_1$}
\put(7.7,1.5){\line(0,1){2.1}}
\put(12.8,1.5){\line(0,1){2.1}}
\put(8,1.5){\line(2,1){4.2}}
\put(8,3.7){\line(2,-1){4.2}}
\end{picture}
\]
Using the same see-saw arguments, we have
\[
\begin{aligned}
\langle  \pkh, I^{\sp_{2n}}(\tau\otimes\pkhp)\otimes \omega_n^+\rangle _{\sp_{2n}(\Fq)}
=&\langle  \pi_{\rho_1,k_1,h_1},  I^{\o^{\ee\cdot\epsilon}_{2n_1-1}}((\chi\otimes\tau)\otimes\pi_{\rho'_1,k_1',h_1',\e})\rangle _{\o^{\ee\cdot\epsilon}_{2n_1-1}(\Fq)}\\
\end{aligned}
\]
where $k_1\in\{\pm k\}$, $h_1\in\{\pm h\}$, $k_1'=|h'|-1=k$ and $h_1'= k'$. As before,
\[
\begin{aligned}
\langle  \pi_{\rho_1,k_1,h_1},  I^{\o^{\ee\cdot\epsilon}_{2n_1-1}}((\chi\otimes\tau)\otimes\pi_{\rho'_1,k_1',h_1',\e})\rangle _{\o^{\ee\cdot\epsilon}_{2n_1-1}(\Fq)}
=\langle  \pi_{\rho_1,k_1,h_1},  I^{\o^{\ee\cdot\epsilon}_{2n_1+1}}(\tau'\otimes\pi_{\rho'_1,k_1',h_1',\e})\rangle _{\o^{\ee\cdot\epsilon}_{2n_1-1}(\Fq)}.
\end{aligned}
\]
Let $n_2=n_1-(|k_1|-1)$. Consider see-saw diagram
\[
\setlength{\unitlength}{0.8cm}
\begin{picture}(20,5)
\thicklines
\put(6.3,4){$\sp_{2n_2}\times \sp_{2n_2}$}
\put(7.2,1){$\sp_{2n_2}$}
\put(12.1,4){$\o^\epsilon_{2(n_1+1)}$}
\put(11.5,1){$\o^{\ee\cdot\epsilon}_{2n_1+1}\times \o^{+}_1$}
\put(7.7,1.5){\line(0,1){2.1}}
\put(12.8,1.5){\line(0,1){2.1}}
\put(8,1.5){\line(2,1){4.2}}
\put(8,3.7){\line(2,-1){4.2}}
\end{picture}
\]
Similarly, we have
\[
\langle  \pi_{\rho_1,k_1,h_1},  I^{\o^{\ee\cdot\epsilon}_{2n_1+1}}(\tau'\otimes\pi_{\rho'_1,k_1',h_1',\e})\rangle _{\o^{\ee\cdot\epsilon}_{2n_1-1}(\Fq)}
=\langle  \pi_{\rho,k_2,h_2},  I^{\sp_{n_2}}(\tau'\otimes\pi_{\rho',k_2',h_2'})\otimes \omega_{n_2}^+\rangle _{\sp_{2n_2}(\Fq)}
\]
where $k_2=k-1$, $h_2\in\{\pm h\}$, $k_2'= k'$ and $h_2'\in\{|h'|-1,-|h'|+1\}$. The rest of the proof runs as before.

We now turn to prove Case (B).
Consider the see-saw diagram like this:
\[
\setlength{\unitlength}{0.8cm}
\begin{picture}(20,5)
\thicklines
\put(6.5,4){$\sp_{2n}\times \sp_{2n}$}
\put(7.3,1){$\sp_{2n}$}
\put(12.3,4){$\o^{\epsilon}_{2n_1+1}$}
\put(11.6,1){$\o^{\epsilon}_{2n_1}\times \o^{+}_1$}
\put(7.7,1.5){\line(0,1){2.1}}
\put(12.8,1.5){\line(0,1){2.1}}
\put(8,1.5){\line(2,1){4.2}}
\put(8,3.7){\line(2,-1){4.2}}
\end{picture}
\]
 The rest of the proof runs as before.
\end{proof}

\subsection{the Bessel case}
We have established the Fourier-Jacobi case. We now prove the similar result for Bessel case.

\begin{proposition}\label{7bp}
 Let $s$ be a semisimple element of $(\rm{O}^\epsilon_{2n+1}(\Fq)^*)^0$, and $s'$ be a semisimple element of $(\rm{O}^{\epsilon'}_{2m}(\Fq)^*)^0$. Let $\pi_{\rho,k,h,\epsilon''}\in\cal{E}(\rm{O}^\epsilon_{2n+1},s)$, and $\pkhp\in\cal{E}(\rm{O}^{\epsilon'}_{2m},s')$.

 (i) Assume that $n\ge m$. Let $P$ be an $F$-stable maximal parabolic subgroup of $\rm{O}^\e_{2(n+1)}$ with Levi factor $\GGL_{n-m+1} \times \rm{O}^{\epsilon'}_{2m}$. Let $s_0$ be a semisimple element of $\GGL_{n-m+1}\fq$ and let $\tau\in\cal{E}(\GGL_{n-m+1},s_0)$ be an irreducible cuspidal representation of $\GGL_{n-m+1}\fq$ which is nontrivial  if $n-m+1=1$. Assume that $s_0$ has no common eigenvalues with $s$ and $s'$. Then we have
\[
\begin{aligned}
&\langle  \pi_{\rho,k,h,\epsilon''},I^{\o^\e_{2(n+1)}}_P(\tau\otimes\pkhp)\rangle _{\rm{O}^{\epsilon}_{2n+1}(\Fq)}\\
=&\left\{
\begin{array}{ll}
m_\psi(\pw_{\rho} ,\pw_{\rho'}),  &\textrm{ if }(\pi_{\rho,k,h,\epsilon''},\pkhp)\textrm{ is strongly relevant;}\\
0,&\textrm{ otherwise}
\end{array}\right.
\end{aligned}
\]
where $m_\psi(\pw_{\rho} ,\pw_{\rho'})$ is given in Theorem \ref{8.6} and $m_\psi(\pw_{\rho} ,\pw_{\rho'})$ does not depend on $\psi$.

 (ii) Assume that $n< m$. Let $P$ be an $F$-stable maximal parabolic subgroup of $\rm{O}^\epsilon_{2m+1}$ with Levi factor $\GGL_{m-n} \times \rm{O}^{\epsilon}_{2n+1}$. Let $s_0$ be a semisimple element of $\GGL_{m-n}\fq$ and let $\tau\in\cal{E}(\GGL_{m-n},s_0)$ be an irreducible cuspidal representation of $\GGL_{m-n}\fq$ which is nontrivial  if $m-n=1$. Assume that $s_0$ has no common eigenvalues with $s$ and $s'$. Then we have
\[
\begin{aligned}
&\langle I^{\o^\epsilon_{2m+1}}_P(\tau\otimes\pi_{\rho,k,h,\epsilon''}),\pkhp\rangle _{\rm{O}^{\e}_{2m}(\Fq)}\\
=&\left\{
\begin{array}{ll}
m(\pw_{\rho,\epsilon''} ,\pw_{\rho'}),  &\textrm{ if }(\pi_{\rho,k,h,\epsilon''},\pkhp)\textrm{ is strongly relevant;}\\
0,&\textrm{ otherwise}.
\end{array}\right.
\end{aligned}
\]
where $m_\psi(\pw_{\rho} ,\pw_{\rho'})$ is given in Theorem \ref{8.6} and $m_\psi(\pw_{\rho} ,\pw_{\rho'})$ does not depend on $\psi$.
\end{proposition}
\begin{proof}
We can get the vanishing result by \cite{LW3}. We now prove the non-vanishing result.

 We only prove (i). The proof of (ii) is similar and will be left to the reader. As before, we suppress various Levi subgroups from the parabolic induction in the sequel.

Let $n^\epsilon$ (resp. $n^\epsilon_0$, $m^\e$ and $m_0^\e$) be the first occurrence index of $\pi_{\rho,k,h,\epsilon''}$ (resp. $\rm{sgn}\otimes\pi_{\rho,k,h,\epsilon''}$, $\pkhp$ and $\sgn\pkhp$). Recall that by definition, $(\pi_{\rho,k,h,\epsilon''},\pkhp)$ is strongly relevant if and only if $(\rm{sgn}\otimes\pi_{\rho,k,h,\epsilon''},\rm{sgn}\otimes\pkhp)$ is. Note that
\[
\langle  \pi_{\rho,k,h,\epsilon''},I^{\o^\e_{2(n+1)}}_P(\tau\otimes\pkhp)\rangle _{\rm{O}^{\epsilon}_{2n+1}(\Fq)}\\
=\langle  \rm{sgn}\otimes\pi_{\rho,k,h,\epsilon''},I^{\o^\e_{2(n+1)}}_P(\tau\otimes\rm{sgn}\otimes\pkhp)\rangle _{\rm{O}^{\epsilon}_{2n+1}(\Fq)}.
\]
By Proposition \ref{cus2}, we have
\[
\left\{
\begin{array}{ll}
n^\epsilon=n+k+1;\\
n^\epsilon_0=n-k
\end{array}\right.
\textrm{ or }
\left\{
\begin{array}{ll}
n^\epsilon=n-k;\\
n^\epsilon_0=n+k+1.
\end{array}\right.
\]
Hence it suffices to prove for $n^\epsilon=n+k+1$.

Put $n_1=n^\epsilon$ and $m_1=m^\e$. By Proposition \ref{cus2}, we have $\Theta^\epsilon_{n,n_1}(\pi_{\rho,k,h,\epsilon''})=\pi_{\rho,k_1,h_1}$ with $k_1= h$ and $h_1\in \{\pm (k+1)\}$ and $\Theta^\epsilon_{n_1,n}(\pi_{\rho,k_1,h_1})=\pi_{\rho,k,h,\epsilon''}$.
Since $(\pi_{\rho,k,h,\epsilon''},\pkhp)$ is strongly relevant, by Proposition \ref{cus2} (i), $m_1-m=k$ or $m_1-m=k+1$.

$\bullet$ Suppose $m_1-m=k$. Then by Proposition \ref{cus2}, $\Theta^\e_{m,m_1}(\pkhp)=\pi_{\rho', k_1',h_1'}$ with $k_1'= |k'|=k=|h_1|-1$ and $h_1'\in \{\pm h'\}$, and $\Theta^\e_{m_1,m}(\pi_{\rho', k_1',h_1'})=\pkhp$. By Proposition \ref{w2}, we know that
\[
\Theta^\e_{n_1,n+1}(I^{\sp_{2n_1}}(\tau\otimes\pi_{\rho', k_1',h_1'}))=I_P^{\o^\e_{2(n+1)}}(\tau\otimes\Theta^\e_{m_1,m}(\pi_{\rho', k_1',h_1'}))=I^{\o^\e_{2(n+1)}}_P(\tau\otimes\pkhp).
\]
Consider the see-saw diagram:
\[
\setlength{\unitlength}{0.8cm}
\begin{picture}(20,5)
\thicklines
\put(6.5,4){$\sp_{2n_1}\times \sp_{2n_1}$}
\put(7.3,1){$\sp_{2n_1}$}
\put(12.3,4){$\o^{\e}_{2(n+1)}$}
\put(11.6,1){$\o^{\epsilon}_{2n+1}\times \o^{\ee\cdot\e\cdot\epsilon}_1$}
\put(7.7,1.5){\line(0,1){2.1}}
\put(12.8,1.5){\line(0,1){2.1}}
\put(8,1.5){\line(2,1){4.2}}
\put(8,3.7){\line(2,-1){4.2}}
\end{picture}
\]
By Theorem \ref{8.6}, one has
\[
\begin{aligned}
&\langle  \pi_{\rho,k,h,\epsilon''},I^{\o^\e_{2(n+1)}}_P(\tau\otimes\pkhp)\rangle _{\rm{O}^{\epsilon}_{2n+1}(\Fq)}\\
=&\langle  \pi_{\rho,k,h,\epsilon''},\Theta^\e_{n_1,n+1}(I^{\sp_{2n_1}}(\tau\otimes\pi_{\rho', k_1',h_1'}))\rangle _{\rm{O}^{\epsilon}_{2n+1}(\Fq)}\\
=&\langle  \Theta^\epsilon_{n,n_1}(\pi_{\rho,k,h,\epsilon''})\otimes\omega^{\ee\cdot\e\cdot\epsilon}_{n_1},I^{\sp_{2n_1}}(\tau\otimes\pi_{\rho', k_1',h_1'})\rangle _{\sp_{2n_1}(\Fq)}\\
=&\langle  \pi_{\rho,k_1,h_1}\otimes\omega^{\ee\cdot\e\cdot\epsilon}_{n_1},I^{\sp_{2n_1}}(\tau\otimes\pi_{\rho', k_1',h_1'})\rangle _{\sp_{2n_1}(\Fq)}\\
=&\left\{
\begin{array}{ll}
m_\psi(\pi_{\rho} ,\pi_{\rho'})  &\textrm{ if }(\pi_{\rho,k_1,h_1},\pi_{\rho', k_1',h_1'})\textrm{ is $\ee\cdot\e\cdot\epsilon$-strongly relevant};\\
0&\textrm{ if }(\pi_{\rho,k_1,h_1},\pi_{\rho', k_1',h_1'})\textrm{ is not $\ee\cdot\e\cdot\epsilon$-strongly relevant}.
\end{array}\right.
\end{aligned}
\]

If $m_\psi(\pi_{\rho} ,\pi_{\rho'})=0$ or $(\pi_{\rho,k_1,h_1},\pi_{\rho', k_1',h_1'})$ is a $\ee\cdot\e\cdot\epsilon$-strongly relevant pair of irreducible representations of symplectic groups, then we complete the proof.

We now suppose that $m_\psi(\pi_{\rho} ,\pi_{\rho'})\ne0$. We show that the pair $(\pi_{\rho,k_1,h_1},\pi_{\rho', k_1',h_1'})$ must be $\ee\cdot\e\cdot\epsilon$-strongly relevant. Assume that it is not $\ee\cdot\e\cdot\epsilon$-strongly relevant. By above see-saw argument, we know that $(\pi_{\rho', k_1',h_1'},\pi_{\rho,k_1,h_1})$ is $\e\cdot\epsilon$-relevant, which implies that $(\pi_{\rho,k_1,h_1},\pi_{\rho', k_1',h_1'})$ is not $\ee\cdot\e\cdot\epsilon$-relevant. Recall that $k_1'= |k'|=k=|h_1|-1$. Since $(\pi_{\rho,k,h,\epsilon''},\pkhp)$ is strongly relevant, by Corollary \ref{strongly relevant1}, we have $h=|h'|$ or $h+1=|h'|$. Then
\[
\left\{
\begin{array}{ll}
k_1'=|h_1|-1;\\
k_1=|h_1'|
\end{array}\right.
\textrm{ or }
\left\{
\begin{array}{ll}
k_1'=|h_1|-1;\\
k_1=|h_1'|-1.
\end{array}\right.
\]
Hence by Corollary \ref{strongly relevant} (iii), $(\pi_{\rho,k_1,h_1},\pi_{\rho', k_1',-h_1'})$ is $\ee\cdot\e\cdot\epsilon$-relevant. On the other hand, by Proposition \ref{cus1} (i), $(\pi_{\rho', k_1',-h_1'},\pi_{\rho,k_1,h_1})$ is $\e\cdot\epsilon$-relevant. So
$(\pi_{\rho,k_1,h_1},\pi_{\rho', k_1',-h_1'})$ is $\ee\cdot\e\cdot\epsilon$-strongly relevant. By Theorem \ref{8.6}, one has
\begin{equation}\label{bess1}
\langle  \pi_{\rho,k_1,h_1}\otimes\omega^{\ee\cdot\e\cdot\epsilon}_{n_1},I^{\sp_{2n_1}}(\tau\otimes\pi_{\rho', k_1',-h_1'})\rangle _{\sp_{2n_1}(\Fq)}
=m_\psi(\pi_{\rho} ,\pi_{\rho'})\ne 0.
\end{equation}

Now consider above see-saw diagram again, by Proposition \ref{cus2} (i):
\[
\begin{aligned}
&\langle  \pi_{\rho,k_1,h_1}\otimes\omega^{\ee\cdot\e\cdot\epsilon}_{n_1},I^{\sp_{2n_1}}(\tau\otimes\pi_{\rho', k_1',-h_1'})\rangle _{\sp_{2n_1}(\Fq)}\\
=&\langle  \Theta^\epsilon_{n,n_1}(\pi_{\rho,k,h,\epsilon''})\otimes\omega^{\ee\cdot\e\cdot\epsilon}_{n_1},I^{\sp_{2n_1}}(\tau\otimes\pi_{\rho', k_1',-h_1'})\rangle _{\sp_{2n_1}(\Fq)}\\
=&\langle  \pi_{\rho,k,h,\epsilon''},\Theta^\e_{n_1,n+1}(I^{\sp_{2n_1}}(\tau\otimes\pi_{\rho', k_1',-h_1'}))\rangle _{\rm{O}^{\epsilon}_{2n+1}(\Fq)}\\
=&\langle  \pi_{\rho,k,h,\epsilon''},I^{\o^\e_{2(n+1)}}_P(\tau\otimes\pi_{\rho',k',-h'})\rangle _{\rm{O}^{\epsilon}_{2n+1}(\Fq)}.
\end{aligned}
\]
Since $(\pi_{\rho,k,h,\epsilon''},\pkhp)$ is strongly relevant, we know that $(\chi\otimes\pi_{\rho,k,h,\epsilon''},\chi\otimes\pkhp)=(\pi_{\rho_1,h,k,\epsilon''},\pi_{\rho_1',h',k'})$ is relevant.
By Proposition \ref{cus2} (i), $(\pi_{\rho_1,h,k,\epsilon''},\pi_{\rho_1',-h',k'})$ is not relevant. Then $(\pi_{\rho,k,h,\epsilon''},\pi_{\rho',k',-h'})$ is not strongly relevant, and by vanishing result
\[
\langle  \pi_{\rho,k_1,h_1}\otimes\omega^{\ee\cdot\e\cdot\epsilon}_{n_1},I^{\sp_{2n_1}}(\tau\otimes\pi_{\rho', k_1',-h_1'})\rangle _{\sp_{2n_1}(\Fq)}=\langle  \pi_{\rho,k,h,\epsilon''},I^{\o^\e_{2(n+1)}}_P(\tau\otimes\pi_{\rho',k',-h'})\rangle _{\rm{O}^{\epsilon}_{2n+1}(\Fq)}=0,
\]
which contradicts with (\ref{bess1}). Hence $(\pi_{\rho,k_1,h_1},\pi_{\rho', k_1',h_1'})$ must be $\ee\cdot\e\cdot\epsilon$-strongly relevant, which completes the proof.

$\bullet$ Suppose that $m_1-m=k+1$. By Corollary \ref{o4}, we have
\[
\langle  \pi_{\rho,k,h,\epsilon''},I^{\o^\e_{2(n+1)}}_P(\tau\otimes\pkhp)\rangle _{\rm{O}^{\epsilon}_{2n+1}(\Fq)}=
\langle  \pi_{\rho,k,h,\epsilon''},I^{\o^\e_{2n}}_P(\tau_2\otimes\pkhp)\rangle _{\rm{O}^{\e}_{2n}(\Fq)},
\]
where $\tau_2$ is given in Corollary \ref{o4}.
Consider the see-saw diagram:
\[
\setlength{\unitlength}{0.8cm}
\begin{picture}(20,5)
\thicklines
\put(6.5,4){$\sp_{2n_1}\times \sp_{2n_1}$}
\put(7.3,1){$\sp_{2n_1}$}
\put(12.3,4){$\o^{\epsilon}_{2n+1}$}
\put(11.6,1){$\o^{\ee}_{2n}\times \o^{\ee\cdot\epsilon}_1$}
\put(7.7,1.5){\line(0,1){2.1}}
\put(12.8,1.5){\line(0,1){2.1}}
\put(8,1.5){\line(2,1){4.2}}
\put(8,3.7){\line(2,-1){4.2}}
\end{picture}
\]
The rest of the proof runs as before.

\end{proof}

\section{The Gan-Gross-Prasad problem: generalization}\label{sec9}
The goal of this section is to prove Theorem \ref{main2}, which extends the previous result to arbitrary representations. We shall follow the method in section \ref{sec8}.

Let
\[
\begin{aligned}
&\mathcal{G}^{\rm{even},+}_{n,m}:=\left\{(\Lambda,\Lambda')|
\Upsilon(\Lambda')_*\preccurlyeq\Upsilon(\Lambda)_*,\Upsilon(\Lambda)^*\preccurlyeq\Upsilon(\Lambda')^*,\rm{def}(\Lambda)>0,\rm{def}(\Lambda')=\rm{def}(\Lambda)-1\right\}; \\
&\mathcal{G}^{\rm{even},-}_{n,m}:=\left\{(\Lambda,\Lambda')|\Upsilon(\Lambda')_*\preccurlyeq\Upsilon(\Lambda)^*,
\Upsilon(\Lambda)_*\preccurlyeq\Upsilon(\Lambda')^*,\rm{def}(\Lambda)>0,\rm{def}(\Lambda')=-\rm{def}(\Lambda)-1\right\}; \\
&\mathcal{G}^{\rm{odd},-}_{n,m}:=\left\{(\Lambda,\Lambda')|\Upsilon(\Lambda')^*\preccurlyeq \Upsilon(\Lambda)^*,
\Upsilon(\Lambda)_*\preccurlyeq\Upsilon(\Lambda')_*,\rm{def}(\Lambda)<0,\rm{def}(\Lambda')=\rm{def}(\Lambda)+1\right\};\\
&\mathcal{G}^{\rm{odd},+}_{n,m}:=\left\{(\Lambda,\Lambda')|\Upsilon(\Lambda')^*\preccurlyeq \Upsilon(\Lambda)_*,
\Upsilon(\Lambda)^*\preccurlyeq\Upsilon(\Lambda')_*,\rm{def}(\Lambda)<0,\rm{def}(\Lambda')=-\rm{def}(\Lambda)+1\right\}
\end{aligned}
\]
be subsets of $\cal{S}_n\times\cal{S}_m^{\pm}$, and let
\[
\mathcal{G}=\bigcup_{n,m}\left(\mathcal{G}^{\rm{even},+}_{n,m}\bigcup\mathcal{G}^{\rm{even},-}_{n,m}
\bigcup\mathcal{G}^{\rm{odd},-}_{n,m}\bigcup\mathcal{G}^{\rm{odd},+}_{n,m}\right).
\]

To prove the Fourier-Jacobi case of Theorem \ref{main2}, by Proposition \ref{so2}, it suffices to calculate (\ref{sp}). It is not hard to see that Theorem \ref{main2} (i) readily follows from Theorem \ref{9.1} below.

\begin{theorem}\label{9.1}
Let $s$ be a semisimple element of $\sp_{2n}(\Fq)^*=\so_{2n+1}\fq$, and $s'$ be a semisimple element of $\sp_{2m}(\Fq)^*=\so_{2m+1}\fq$. Let $\prll\in\cal{E}(\sp_{2n},s)$ be an irreducible representation of $\sp_{2n}(\Fq)$, and $\pi_{\rho_1,\Lambda_1,\Lambda_1'}\in\cal{E}(\sp_{2m},s')$ be an irreducible representation of $\sp_{2m}\fq$.
Assume that $n\ge m$, and let $\ell=n-m$. Let $P$ be an $F$-stable maximal parabolic subgroup of $\sp_{2n}$ with Levi factor $\GGL_{\ell} \times \sp_{2m}$. Let $s_0$ be a semisimple element of $\GGL_{\ell}\fq$ and let $\tau\in\cal{E}(\GGL_{\ell},s_0)$ be an irreducible cuspidal representation of $\GGL_{\ell}\fq$ which is nontrivial if $\ell=1$. Assume that $s_0$ has no common eigenvalues with $s$ and $s'$. Then we have
\[
\begin{aligned}
&\langle  \prll\otimes \omega_n^{\epsilon_0}, I_{P}^{\sp_{2n}}(\tau\otimes\pi_{\rho_1,\Lambda_1,\Lambda_1'})\rangle _{\sp_{2n}(\Fq)}\\
=&\left\{
\begin{array}{ll}
m_\psi(\pw_{\rho} ,\pw_{\rho_1}),  &\textrm{ if }(\prll,\pi_{\rho_1,\Lambda_1,\Lambda_1'})\textrm{ is $(\psi,\epsilon_0)$-strongly relevant, and there are } \widetilde{\Lambda_1'}\in\{\Lambda_1',\Lambda_1^{\prime t}\}\\
&\textrm{ and }\widetilde{\Lambda'}\in\{\Lambda',\Lambda^{\prime t}\}\textrm{ such that }(\Lambda,\widetilde{\Lambda_1'})\textrm{ and }(\Lambda_1,\widetilde{\Lambda'})\in \mathcal{G};\\
0,&\textrm{ otherwise}.
\end{array}\right.
\end{aligned}
\]
\end{theorem}

In particular, for unipotent representation $\pl$ of $\sp_{2n}\fq$ and $\theta$-epresentation $\pi_{-,\Lambda'}$ of $\sp_{2m}\fq$, we have

\begin{corollary}\label{7.1}
Let $n\ge m$. Let $\pl$ be an irreducible unipotent representation of $\sp_{2n}\fq$, and let $\pi_{-,\Lambda'}$ be an irreducible $\theta$-epresentation of $\sp_{2m}\fq$. Then
\[
m_\psi(\pl,\pi_{-,\Lambda'})=\left\{
\begin{array}{ll}
1, &\textrm{if }(\pl,\pi_{-,\Lambda'})\textrm{ is $(\psi,\ee)$-strongly relevant, and there is } \widetilde{\Lambda'}\in\{\Lambda',\Lambda^{\prime t}\}\textrm{ such}\\
&\textrm{that }(\Lambda,\widetilde{\Lambda'})\in \mathcal{G};\\
0, & \textrm{otherwise,}
\end{array}\right.
\]
\end{corollary}

Similarly, we have the same result for Bessel case.
\begin{theorem}\label{bess}
 Let $s$ be a semisimple element of $(\rm{O}^\epsilon_{2n+1}(\Fq)^*)^0$, and $s'$ be a semisimple element of $(\rm{O}^{\epsilon'}_{2m}(\Fq)^*)^0$. Let  $\pi_{\rho,\Omega,\Omega'}\in\cal{E}(\rm{O}^{\epsilon'}_{2m},s')$, and $\pi_{\rho_1,\Omega_1,\Omega_1',\epsilon''}\in\cal{E}(\rm{O}^\epsilon_{2n+1},s)$.

 (i) Assume that $n\ge m$. Let $P$ be an $F$-stable maximal parabolic subgroup of $\rm{O}^\e_{2(n+1)}$ with Levi factor $\GGL_{n-m+1} \times \rm{O}^{\epsilon'}_{2m}$. Let $s_0$ be a semisimple element of $\GGL_{n-m+1}\fq$ and let $\tau\in\cal{E}(\GGL_{n-m+1},s_0)$ be an irreducible cuspidal representation of $\GGL_{n-m+1}\fq$ which is nontrivial  if $n-m+1=1$. Assume that $s_0$ has no common eigenvalues with $s$ and $s'$. Then we have
\[
\begin{aligned}
&\langle  \pi_{\rho_1,\Omega_1,\Omega_1',\epsilon''},I^{\o^\e_{2(n+1)}}_P(\tau\otimes\pi_{\rho,\Omega,\Omega'})\rangle _{\rm{O}^{\epsilon}_{2n+1}(\Fq)}\\
=&\left\{
\begin{array}{ll}
m_\psi(\pw_{\rho} ,\pw_{\rho_1}),  &\textrm{ if }(\pi_{\rho,\Omega,\Omega',\epsilon''},\pi_{\rho_1,\Omega_1,\Omega_1'})\textrm{ is strongly relevant, and there are } \widetilde{\Omega}\in\{\Omega,\Omega^{ t}\}\\
&\textrm{ and } \widetilde{\Omega'}\in\{\Omega',\Omega^{\prime t}\}\textrm{ such that }(\Omega_1,\widetilde{\Omega})\textrm{ and }(\Omega_1',\widetilde{\Omega'})\in \mathcal{G};\\
0,&\textrm{ otherwise}.
\end{array}\right.
\end{aligned}
\]
where $m_\psi(\pw_{\rho} ,\pw_{\rho_1})$ is given in Theorem \ref{9.1} and $m_\psi(\pw_{\rho} ,\pw_{\rho_1})$ does not depend on $\psi$.

 (ii) Assume that $n< m$. Let $P$ be an $F$-stable maximal parabolic subgroup of $\rm{O}^\epsilon_{2m+1}$ with Levi factor $\GGL_{m-n} \times \rm{O}^{\epsilon}_{2n+1}$. Let $s_0$ be a semisimple element of $\GGL_{m-n}\fq$ and let $\tau\in\cal{E}(\GGL_{m-n},s_0)$ be an irreducible cuspidal representation of $\GGL_{m-n}\fq$ which is nontrivial  if $m-n=1$. Assume that $s_0$ has no common eigenvalues with $s$ and $s'$. Then we have
\[
\begin{aligned}
&\langle I^{\o^\epsilon_{2m+1}}_P( \tau\otimes\pi_{\rho_1,\Omega_1,\Omega'_1,\epsilon''}),\pi_{\rho,\Omega,\Omega'}\rangle _{\rm{O}^{\epsilon'}_{2m}(\Fq)}\\
=&\left\{
\begin{array}{ll}
m_\psi(\pw_{\rho} ,\pw_{\rho_1}),  &\textrm{ if }(\pi_{\rho,\Omega,\Omega',\epsilon''},\pi_{\rho_1,\Omega_1,\Omega_1'})\textrm{ is strongly relevant, and there are } \widetilde{\Omega}\in\{\Omega,\Omega^{ t}\}\\
&\textrm{ and } \widetilde{\Omega'}\in\{\Omega',\Omega^{\prime t}\}\textrm{ such that }(\Omega_1,\widetilde{\Omega})\textrm{ and }(\Omega_1',\widetilde{\Omega'})\in \mathcal{G};\\
0,&\textrm{ otherwise}.
\end{array}\right.
\end{aligned}
\]
where $m_\psi(\pw_{\rho} ,\pw_{\rho_1})$ is given in Theorem \ref{9.1} and $m_\psi(\pw_{\rho} ,\pw_{\rho_1})$ does not depend on $\psi$.

\end{theorem}
The rest of this section is devoted to the proof of Theorem \ref{9.1} and Theorem \ref{bess}, which will be divided into two parts.
\subsection{Vanishing result}
As before, we establish the cases where the multiplicity in Theorem \ref{9.1} and Theorem \ref{bess} vanishes.

\begin{proposition}\label{van}
Keep the assumptions in Proposition \ref{9.1}. If $(\pi_{\rho_{},\Lambda_{},\Lambda_{}'},\pi_{\rho_1,\Lambda_1,\Lambda_1'})$ is not $(\psi,\epsilon_0)$-strongly relevant, then we have
\[
\langle  \pi_{\rho_{},\Lambda_{},\Lambda_{}'}\otimes \omega_n^{\epsilon_0}, I_{P}^{\sp_{2n}}(\tau\otimes\pi_{\rho_1,\Lambda_1,\Lambda_1'})\rangle _{\sp_{2n}(\Fq)}=0.
\]
\end{proposition}
\begin{proof}
Assume that for any $\Lambda_0'\in\{\Lambda_{}',\Lambda_1'\}$, either $\rm{def}(\Lambda_0')\ne 0$ or $\Lambda_0'=\begin{pmatrix}
-\\
-
\end{pmatrix}$. In this case, the proposition follows immediately from \cite[Proposition 5.6]{LW3}, Proposition \ref{cus1} and the standard arguments of theta correspondence and see-saw dual pairs.

Assume that there is $\Lambda_0'\in\{\Lambda_{}',\Lambda_1'\}$ such that $\rm{def}(\Lambda_0')= 0$ and $\Lambda_0'\ne \begin{pmatrix}
-\\
-
\end{pmatrix}$. If the first occurrence index conditions of $(\psi,\epsilon_0)$-strongly relevant are not satisfied, then with the same see-saw argument in the proof of \cite[Proposition 5.6]{LW3} or in the see-saw example in P.6, the  multiplicity is 0. If the first occurrence index conditions hold, then $(\pi_{\rho_{},\Psi_{},\Psi_{}'},\pi_{\rho_1,\Psi_1,\Psi_1'})$ is not $(\psi,\epsilon_0)$-strongly relevant. By the same see-saw argument in the proof of Proposition \ref{cusnv}, we have
\[
\langle  \pi_{\rho_{},\Lambda_{},\Lambda_{}'}\otimes \omega_n^{\epsilon_0}, I_{P}^{\sp_{2n}}(\tau\otimes\pi_{\rho_1,\Lambda_1,\Lambda_1'})\rangle _{\sp_{2n}(\Fq)}\le\langle  \pi_{\rho_{},\Psi_{},\Psi_{}'}\otimes \omega_{n^{*}}^{\epsilon_0}, I_{P'}^{\sp_{2n^\star}}(\tau\otimes\pi_{\rho_1,\Psi_1,\Psi_1'})\rangle _{\sp_{2n^\star}(\Fq)}
\]
where $n^\star$, $\pi_{\rho_{},\Psi_{},\Psi_{}'}$ and $\pi_{\rho_1,\Psi_1,\Psi_1'}$ are defined in Definition \ref{strongly relevant0}. Note that $|\Upsilon(\Psi_1')^*|+|\Upsilon(\Psi_1')_*|<|\Upsilon(\Lambda_1')^*|+|\Upsilon(\Lambda_1')_*|$. Then we complete the proof by induction on $|\Upsilon(\Lambda_1')^*|+|\Upsilon(\Lambda_1')_*|$.

\end{proof}

\subsection{Non-vanishing result}

To prove the non-vanishing result we first need to know the theta correspondence of irreducible representations in the first occurrence index.

For a partition $\lambda=[\lambda_1,\lambda_2,\cdots,\lambda_k]$, we will denote by $\lambda^2$ the parition $[\lambda_2,\cdots,\lambda_k]$.
\begin{proposition}\label{f1}
Let $(G,G')=(\sp_{2n},\o^\epsilon_{2n'})$. Let $\Lambda\in\cal{S}_n$, and let $\Upsilon(\Lambda)=\begin{bmatrix}
\lambda\\
\mu
\end{bmatrix}$.
Let $\Lambda'\in\cal{S}^\epsilon_{n'}$, and let $\Upsilon(\Lambda')=\begin{bmatrix}
\mu'\\
\lambda'
\end{bmatrix}$.
Let $\pl$ be an irreducible unipotent representation of $\sp_{2n}\fq$, and let $\pll$ be an irreducible unipotent representation of $\o^\epsilon_{2n'}\fq$.

(i) Assume that $\epsilon=+$ and $n'$ is the first occurrence index of $\pl$. Then $n'=n-\lambda_1-\frac{\rm{def}(\Lambda)-1}{2}$ and $\Theta_{n,n'}^+(\pl)=\pll$ where $\Upsilon(\Lambda')=\begin{bmatrix}
\mu\\
\lambda^2
\end{bmatrix}$ and $\rm{def}(\Lambda')=-\rm{def}(\Lambda)+1$. Moreover, if $\Lambda''\ne \Lambda$ and $\pi_{\Lambda''}$ occurs in $\Theta_{n',n}^+\left(\Theta_{n,n'}^+(\pl)\right)$, then $(\Upsilon(\Lambda'')^*)_1>\lambda_1.$

(ii) Assume that $\epsilon=-$ and $n'$ is the first occurrence index of $\pl$. Then  $n'=n-\mu_1-\frac{-\rm{def}(\Lambda)-1}{2}$ and $\Theta_{n,n'}^-(\pl)=\pll$ where $\Upsilon(\Lambda')=\begin{bmatrix}
\mu^2\\
\lambda
\end{bmatrix}$ and $\rm{def}(\Lambda')=-\rm{def}(\Lambda)-1$. Moreover, if $\Lambda''\ne \Lambda$ and $\pi_{\Lambda''}$ occurs in $\Theta_{n',n}^-\left(\Theta_{n,n'}^-(\pl)\right))$, then $(\Upsilon(\Lambda'')_*)_1>\mu_1.$

(iii)  Assume that $\epsilon=+$ and $n$ is the first occurrence index of $\pll$. Then $n=n'-\mu_1'-\frac{\rm{def}(\Lambda)}{2}$ and $\Theta_{n',n}^+(\pll)=\pl$ where $\Upsilon(\Lambda)=\begin{bmatrix}
\lambda'\\
\mu^{\prime 2}
\end{bmatrix}$ and $\rm{def}(\Lambda'')=-\rm{def}(\Lambda')+1$. Moreover, if $\Lambda''\ne \Lambda$ and $\pi_{\Lambda''}$ occurs in $\Theta_{n,n'}^+\left(\Theta_{n',n}^+(\pl)\right)$, then $(\Upsilon(\Lambda'')^*)_1>\mu'_1.$

(iv) Assume that $\epsilon=-$ and $n$ is the first occurrence index of $\pll$. Then $n=n'-\lambda_1'+\frac{\rm{def}(\Lambda)}{2}$ and $\Theta_{n',n}^-(\pll)=\pl$ where $\Upsilon(\Lambda)=\begin{bmatrix}
\lambda^{\prime 2}\\
\mu'
\end{bmatrix}$ and $\rm{def}(\Lambda'')=-\rm{def}(\Lambda')-1$.  Moreover, if $\Lambda''\ne \Lambda$ and $\pi_{\Lambda''}$ occurs in $\Theta_{n,n'}^-\left(\Theta_{n',n}^-(\pl)\right)$, then $(\Upsilon(\Lambda'')_*)_1>\lambda'_1.$

\end{proposition}
\begin{proof}
We will only prove the (i). The proof of (ii), (iii) and (iv) is similar and will be left to the reader.
Recall that $\pl\otimes\pll\in \omega_{n,n'}^+$ if and only if
 \[
(\Lambda,\Lambda')\in\mathcal{B}^+_{n,n'}=\left\{(\Lambda,\Lambda')|{}^t\Upsilon(\Lambda')_*\preccurlyeq{}^t\Upsilon(\Lambda)^*,
{}^t\Upsilon(\Lambda)_*\preccurlyeq{}^t\Upsilon(\Lambda')^*,\rm{def}(\Lambda')=-\rm{def}(\Lambda)+1\right\},
\]
which implies that
\begin{equation}\label{eq9.1}
{}^t\lambda_i-1\le{}^t\lambda'_i\le{}^t\lambda_i
\end{equation}
and
\begin{equation}\label{eq9.2}
{}^t\mu'_i-1\le{}^t\mu\le{}^t\mu'_i.
\end{equation}
It follows that
\[
|\lambda|-\lambda_1=|\lambda^2|\le|\lambda'|\le|\lambda|
\]
and
\[
|\mu|\le|\mu'|.
\]
Recall that for every symbol $\Lambda_0$ (c.f. \cite[p.10]{P3} for details):
\[
\rm{rank}(\Lambda_0)=
\left\{
\begin{array}{ll}
|\Upsilon(\Lambda_0)^*|+|\Upsilon(\Lambda_0)_*|+(\frac{\rm{def}(\Lambda_0)-1}{2})(\frac{\rm{def}(\Lambda_0)+1}{2}),&\textrm{ if }\rm{def}(\Lambda_0)\textrm{ is odd;}\\
|\Upsilon(\Lambda_0)^*|+|\Upsilon(\Lambda_0)_*|+(\frac{\rm{def}(\Lambda_0)}{2})^2,&\textrm{ if }\rm{def}(\Lambda_0)\textrm{ is even.}\\
\end{array}
\right.
\]
Hence, there exist $\Lambda'$ such that $(\Lambda,\Lambda')\in\mathcal{B}^+_{n,n'}$ if and only if
\[
n'\ge |\mu|+|\lambda^2|+\left(\frac{\rm{def}(\Lambda')}{2}\right)^2=|\mu|+|\lambda^2|+\left(\frac{\rm{def}(\Lambda)-1}{2}\right)^2=n-\lambda_1-\frac{\rm{def}(\Lambda)-1}{2}.
\]
Moreover, if $n'=n-\lambda_1-\frac{\rm{def}(\Lambda)-1}{2}$, then $|\mu|=|\mu'|$ and $|\lambda^2|=|\lambda'|$. By (\ref{eq9.1}) and (\ref{eq9.2}), we get $
{}^t\lambda_i-1={}^t\lambda'_i\textrm{ and }{}^t\mu={}^t\mu'_i$.
Therefore,
\[
\Theta^{+}_{n,n'}(\pl)=\left\{
\begin{array}{ll}
0&\textrm{ if }n'<n-\lambda_1-\frac{\rm{def}(\Lambda)-1}{2};\\
\pll&\textrm{ if }n'=n-\lambda_1-\frac{\rm{def}(\Lambda)-1}{2}.
\end{array}\right.
\]
where $\Upsilon(\Lambda')=\begin{bmatrix}
\mu\\
\lambda^2
\end{bmatrix}$ and $\rm{def}(\Lambda')=-\rm{def}(\Lambda)+1$.

For any irreducible representation $\pi_{\Lambda''}\in\Theta_{n',n}^+\left(\Theta_{n,n'}^+(\pl)\right)=\Theta_{n',n}^+\left(\pll\right)$, we have $\rm{def}(\Lambda'')=\rm{def}(\Lambda)$ and $|\mu|+|\lambda|=|\Upsilon(\Lambda'')^*|+|\Upsilon(\Lambda'')_*|$. If $(\Upsilon(\Lambda'')^*)_1<\lambda_1$, then the first occurrence index $n''$ of $\pi_{\Lambda''}$ is equal to $n-(\Upsilon(\Lambda'')^*)_1-\frac{\rm{def}(\Lambda'')-1}{2}>n'$, which is impossible.
Suppose that $(\Upsilon(\Lambda'')^*)_1=\lambda_1$. Since $\pi_{\Lambda''}\otimes\pll\in \omega_{n,n'}^+$, we have
\begin{equation}\label{eq9.3}
({}^t\Upsilon(\Lambda'')^*)_i-1\le{}^t\lambda'_i\le({}^t\Upsilon(\Lambda'')^*)_i
\end{equation}
and
\begin{equation}\label{eq9.4}
{}^t\mu'_i-1\le({}^t\Upsilon(\Lambda'')_*)_i\le{}^t\mu'_i.
\end{equation}
Then
\[
|\Upsilon(\Lambda'')^*|-\lambda_1\le|\lambda^2|=|\lambda|-\lambda_1.
\]
and
\[
|\Upsilon(\Lambda'')_*|\le|\mu'|=|\mu|.
\]
Note that
\[
|\Upsilon(\Lambda'')^*|=|\mu|+|\lambda|-|\Upsilon(\Lambda'')_*|\ge|\lambda|.
 \]
Hence, $|\Upsilon(\Lambda'')^*|=|\lambda|$ and $|\Upsilon(\Lambda'')_*|=|\mu|$. By (\ref{eq9.3}) and (\ref{eq9.4}), we get $\Lambda''=\Lambda$.
So if $\Lambda''\ne\Lambda$, then $(\Upsilon(\Lambda'')^*)_1>\lambda_1.$
\end{proof}

\begin{proposition}\label{van1}
Keep the assumptions in Theorem \ref{9.1}. Assume that $(\prll,\pi_{\rho_1,\Lambda_1,\Lambda_1'})$ is $\epsilon_0$-strongly relevant.
If
\[
\langle  \prll\otimes \omega_n^{\epsilon_0}, I_{P}^{\sp_{2n}}(\tau\otimes\pi_{\rho_1,\Lambda_1,\Lambda_1'})\rangle _{\sp_{2n}(\Fq)}\ne0,
\]
then there are $\widetilde{\Lambda_1'}\in\{\Lambda_1',\Lambda_1^{\prime t}\}$ and $\widetilde{\Lambda'}\in\{\Lambda',\Lambda^{\prime t}\}$ such that $(\Lambda,\widetilde{\Lambda_1'})$ and $(\Lambda_1,\widetilde{\Lambda'})\in \mathcal{G}$.

\end{proposition}
\begin{proof}

We prove the proposition by induction on
\[
r=|\Upsilon(\Lambda)^*|+|\Upsilon(\Lambda)_*|+|\Upsilon(\Lambda')^*|+|\Upsilon(\Lambda')_*|+|\Upsilon(\Lambda_1)^*|+|\Upsilon(\Lambda_1)_*|+|\Upsilon(\Lambda_1')^*|+|\Upsilon(\Lambda_1')_*|.
 \]
For $r=0$, it is Theorem \ref{8.6}. Assume that this proposition hold for $r<N$, we prove for $r=N$. We only prove $(\Lambda,\widetilde{\Lambda_1'})\in \cal{G}$, and the proof for $(\Lambda',\widetilde{\Lambda_1})\in \mathcal{G}$ is similar. And we only prove for $\epsilon_0=\ee$.

Since $r>0$, there are two cases of symbols as follows:
\begin{itemize}
\item[]
Case (A): $|\Upsilon(\Lambda)^*|+|\Upsilon(\Lambda)_*|+|\Upsilon(\Lambda_1')^*|+|\Upsilon(\Lambda_1')_*|\ne 0$;
\item[]
Case (B): $|\Upsilon(\Lambda)^*|+|\Upsilon(\Lambda)_*|+|\Upsilon(\Lambda_1')^*|+|\Upsilon(\Lambda_1')_*|= 0$ and $|\Upsilon(\Lambda')^*|+|\Upsilon(\Lambda')_*|+|\Upsilon(\Lambda_1)^*|+|\Upsilon(\Lambda_1)_*|\ne 0$.
\end{itemize}
We will only prove the case (A). The proof of the case (B) is similar.
Let $\prll\in\cal{E}(\sp_{2n},\pi_{\rho',k,h})$ and $\pi_{\rho_1,\Lambda_1,\Lambda_1'}\in\cal{E}(\sp_{2m},\pi_{\rho'_1,k_1,h_1})$ where $\pi_{\rho',k,h}$ and $\pi_{\rho'_1,k_1,h_1}$ are two cuspidal representations of $\sp_{2n'}\fq$ and $\sp_{2m'}\fq$, respectively.
Since $(\prll,\pi_{\rho_1,\Lambda_1,\Lambda_1'})$ is $\ee$-strongly relevant, by Corollary \ref{strongly relevant}, there are four possibilities:
\begin{itemize}
\item[]
Case (A.1): $\rm{def}(\Lambda)>0$ and $k=|h_1| $;
\item[]
Case (A.2): $\rm{def}(\Lambda)>0$ and $k=|h_1|-1$;
\item[]
Case (A.3): $\rm{def}(\Lambda)<0$ and $k=|h_1| $;
\item[]
Case (A.4): $\rm{def}(\Lambda)<0$ and $k=|h_1|-1$.
\end{itemize}
We will only prove the case (A.1). The proof of the rest cases is similar and will be left to the reader.

Assume that $\rm{def}(\Lambda_1')\ne 0$ or $\Lambda_1'=\begin{pmatrix}
-\\
-
\end{pmatrix}$. Pick $\epsilon\in\{\pm\}$ such that the first occurrence index of $\pi_{\rho',k,h}$ in the Witt tower ${\bf O}^\epsilon_\rm{even}$ is $n'-k$. Since $\rm{def}(\Lambda)>0$, by Proposition \ref{q1} (iv) and (\ref{cussp}), we can conclude that $k$ is even. Let $\pi_{\rho',k',h'}=\Theta^\epsilon_{n'.n'-k}(\pi_{\rho',k,h})$. By Proposition \ref{cus1}, we know that $k'$ is also even. For any irreducible representation $\pi_{\rho,\Omega,\Omega'}$ of $\o^\epsilon_{2n^*}\fq$, if $\prll\otimes\pi_{\rho,\Omega,\Omega'}$ appears in $\omega^\epsilon_{n,n^*}$, then by Proposition \ref{o3}, we have
\[
\pi_{\rho,\Omega,\Omega'}\in\cal{E}(\o^\epsilon_{2n^*},\pi_{\rho',k',h'}) \textrm{ and }\rm{def}(\Omega)=2k'=0\textrm{ mod }4.
\]
Hence, by Corollary \ref{ctheta2} there is a symbol $\widetilde{\Omega}\in\{\Omega,\Omega^t\}$ such that if $\prll\otimes\pi_{\rho,\Omega,\Omega'}$ appears in $\omega^\epsilon_{n,n^*}$, then $(\Lambda,\widetilde{\Omega})\in \cal{B}^+_{\rm{rk}(\Lambda),\rm{rk}(\widetilde{\Omega})}$ i.e.
\begin{equation}\label{b+1}
\pi_\Lambda\otimes\pi_{\widetilde{\Omega}}\textrm{ appears in } \omega^+_{\rm{rk}(\Lambda),\rm{rk}(\widetilde{\Omega})}.
\end{equation}
With same argument, for any irreducible representation $\pi_{\rho,\Gamma,\Gamma'}$ of $\sp_{2n_*}\fq$, if $\pi_{\rho,\Gamma,\Gamma'}\otimes\pi_{\rho,\Omega,\Omega'}$ appears in $\omega^\epsilon_{n_*,n^*}$, then we have $(\Gamma,\widetilde{\Omega})\in \cal{B}^+_{\rm{rk}(\Gamma),\rm{rk}(\widetilde{\Omega})}$ i.e.
\begin{equation}\label{b++1}
\pi_\Gamma\otimes\pi_{\widetilde{\Omega}} \textrm{ appears in } \omega^+_{\rm{rk}(\Gamma),\rm{rk}(\widetilde{\Omega})}.
\end{equation}

Recall that $k=|h_1|$, then $h_1$ is even.
Since $(\pi_{\rho',k,h},\pi_{\rho'_1,k_1,h_1})$ is $\ee$-strongly relevant and $k=|h_1|$, by Proposition \ref{cus1}, the first occurrence index of $\pi_{\rho_1',k_1,h_1}$ in the Witt tower ${\bf O}^{\ee\cdot\epsilon}_\rm{odd}$ is $m'-|h_1|=m'-k$. Moreover, we have
$
\Theta^{\ee\cdot\epsilon}_{n',m'-k}(\pi_{\rho_1',k_1,h_1})=\pi_{\rho_1',k'_1,h'_1,\epsilon_1}
$
 with $k_1'=|h_1'|-1=k-1$ and $h'_1= k_1$. Since $\pi_{\rho_1,\Lambda_1,\Lambda_1'}\in\cal{E}(\sp_{2m},\pi_{\rho'_1,k_1,h_1})$, by Proposition \ref{o3} (i), for any irreducible representation $\pi_{\rho_1,\Omega_1,\Omega_1',\epsilon_1}$ of $\o^{\ee\cdot\epsilon}_{2m^*+1}\fq$ such that $\pi_{\rho_1,\Lambda_1,\Lambda_1'}\otimes\pi_{\rho_1,\Omega_1,\Omega_1',\epsilon_1}$ appears in $\omega^{\ee\cdot\epsilon}_{m,m^*}$, we conclude that
 \[
 \pi_{\rho_1,\Omega_1,\Omega_1',\epsilon_1}\in \cal{E}(\o^{\ee\cdot\epsilon}_{2m^*+1},\pi_{\rho_1',k'_1,h'_1,\epsilon_1})\textrm{ and }\frac{|\rm{def}(\Omega_1)|-1}{2}=k-1.
 \]
 Since
$\rm{def}(\Omega_1)=1\ (\textrm{mod }4)$ and $k$ is even,
  we get $\rm{def}(\Omega_1)=-2k+1$.
Note that $h_1$ is even, we have
\[
\rm{def}(\Lambda_1')=2h_1=0\ (\textrm{mod }4).
\]
By Corollary \ref{ctheta3}, either $(\Omega_1,\Lambda_1')$ or $(\Omega_1,\Lambda_1^{\prime t})\in \cal{B}^{+}_{\rm{rk}(\Omega_1),\rm{rk}(\Lambda_1')}$. Pick a symbol $\widetilde{\Lambda_1'}\in\{\Lambda_1',\Lambda_1^{\prime t}\}$ with
$
\rm{def}(\widetilde{\Lambda_1'})=2|h_1|=2k>0,
$
 then we have
\begin{equation}\label{b+2}
\pi_{\Omega_1}\otimes\pi_{\widetilde{\Lambda_1'}} \textrm{ appears in } \omega^+_{\rm{rk}(\Omega_1),\rm{rk}(\Lambda_1')}.
\end{equation}
With same argument, for any irreducible representation $\pi_{\rho_1,\Gamma_1,\Gamma_1'}$ of $\sp_{2m_*}\fq$, if $\pi_{\rho_1,\Gamma_1,\Gamma_1'}\otimes\pi_{\rho_1,\Omega_1,\Omega_1'}$ appears in $\omega^{\ee\cdot\epsilon}_{m_*,m^*}$, then we have
\begin{equation}\label{b++2}
\pi_{\Omega_1}\otimes\pi_{\widetilde{\Gamma_1'}} \textrm{ appears in } \omega^+_{\rm{rk}(\Omega_1),\rm{rk}(\widetilde{\Gamma_1'})}
\end{equation}
where $\widetilde{\Gamma_1'}\in\{\Gamma_1',\Gamma_1^{\prime t}\}\in$ such that $\rm{def}(\Gamma_1')=\rm{def}(\Lambda_1')$.

Write $\Upsilon(\Lambda)=\begin{bmatrix}
\lambda\\
\mu
\end{bmatrix}$ and
$\Upsilon(\widetilde{\Lambda_1'})=\begin{bmatrix}
\lambda'\\
\mu'
\end{bmatrix}$. To ease notations we suppress various Levi subgroups from the parabolic induction in the
sequel, which should be clear from the context.

Let $n_1=n-\lambda_1-\frac{\rm{def}(\Lambda)-1}{2}=n-\lambda_1-k$. By Theorem \ref{p1}, Proposition \ref{f1} (i) and (\ref{b+1}), $n_1$ is the first occurrence index of $\prll$ in the Witt tower ${\bf O}^\epsilon_\rm{even}$, and there is an irreducible representation $\pi_{\rho,\Omega,\Omega'}$ of $\o^\epsilon_{2n_1}\fq$ such that $\Upsilon(\widetilde{\Omega})=\begin{bmatrix}
\mu\\
\lambda^2
\end{bmatrix}$ and $\rm{def}(\widetilde{\Omega})=-\rm{def}(\Lambda)+1$,
and
\[
\prll\subset\Theta^\epsilon_{n_1,n}(\pi_{\rho,\Omega,\Omega'})
\]
where $\Omega'\in\{\Lambda',\Lambda^{\prime t}\}$.

Consider the see-saw diagram
\[
\setlength{\unitlength}{0.8cm}
\begin{picture}(20,5)
\thicklines
\put(6.5,4){$\sp_{2n}\times \sp_{2n}$}
\put(7.3,1){$\sp_{2n}$}
\put(12.3,4){$\o^{\ee\cdot\epsilon}_{2n_1+1}$}
\put(11.6,1){$\o^\epsilon_{2n_1}\times \o^{\ee}_1$}
\put(7.7,1.5){\line(0,1){2.1}}
\put(12.8,1.5){\line(0,1){2.1}}
\put(8,1.5){\line(2,1){4.2}}
\put(8,3.7){\line(2,-1){4.2}}
\end{picture}
\]

By Proposition \ref{w2}, one has
\[
\begin{aligned}
\langle \prll\otimes\omega_{n}^{\ee},I^{\sp_{2n}}(\tau\otimes \pi_{\rho_1,\Lambda_1,\Lambda_1'})\rangle_{\sp_{2n}(\Fq)}\le&\langle \Theta^\epsilon_{n_1,n}(\pi_{\rho,\Omega,\Omega'})\otimes\omega_{n}^{\ee},I^{\sp_{2n}}(\tau\otimes \pi_{\rho_1,\Lambda_1,\Lambda_1'})\rangle_{\sp_{2n}(\Fq)},\\
=&\langle \pi_{\rho,\Omega,\Omega'},\Theta^{\ee\cdot\epsilon}_{n,n_1}(I^{\sp_{2n}}(\tau\otimes \pi_{\rho_1,\Lambda_1,\Lambda_1'}))\rangle_{\o^\epsilon_{2n_1}(\Fq)}
\end{aligned}
\]
By Proposition \ref{w1}, every irreducible constituent of $\Theta^{\ee\cdot\epsilon}_{n,n_1}(I^{\sp_{2n}}(\tau\otimes \pi_{\rho_1,\Lambda_1,\Lambda_1'}))$ occurs in
 \[
 I^{\o^{\ee\cdot\epsilon}_{2n_1+1}}\left((\chi\otimes\tau)\otimes \Theta^{\ee\cdot\epsilon}_{m,m-(n-n_1)}(\pi_{\rho_1,\Lambda_1,\Lambda_1'})\right).
  \]
 By Theorem \ref{p2}, Proposition \ref{f1} (iii) and (\ref{b+2}), the first occurrence index of $\pi_{\rho_1,\Lambda_1,\Lambda_1'}$ is
 \[
 m-\lambda_1'-|h_1|=m-\lambda_1'-k.
 \]
  If $\lambda_1'<\lambda_1$, then
\[
\Theta^{\ee\cdot\epsilon}_{m,m-(n-n_1)}(\pi_{\rho_1,\Lambda_1,\Lambda_1'})
=\Theta^{\ee\cdot\epsilon}_{m,m-\lambda_1-k}(\pi_{\rho_1,\Lambda_1,\Lambda_1'})=0,
\]
which implies that
\[
\langle  \prll\otimes \omega_n^{\ee}, I_{P}^{\sp_{2n}}(\tau\otimes\pi_{\rho_1,\Lambda_1,\Lambda_1'})\rangle _{\sp_{2n}(\Fq)}=0.
\]
This contradicts our assumption. So
\begin{equation}\label{eq7.1}
\lambda_1'\ge\lambda_1.
\end{equation}

Let $m_1'=m-\lambda_1'-\frac{\rm{def}(\widetilde{\Lambda_1'})}{2}=m-\lambda_1'-|h_1|$. By Theorem \ref{p2}, Proposition \ref{f1} (iii) and (\ref{b+2}), there is an irreducible representation $\pi_{\rho_1,\Omega_1,\Omega_1',\epsilon_1}$ of $\o^{\ee\cdot\epsilon}_{2m_1'+1}\fq$ where $\Omega_1$ is a symbol such that $\Upsilon(\Omega_1)=\begin{bmatrix}
\mu'\\
\lambda^{\prime 2}
\end{bmatrix}$ and $\rm{def}(\Omega_1)=-\rm{def}(\widetilde{\Lambda_1'})+1$, and
\[
\pi_{\rho_1,\Lambda_1,\Lambda_1'}\subset\Theta^{\ee\cdot\epsilon}_{m_1',m}(\pi_{\rho_1,\Omega_1,\Omega_1',\epsilon_1}).
\]
with $\Omega_1'=\Lambda_1$. Let $n_1'=n-\lambda_1'-|h_1|$.

Now consider the see-saw diagram
\[
\setlength{\unitlength}{0.8cm}
\begin{picture}(20,5)
\thicklines
\put(6.5,4){$\sp_{2n}\times \sp_{2n}$}
\put(7.3,1){$\sp_{2n}$}
\put(12,4){$\o^\epsilon_{2(n_1'+1)}$}
\put(11.3,1){$\o^{\ee\cdot\epsilon}_{2n_1'+1}\times \o^{+}_1$}
\put(7.7,1.5){\line(0,1){2.1}}
\put(12.8,1.5){\line(0,1){2.1}}
\put(8,1.5){\line(2,1){4.2}}
\put(8,3.7){\line(2,-1){4.2}}
\end{picture}
\]
One has
\[
\begin{aligned}
\langle \prll\otimes\omega_{n}^{\ee},I^{\sp_{2n}}(\tau\otimes \prllc)\rangle_{\sp_{2n}(\Fq)}
=\langle  I^{\sp_{2n}}(\tau\otimes \pi_{\rho_1,\Lambda_1,\Lambda_1'})\otimes\omega_{n}^+, \prll \rangle_{\sp_{2n}(\Fq)}.\\
\end{aligned}
\]
For every irreducible constituent $\rho_1\in I^{\sp_{2n}}(\tau\otimes \pi_{\rho_1,\Lambda_1,\Lambda_1'})$, by Proposition \ref{w1}, there is an irreducible representation
\[
\rho_1'\in I^{\o^{\ee\cdot\epsilon}_{2n_1'+1}}\left((\chi\otimes\tau)\otimes \Theta^{\ee\cdot\epsilon}_{m,m_1'}(\pi_{\rho_1,\Omega_1,\Omega_1',\epsilon_1})\right)
 \]
 such that $\rho_1\in \Theta^{\ee\cdot\epsilon}_{n_1',n}(\rho_1')$. Then
\[
\langle  \rho_1\otimes\omega_{n}^+, \prll \rangle_{\sp_{2n}(\Fq)}
\le\langle  \Theta^{\ee\cdot\epsilon}_{n_1',n}(\rho_1')\otimes\omega_{n}^+, \prll \rangle_{\sp_{2n}(\Fq)}
=\langle  \rho_1', \Theta^\epsilon_{n,n_1'+1}(\prll)\rangle_{\o^{\ee\cdot\epsilon}_{2n_1'+1}(\Fq)}.
\]
Recall that the first occurrence index of $\prll$ is $n-\lambda_1-k$. If $\lambda_1<\lambda_1'-1$, then
\[
\Theta^\epsilon_{n,n_1+1}(\prll)
=\Theta^\epsilon_{n,n-(\lambda_1'-1)-k}(\prll)=0.
\]
which implies that
\[
\langle  \prll\otimes \omega_n^{\ee}, I_{P}^{\sp_{2n}}(\tau\otimes\pi_{\rho_1,\Lambda_1,\Lambda_1'})\rangle _{\sp_{2n}(\Fq)}=0.
\]
This contradicts our assumption. So
\begin{equation}\label{eq7.2}
\lambda_1'-1\le\lambda_1.
\end{equation}

By (\ref{eq7.1}) and (\ref{eq7.2}), there are only two cases for $\lambda$ to be considered: $\lambda_1=\lambda_1'$ or $\lambda_1=\lambda_1'-1$.
We will only prove the proposition for $\lambda_1=\lambda_1'$ by using our first see-saw again. The proof for $\lambda_1=\lambda_1'-1$ is similar by the second one, and will be left to the reader.

Suppose that $\lambda_1=\lambda_1'$.
Then by above discussion, one has
\[
\Theta^{\ee\cdot\epsilon}_{m,m-(n-n_1)}(\pi_{\rho_1,\Lambda_1,\Lambda_1'})={\pi}_{\rho_1,\Omega_1,\Omega_1',\epsilon_1}.
\]
Consider the first see-saw, we can conclude that if
\[
\langle \pi_{\rho,\Omega,\Omega'},I^{\o^{\ee\cdot\epsilon}_{2n_1+1}}((\chi\otimes\tau)\otimes {\pi}_{\rho_1,\Omega_1,\Omega_1',\epsilon_1})\rangle_{\o^\epsilon_{2n_1}(\Fq)}=0 ,
\]
then
\[
\langle \prll\otimes\omega^{\ee}_{n},I^{\sp_{2n}}(\tau\otimes \pi_{\rho_1,\Lambda_1,\Lambda_1'})\rangle_{\sp_{2n}(\Fq)}=0.
\]
So it remains to prove that if
\[
\langle \pi_{\rho,\Omega,\Omega'},I^{\o^{\ee\cdot\epsilon}_{2n_1+1}}((\chi\otimes\tau)\otimes {\pi}_{\rho_1,\Omega_1,\Omega_1',\epsilon_1})\rangle_{\o^\epsilon_{2n_1}(\Fq)}\ne0 ,
\]
 then $(\Lambda,\widetilde{\Lambda_1'})\in \mathcal{G}$.

We now turn to prove (\ref{eq7.3}), (\ref{eq7.4}) and (\ref{eq7.5}) which means $\mu'_1=\mu_1$ or $\mu_1'=\mu_1-1$. Let $m_2'=n_1-\mu_1'-\frac{\rm{def}(\Omega_1)-1}{2}=m-\mu_1'-\lambda_1=m-\mu_1'-\lambda_1'$, and let $n_2'=n-\mu_1'-\lambda_1'$. By Theorem \ref{p2}, Proposition \ref{f1} (i) and (\ref{b++2}), with the same argument as $\widetilde{\Lambda_1'}$, there is a representation $\pi_{\rho_1,\Gamma_1,\Gamma_1'}$ of $\sp_{2m_2'}\fq$ such that $\Upsilon(\widetilde{\Gamma_1'})=\begin{bmatrix}
\lambda^{\prime 2}\\
\mu^{\prime 2}
\end{bmatrix}$ and $\rm{def}(\widetilde{\Gamma_1'})=-\rm{def}(\Omega_1)+1=\rm{def}(\widetilde{\Lambda_1'})=2k>0$, and
\[
{\pi}_{\rho_1,\Omega_1,\Omega_1',\epsilon_1}\subset\Theta^{\ee\cdot\epsilon}_{n_2',n_1}(\pi_{\rho_1,\Gamma_1,\Gamma_1'})
\]
where $\Gamma_1=\Lambda_1$ and $\widetilde{\Gamma_1'}\in\{\Gamma_1',\Gamma_1^{\prime t}\}$ with $\rm{def}(\widetilde{\Gamma_1'})>0$.

 Consider the see-saw diagram
\[
\setlength{\unitlength}{0.8cm}
\begin{picture}(20,5)
\thicklines
\put(6.3,4){$\sp_{2n_2'}\times \sp_{2n_2'}$}
\put(7.3,1){$\sp_{2n_2'}$}
\put(12.3,4){$\o^{\ee\cdot\epsilon}_{2n_1+1}$}
\put(11.6,1){$\o^\epsilon_{2n_1}\times \o^{\ee}_1$}
\put(7.7,1.5){\line(0,1){2.1}}
\put(12.8,1.5){\line(0,1){2.1}}
\put(8,1.5){\line(2,1){4.2}}
\put(8,3.7){\line(2,-1){4.2}}
\end{picture}
\]
For every irreducible constituent $\rho_2\in I^{\o^{\ee\cdot\epsilon}_{2n_1+1}}((\chi\otimes\tau)\otimes {\pi}_{\rho_1,\Omega_1,\Omega_1',\epsilon_1})$, by Proposition \ref{w1}, there is a representation $\rho_2'\in I^{\sp_{2n_2'}}(\tau\otimes \pi_{\rho_1,\Gamma_1,\Gamma_1'})$ such that $\rho_2\in \Theta^{\ee\cdot\epsilon}_{n_2',n_1}(\rho_2')$. Then
\[
\begin{aligned}
\langle \pi_{\rho,\Omega,\Omega'},\rho_2\rangle_{\o^\epsilon_{2n_1}(\Fq)}
\le\langle \pi_{\rho,\Omega,\Omega'},\Theta^{\ee\cdot\epsilon}_{n_2',n_1}(\rho_2')\rangle_{\o^\epsilon_{2n_1}(\Fq)}
=\langle \Theta^\epsilon_{n_1,n_2'}(\pi_{\rho,\Omega,\Omega'})\otimes\omega_{n_2'}^{\ee} ,\rho_2'\rangle_{\sp_{2n_2'}(\Fq)}.
\end{aligned}
\]
By Theorem \ref{p1}, Proposition \ref{f1} and (\ref{b++1}), the first occurrence index of $\pi_{\rho,\Omega,\Omega'}$ is
\[
n_1-\mu_1+k=n-\lambda_1-\mu_1.
\]
If $\mu_1<\mu_1'$, then
\[
\Theta^\epsilon_{n_1,n_2'}(\pi_{\rho,\Omega,\Omega'})=\Theta^\epsilon_{n_1,n-\mu_1'-\lambda_1'}(\pi_{\rho,\Omega,\Omega'})=0,
\]
which implies that
\[
\langle \pi_{\rho,\Omega,\Omega'},I^{\o^{\ee\cdot\epsilon}_{2n_1+1}}((\chi\otimes\tau)\otimes {\pi}_{\rho_1,\Omega_1,\Omega_1',\epsilon_1})\rangle_{\o^\epsilon_{2n_1}(\Fq)}=0 .
\]
This contradicts our assumption. So
\begin{equation}\label{eq7.3}
\mu_1\ge\mu_1'.
\end{equation}

Let $n_2=n_1-\mu_1-\frac{\rm{def}(\widetilde{\Omega})}{2}=n-\mu_1-\lambda_1$. By Theorem \ref{p1}, Proposition \ref{f1} (iii) and (\ref{b++1}), there is a representation $\pi_{\rho,\Gamma,\Gamma'}$ of $\sp_{2n_2}\fq$ where $\Gamma$ is a symbol such that $\Upsilon(\Gamma)=\begin{bmatrix}
\lambda^2\\
\mu^2
\end{bmatrix}$ and $\rm{def}(\Gamma)=-\rm{def}(\widetilde{\Omega})+1=\rm{def}(\Lambda)$, and
\[
{\pi}_{\rho,\Omega,\Omega'}\subset\Theta^\epsilon_{n_2,n_1}(\pi_{\rho,\Gamma,\Gamma'}).
\]
with $\Gamma'=\Lambda'$.

By Proposition \ref{7.21} and Corollary \ref{o4}, recall that $\tau\in\cal{E}(\GGL_\ell)$, one has
\[
\begin{aligned}
&\langle {\pi}_{\rho,\Omega,\Omega'},I^{\o^{\ee\cdot\epsilon}_{2n_1+1}}((\chi\otimes\tau)\otimes {\pi}_{\rho_1,\Omega_1,\Omega_1',\epsilon_1})\rangle_{\o^\epsilon_{2n_1}(\Fq)}\\
=&\left\{
\begin{array}{ll}
\langle {\pi}_{\rho,\Omega,\Omega'},I^{\o^{\ee\cdot\epsilon}_{2n_1-1}}(\tau_1\otimes \pi_{\rho_1,\Omega_1,\Omega_1',\epsilon_1})\rangle_{\o^{\ee\cdot\epsilon}_{2n_1-1}(\Fq)},&\textrm{ if }\ell\ne 0;\\
\langle I^{\o^\epsilon_{2(n_1+1)}}(\tau_1\otimes{\pi}_{\rho,\Omega,\Omega'}), \pi_{\rho_1,\Omega_1,\Omega_1',\epsilon_1}\rangle_{\o^{\ee\cdot\epsilon}_{2n_1+1}(\Fq)},&\textrm{ if }\ell= 0,\\
\end{array}\right.
\end{aligned}
\]
where $\tau_1\in\cal{E}(\GGL_1(\bb{F}_{q^2}),s_1)$ is a cuspidal representation with $s_1\ne s_1^{-1}$ and $s_1$ have no common eigenvalues with $s$ and $s'$.

Suppose that $\ell\ne 0$. Now consider the see-saw diagram
\[
\setlength{\unitlength}{0.8cm}
\begin{picture}(20,5)
\thicklines
\put(6.3,4){$\sp_{2n_2}\times \sp_{2n_2}$}
\put(7.3,1){$\sp_{2n_2}$}
\put(12.2,4){$\o^\epsilon_{2n_1}$}
\put(11.3,1){$\o^{\ee\cdot\epsilon}_{2n_1-1}\times \o^{+}_1$}
\put(7.7,1.5){\line(0,1){2.1}}
\put(12.8,1.5){\line(0,1){2.1}}
\put(8,1.5){\line(2,1){4.2}}
\put(8,3.7){\line(2,-1){4.2}}
\end{picture}
\]
By Proposition \ref{f1}, one has
\[
\langle {\pi}_{\rho,\Omega,\Omega'},I^{\o^{\ee\cdot\epsilon}_{2n_1-1}}(\tau_1\otimes \pi_{\rho_1,\Omega_1,\Omega_1',\epsilon_1})\rangle_{\o^{\ee\cdot\epsilon}_{2n_1-1}(\Fq)}
\le\langle\Theta^\epsilon_{n_2,n_1} (\pi_{\rho,\Gamma,\Gamma'}),I^{\o^{\ee\cdot\epsilon}_{2n_1-1}}(\tau_1\otimes \pi_{\rho_1,\Omega_1,\Omega_1',\epsilon_1})\rangle_{\o^{\ee\cdot\epsilon}_{2n_1-1}(\Fq)}
\]
For every irreducible constituent $\rho_3\in I^{\o^{\ee\cdot\epsilon}_{2n_1-1}}(\tau_1\otimes \pi_{\rho_1,\Omega_1,\Omega_1',\epsilon_1})$, one has
\[
\begin{aligned}
\langle\Theta^\epsilon_{n_2,n_1} (\pi_{\rho,\Gamma,\Gamma'}),\rho_3\rangle_{\o^{\ee\cdot\epsilon}_{2n_1-1}(\Fq)}
=\langle\pi_{\rho,\Gamma,\Gamma'},\Theta^{\ee\cdot\epsilon}_{n_1-1,n_2}(\rho_3)\otimes\omega_{n_2}^+\rangle_{\sp_{2n_2}(\Fq)}.
\end{aligned}
\]
By Proposition \ref{w1}, every irreducible constituent of $\Theta^{\ee\cdot\epsilon}_{n_1-1,n_2}(\rho_3)$ appears in
\[
I^{\sp_{2n_2}}\left((\chi\otimes\tau_1)\otimes\Theta^{\ee\cdot\epsilon}_{m-(n-n_1),m-(n-n_2)}(\pi_{\rho_1,\Omega_1,\Omega_1',\epsilon_1})\right).
\]
By Theorem \ref{p2}, Proposition \ref{f1} (i) and (\ref{b++2}), the first occurrence index of $\pi_{\rho_1,\Omega_1,\Omega_1',\epsilon_1}$ is $m-\lambda_1-\mu_1'$. If $\mu_1'<\mu_1-1$, then
\[
\Theta^{\ee\cdot\epsilon}_{n_1-1,n_2}(I^{\o^{\ee\cdot\epsilon}_{2n_1-1}}(\tau_1\otimes \pi_{\rho_1,\Omega_1,\Omega_1',\epsilon_1}))=\Theta^{\ee\cdot\epsilon}_{n_1-1,n_1-\mu_1+k-1}(I^{\o^{\ee\cdot\epsilon}_{2n_1-1}}(\tau_1\otimes \pi_{\rho_1,\Omega_1,\Omega_1',\epsilon_1}))=0.
\]
which implies that
\[
\langle \pi_{\rho,\Omega,\Omega'},I^{\o^{\ee\cdot\epsilon}_{2n_1+1}}((\chi\otimes\tau)\otimes {\pi}_{\rho_1,\Omega_1,\Omega_1',\epsilon_1})\rangle_{\o^\epsilon_{2n_1}(\Fq)}=0 .
\]
This contradicts our assumption. So
\begin{equation}\label{eq7.4}
\mu_1'\ge\mu_1-1.
\end{equation}

Suppose that $\ell= 0$. Now consider the see-saw diagram
\[
\setlength{\unitlength}{0.8cm}
\begin{picture}(20,5)
\thicklines
\put(5.5,4){$\sp_{2(n_2+1)}\times \sp_{2(n_2+1)}$}
\put(6.8,1){$\sp_{2(n_2+1)}$}
\put(12.2,4){$\o^\epsilon_{2(n_1+1)}$}
\put(11.3,1){$\o^{\ee\cdot\epsilon}_{2n_1+1}\times \o^{+}_1$}
\put(7.7,1.5){\line(0,1){2.1}}
\put(12.8,1.5){\line(0,1){2.1}}
\put(8,1.5){\line(2,1){4.2}}
\put(8,3.7){\line(2,-1){4.2}}
\end{picture}
\]
By Proposition \ref{w2} and Proposition \ref{f1}, one has
\[
\begin{aligned}
&\langle I^{\o^\epsilon_{2(n_1+1)}}(\tau_1\otimes\pi_{\rho,\Omega,\Omega'}), \pi_{\rho_1,\Omega_1,\Omega_1',\epsilon_1}\rangle_{\o^{\ee\cdot\epsilon}_{2n_1+1}(\Fq)}\\
\le&\langle\Theta^\epsilon_{n_2+1,n_1+1}(I^{\sp_{2(n_2+1)}}(\tau_1\otimes\pi_{\rho,\Gamma,\Gamma'})), \pi_{\rho_1,\Omega_1,\Omega_1',\epsilon_1}\rangle_{\o^{\ee\cdot\epsilon}_{2n_1+1}(\Fq)}\\
=&\langle I^{\sp_{2(n_2+1)}}(\tau_1\otimes\pi_{\rho,\Gamma,\Gamma'}),\Theta^{\ee\cdot\epsilon}_{n_1+1,n_2+1}(\pi_{\rho_1,\Omega_1,\Omega_1',\epsilon_1})\otimes\omega_{n_2+1}^+\rangle_{\sp_{2(n_2+1)}(\Fq)}\\
\end{aligned}
\]
Similarly, if $\mu_1'<\mu_1-1$, then
\[
\Theta^{\ee\cdot\epsilon}_{n_1+1,n_2+1}(\pi_{\rho_1,\Omega_1,\Omega_1',\epsilon_1})=0.
\]
which implies
\[
\langle \pi_{\rho,\Omega,\Omega'},I^{\o^{\ee\cdot\epsilon}_{2n_1+1}}((\chi\otimes\tau)\otimes {\pi}_{\rho_1,\Omega_1,\Omega_1',\epsilon_1})\rangle_{\o^\epsilon_{2n_1}(\Fq)}=0 .
\]
This contradicts our assumption. So
\begin{equation}\label{eq7.5}
\mu_1'\ge\mu_1-1
\end{equation}

By (\ref{eq7.3}), (\ref{eq7.4}) and (\ref{eq7.5}), there are only two cases for $\mu'$ to be considered: $\mu_1=\mu_1'$ or $\mu_1'=\mu_1-1$.
If
\[
\left\{
\begin{array}{ll}
\langle \pi_{\rho,\Gamma,\Gamma'}\otimes\omega_{n_2' }^{\ee},I^{\sp_{2n_2'}}(\tau\otimes \pi_{\rho_1,\Gamma_1,\Gamma_1'})\rangle_{\sp_{2n_2'}(\Fq)}=0&\textrm{ if }\mu_1'=\mu_1;\\
\langle \pi_{\rho,\Gamma,\Gamma'}\otimes\omega_{n_2 }^{\ee},I^{\sp_{2n_2}}(\tau_1\otimes \pi_{\rho_1,\Gamma_1,\Gamma_1'})\rangle_{\sp_{2n_2}(\Fq)}=0&\textrm{ if }\mu_1'=\mu_1-1\textrm{ and }\ell\ne0;\\
\langle I^{\sp_{2(n_2+1)}}(\tau_1\otimes \pi_{\rho,\Gamma,\Gamma'})\otimes\omega_{n_2+1 }^{\ee},\pi_{\rho_1,\Gamma_1,\Gamma_1'}\rangle_{\sp_{2(n_2+1)}(\Fq)}=0&\textrm{ if }\mu_1'=\mu_1-1\textrm{ and }\ell=0,\\
\end{array}\right.
\]
then
\[
\langle \pi_{\rho,\Omega,\Omega'},I^{\o^{\ee\cdot\epsilon}_{2n_1+1}}((\chi\otimes\tau)\otimes {\pi}_{\rho_1,\Omega_1,\Omega_1',\epsilon_1})\rangle_{\o^\epsilon_{2n_1}(\Fq)}=0 .
\]
So it remains to prove that if
\[
\left\{
\begin{array}{ll}
\langle \pi_{\rho,\Gamma,\Gamma'}\otimes\omega_{n_2' }^{\ee},I^{\sp_{2n_2'}}(\tau\otimes \pi_{\rho_1,\Gamma_1,\Gamma_1'})\rangle_{\sp_{2n_2'}(\Fq)}\ne0&\textrm{ if }\mu_1'=\mu_1;\\
\langle \pi_{\rho,\Gamma,\Gamma'}\otimes\omega_{n_2 }^{\ee},I^{\sp_{2n_2}}(\tau_1\otimes \pi_{\rho_1,\Gamma_1,\Gamma_1'})\rangle_{\sp_{2n_2}(\Fq)}\ne0&\textrm{ if }\mu_1'=\mu_1-1\textrm{ and }\ell\ne0;\\
\langle I^{\sp_{2(n_2+1)}}(\tau_1\otimes \pi_{\rho,\Gamma,\Gamma'})\otimes\omega_{n_2+1 }^{\ee},\pi_{\rho_1,\Gamma_1,\Gamma_1'}\rangle_{\sp_{2(n_2+1)}(\Fq)}\ne0&\textrm{ if }\mu_1'=\mu_1-1\textrm{ and }\ell=0,\\
\end{array}\right.
\]
 then $(\Lambda,\widetilde{\Lambda_1'})\in \mathcal{G}$.

By (\ref{eq7.1}), (\ref{eq7.2}), (\ref{eq7.3}), (\ref{eq7.4}) and (\ref{eq7.5}), if $(\Gamma,\widetilde{\Gamma_1'})\in \mathcal{G}^{\rm{even}+}_{n'',m''}$ for some integers $n''$ and $m''$, then $(\Lambda,\widetilde{\Lambda_1'})\in \mathcal{G}^{\rm{even}+}_{n,m}$.
  Recall that in the case (A), we have $|\Upsilon(\Lambda)^*|+|\Upsilon(\Lambda)_*|+|\Upsilon(\Lambda_1')^*|+|\Upsilon(\Lambda_1')_*|\ne 0$ i.e. $\lambda_1+\lambda_1'+\mu_1+\mu_1'> 0$, which implies that
   \[
   \begin{aligned}
    &|\Upsilon(\Gamma)^*|+|\Upsilon(\Gamma)_*|+|\Upsilon(\Gamma')^*|+|\Upsilon(\Gamma')_*|+|\Upsilon(\Gamma_1)^*|+|\Upsilon(\Gamma_1)_*|+|\Upsilon(\Gamma_1')^*|+|\Upsilon(\Gamma_1')_*|\\
    <&
   |\Upsilon(\Lambda)^*|+|\Upsilon(\Lambda)_*|+|\Upsilon(\Lambda')^*|+|\Upsilon(\Lambda')_*|+|\Upsilon(\Lambda_1)^*|+|\Upsilon(\Lambda_1)_*|+|\Upsilon(\Lambda_1')^*|+|\Upsilon(\Lambda_1')_*|.
   \end{aligned}
   \]
    By Proposition \ref{van}, we know that $(\pi_{\rho,\Gamma,\Gamma'}, \pi_{\rho_1,\Gamma_1,\Gamma_1'})$ is $\epsilon_{-1}$-strongly relevant, and by induction hypothesis, we have
 $(\Gamma,\widetilde{\Gamma_1'})\in \mathcal{G}$. Recall that we now consider the case (A.1), we have $\rm{def}(\Gamma)-1=\rm{def}(\Lambda)-1=\rm{def}(\Lambda_1')=\rm{def}(\Gamma_1')>0$, which implies $(\Gamma,\widetilde{\Gamma_1'})\in \mathcal{G}^{\rm{even}+}_{n'',m''}$.

We now turn to prove the case $\rm{def}(\Lambda_1')= 0$ and $\Lambda_1'\ne\begin{pmatrix}
-\\
-
\end{pmatrix}$. With the same argument in the proof of Proposition \ref{van}, we have
\[
\langle  \pi_{\rho_{},\Lambda_{},\Lambda_{}'}\otimes \omega_n^{\epsilon_0}, I_{P}^{\sp_{2n}}(\tau\otimes\pi_{\rho_1,\Lambda_1,\Lambda_1'})\rangle _{\sp_{2n}(\Fq)}=\langle  \pi_{\rho_{},\Psi_{},\Psi_{}'}\otimes \omega_{n^{\star}}^{\epsilon_0}, I_{P'}^{\sp_{2n^\star}}(\tau\otimes\pi_{\rho_1,\Psi_1,\Psi_1'})\rangle _{\sp_{2n^\star}(\Fq)}
\]
where $n^\star$, $\pi_{\rho_{},\Psi_{},\Psi_{}'}$ and $\pi_{\rho_1,\Psi_1,\Psi_1'}$ are defined in Definition \ref{strongly relevant0}. In the similar manner we can see that
\begin{itemize}
\item $|\Upsilon(\Psi_1')^*|+|\Upsilon(\Psi_1')_*|<|\Upsilon(\Lambda_1')^*|+|\Upsilon(\Lambda_1')_*|$ and $|\Upsilon(\Lambda)^*|+|\Upsilon(\Lambda)_*|=|\Upsilon(\Psi)^*|+|\Upsilon(\Psi)_*|$;
\item $|\Upsilon(\Lambda')^*|+|\Upsilon(\Lambda')_*|+|\Upsilon(\Lambda_1)^*|+|\Upsilon(\Lambda_1)_*|=|\Upsilon(\Psi')^*|+|\Upsilon(\Psi')_*|+|\Upsilon(\Psi_1)^*|+|\Upsilon(\Psi_1)_*|$;
\item
there are $\widetilde{\Lambda_1'}\in\{\Lambda_1',\Lambda_1^{\prime t}\}$ such that $(\Lambda,\widetilde{\Lambda_1'})\in \mathcal{G}$ if and only if $\widetilde{\Psi_1'}\in\{\Psi_1',\Psi_1^{\prime t}\}$ such that $(\Psi,\widetilde{\Psi_1'})\in \mathcal{G}$.
\end{itemize}
Here the difference between this $\rm{def}(\Lambda_0')= 0$ and $\Lambda_0'\ne\begin{pmatrix}
-\\
-
\end{pmatrix}$ case and the above case is that we can write down the symbol $\Gamma_1'$ in the above case, but we can only get $\Psi$ up to a transpose in this case, i.e. we can only write down $\{\Psi,\Psi^t\}$. However, since $\rm{def}(\Psi)=0$, the set of partitions $\{\Upsilon(\Psi_1')^*,\Upsilon(\Psi_1')_*\}$ does not depend on the choice of $\{\Psi,\Psi^t\}$. So  we can still prove this case by our induction assumption.
\end{proof}

We now turn to the Bessel case.
\begin{proposition}\label{bv}
Keep the assumptions in Theorem \ref{bess}. Assume that $(\pi_{\rho,\Omega,\Omega',\epsilon''},\pi_{\rho_1,\Omega_1,\Omega_1'})$ is strongly relevant.

(i) Assume that $n\ge m$. If
\[
\langle  \pi_{\rho_1,\Omega_1,\Omega_1',\epsilon''},I^{\o^\e_{2(n+1)}}_P(\tau\otimes\pi_{\rho,\Omega,\Omega'})\rangle _{\rm{O}^{\epsilon}_{2n+1}(\Fq)}\ne0.
\]
then  there are $ \widetilde{\Omega}\in\{\Omega,\Omega^{t}\}$ and $\widetilde{\Omega'}\in\{\Omega',\Omega^{\prime t}\}$ such that $(\Omega_1,\widetilde{\Omega})$ and $(\Omega_1',\widetilde{\Omega'})\in \mathcal{G}$.

(ii)
Assume that $n\le m$. If
\[
\langle I^{\o^\epsilon_{2m+1}}_P( \tau\otimes\pi_{\rho,\Omega,\Omega',\epsilon''}),\pi_{\rho_1,\Omega_1,\Omega_1'}\rangle _{\rm{O}^{\epsilon'}_{2m}(\Fq)}\ne0.
\]
then there are $ \widetilde{\Omega}\in\{\Omega,\Omega^{t}\}$ and $\widetilde{\Omega'}\in\{\Omega',\Omega^{\prime t}\}$ such that $(\Omega_1,\widetilde{\Omega})$ and $(\Omega_1',\widetilde{\Omega'})\in \mathcal{G}$.
\end{proposition}
\begin{proof}
We will only prove (ii) for $\epsilon=\e$. The rest of the proof is similar and will be left to the reader.

Let $ \pi_{\rho_1,\Omega_1,\Omega_1',\epsilon''}\in\cal{E}(\o^\epsilon_{2n+1},\pi_{\rho_1',k_1,h_1,\epsilon''})$ and $\pi_{\rho,\Omega,\Omega'}\in\cal{E}(\o^\e_{2m},\pi_{\rho',k,h})$ where $\pi_{\rho'_1,k_1,h_1,\epsilon''}$ and $\pi_{\rho',k,h}$ are two cuspidal representations of $\o^\epsilon_{2n'+1}\fq$ and $\o^\e_{2m'}\fq$, respectively. Let $m_0$ and $m_0'$ be the first occurrence index of $\pi_{\rho',k,h}$ and $\sgn\pi_{\rho',k,h}$, respectively. Note that
\[
\langle  \pi_{\rho_1,\Omega_1,\Omega_1',\epsilon''},I^{\o^\e_{2(n+1)}}_P(\tau\otimes\pi_{\rho,\Omega,\Omega'})\rangle _{\rm{O}^{\epsilon}_{2n+1}(\Fq)}=\langle  \sgn\pi_{\rho_1,\Omega_1,\Omega'_1,\epsilon''},I^{\o^\e_{2(n+1)}}_P(\tau\otimes(\sgn\pi_{\rho,\Omega,\Omega'}))\rangle _{\rm{O}^{\epsilon}_{2n+1}(\Fq)}
\]
and
if $m_0\ge m'$, then $m_0'<m'$.
So we only need to prove the case $m_0\ge m'$.
Recall that by Proposition \ref{q1} and Proposition \ref{o3}, there exists an integer $N$ with the following property: for any irreducible representation $\prll$ such that $\pi_{\rho,\Omega,\Omega'}\otimes\prll$ appears in $\omega^\epsilon_{m,m^*}$, we have $N=\rm{def}(\Lambda)$.
Since $(\pi_{\rho,\Omega,\Omega',\epsilon''},\pi_{\rho_1,\Omega_1,\Omega_1'})$ is strongly relevant, by Corollary \ref{strongly relevant1}, there are four possibilities:
\begin{itemize}
\item[]
Case (1): $N>0$ and $k_1=|k|-1 $;
\item[]
Case (2): $N>0$ and $k_1=|k|$;
\item[]
Case (3): $N<0$ and $k_1=|k|-1 $;
\item[]
Case (4): $N<0$ and $k_1=|k|$.
\end{itemize}
We will only prove the Case (1), which is correspondence to Case (A.1) in the proof of Proposition \ref{van1}. And we only prove the case either $\rm{def}(\Omega_1)\ne 0$ or $\Omega_1=\begin{pmatrix}
-\\
-
\end{pmatrix}$ in the Case (1). The proof of the rest cases is similar and will be left to the reader.

Note that in this case $k$ and $k_1$ are even, and $N=2|k|+1$. With the same argument of the proof of Proposition \ref{van1}, the following hold.

 \begin{itemize}
\item
For any irreducible representation $\prll$ of $\sp_{2m^*}\fq$, if $\prll\otimes\pi_{\rho,\Omega,\Omega'}$ appears in $\omega^\epsilon_{m^*,m}$, then there is a symbol $\widetilde{\Omega}\in\{\Omega,\Omega^t\}$ such that $(\Lambda,\widetilde{\Omega})\in \cal{B}^+_{\rm{rk}(\Lambda),\rm{rk}(\widetilde{\Omega})}$.

\item
For any irreducible representation $\pi_{\rho_1,\Lambda_1,\Lambda_1'}$ of $\sp_{2n^*}\fq$, if $\pi_{\rho_1,\Lambda_1,\Lambda_1'}\otimes\pi_{\rho_1,\Omega_1,\Omega_1',\epsilon''}$ appears in $\omega^\epsilon_{n^*,n}$, then there is a symbol $\widetilde{\Lambda_1'}\in\{\Lambda_1',\Lambda_1^{\prime t}\} $ with $\rm{def}(\widetilde{\Lambda_1'})>0$ such that $(\Omega_1,\widetilde{\Lambda_1'})\in  \cal{B}^+_{\rm{rk}(\Omega_1),\rm{rk}(\widetilde{\Lambda_1'})}$.

\end{itemize}

As before, we suppress various Levi subgroups from the parabolic induction.
Write $\Upsilon(\widetilde{\Omega})=\begin{bmatrix}
\mu\\
\lambda
\end{bmatrix}$ and
$\Upsilon(\Omega_1)=\begin{bmatrix}
\mu'\\
\lambda'
\end{bmatrix}$.
Let $M>n+m$ be an integer, and let $\lambda^0=[M,\lambda]$ and $\lambda^{0\prime}=[M,\lambda']$ be two partitions. By Theorem \ref{p1}, Theorem \ref{p2}, Proposition \ref{f1} and above discussion, there exists an irreducible representation $\prll$ of $\sp_{2m^*}\fq$ such that
 \begin{itemize}
\item
$\Upsilon(\Lambda)=\begin{bmatrix}
\lambda^0\\
\mu
\end{bmatrix}$
and $\rm{def}(\Lambda)=2|k|+1$;
\item
$\Lambda'\in\{\Omega',\Omega^{\prime t}\}$;
\item
$m$ is the first occurrence index of $\prll$ in the Witt tower ${\bf O}^\epsilon_{\rm{even}}$ and $\Theta^\epsilon_{m^*,m}(\prll)=\pi_{\rho,\Omega,\Omega'}$;
\item
$m^*-m=M+|k|$.
\end{itemize}
and an irreducible representation $\pi_{\rho_1,\Lambda_1,\Lambda_1'}$ of $\sp_{2n^*}\fq$ such that
 \begin{itemize}
\item
There exist $\widetilde{\Lambda_1'}\in\{\Lambda_1',\Lambda_1^{\prime t}\}$ such that $\Upsilon(\widetilde{\Lambda_1'})=\begin{bmatrix}
\lambda^{0 \prime}\\
\mu'
\end{bmatrix}$
and $\rm{def}(\widetilde{\Lambda_1'})=2(k_1+1)=2|k|>0$;
\item
$\Lambda_1'\in\{\Omega_1',\Omega_1^{\prime t}\}$ ;
\item
$n$ is the first occurrence index of $\pi_{\rho_1,\Lambda_1,\Lambda_1'}$ in the Witt tower ${\bf O}^{\ee\cdot\epsilon}_{\rm{odd}}$ and $\Theta^{\ee\cdot\epsilon}_{n^*,n}(\pi_{\rho_1,\Lambda_1,\Lambda_1'})=\pi_{\rho_1,\Omega_1,\Omega_1',\epsilon''}$;
\item
$n^*-n=M+k_1=M+|k|$.
\end{itemize}

Consider the see-saw diagram
\[
\setlength{\unitlength}{0.8cm}
\begin{picture}(20,5)
\thicklines
\put(6.5,4){$\sp_{2m^*}\times \sp_{2m^*}$}
\put(7.3,1){$\sp_{2m^*}$}
\put(12.3,4){$\o^{\ee\cdot\epsilon}_{2m+1}$}
\put(11.6,1){$\o^\epsilon_{2m_1}\times \o^{\ee}_1$}
\put(7.7,1.5){\line(0,1){2.1}}
\put(12.8,1.5){\line(0,1){2.1}}
\put(8,1.5){\line(2,1){4.2}}
\put(8,3.7){\line(2,-1){4.2}}
\end{picture}
\]

By Proposition \ref{w2}, one has
\[
\begin{aligned}
&\langle \prll\otimes\omega_{m^*}^{\ee},I^{\sp_{2m}}(\tau\otimes \pi_{\rho_1,\Lambda_1,\Lambda_1'})\rangle_{\sp_{2m^*}(\Fq)}\\
\le&\langle \Theta^\epsilon_{m,m^*}(\pi_{\rho,\Omega,\Omega'})\otimes\omega_{m^*}^{\ee},I^{\sp_{2m^*}}(\tau\otimes \pi_{\rho_1,\Lambda_1,\Lambda_1'})\rangle_{\sp_{2m^*}(\Fq)},\\
=&\langle \pi_{\rho,\Omega,\Omega'},\Theta^{\ee\cdot\epsilon}_{m^*,m}(I^{\sp_{2m^*}}(\tau\otimes \pi_{\rho_1,\Lambda_1,\Lambda_1'}))\rangle_{\o^\epsilon_{2m^*_1}(\Fq)}
\end{aligned}
\]
With same argument in the proof of Proposition \ref{7bp}, $(\prll,\pi_{\rho_1,\Lambda_1,\Lambda_1'})$ is $\ee$-strongly relevant. So $(\prll,\pi_{\rho_1,\Lambda_1,\Lambda_1'})$ is a pair of representations satisfying the conditions in Proposition \ref{van1}.
On the other hand, by Corollary \ref{ctheta2}, Proposition \ref{f1} and the proof of Proposition \ref{van1}, one has
\[
\begin{aligned}
&\langle \Theta^\epsilon_{m,m^*}(\pi_{\rho,\Omega,\Omega'})\otimes\omega_{m^*}^{\ee},I^{\sp_{2m^*}}(\tau\otimes \pi_{\rho_1,\Lambda_1,\Lambda_1'})\rangle_{\sp_{2m^*}(\Fq)}\\
=&\langle \pi_{\rho,\Lambda,\Lambda'}\otimes\omega_{m^* }^{\ee},I^{\sp_{2m^*}}(\tau\otimes \pi_{\rho_1,\Lambda_1,\Lambda_1'})\rangle_{\sp_{2m^*}(\Fq)}+\bigoplus_{\Lambda''}\langle \pi_{\rho,\Lambda'',\Lambda'}\otimes\omega_{m^* }^{\ee},I^{\sp_{2m^*}}(\tau\otimes \pi_{\rho_1,\Lambda_1,\Lambda_1'})\rangle_{\sp_{2m^*}(\Fq)}
\end{aligned}
\]
where $(\Upsilon(\Lambda'')^*)_1>M$. Note that $\rm{def}(\Lambda'')=\rm{def}(\Lambda)=2|k|+1>0$ and $|\rm{def}(\Lambda_1')|=\rm{def}(\Lambda)-1$. So If $(\Lambda'',\hat{\Lambda_1'})\in\cal{G}$ with $\hat{\Lambda_1'}\in\{\Lambda_1',\Lambda_1^{\prime t}\}$, then $(\Lambda'',\hat{\Lambda_1'})\in\cal{G}^{\rm{even},+}_{m^*,n^*}$.
Since $\Upsilon(\hat{\Lambda_1'})^*_1\le\rm{max}\{M,\mu_1'\}=M$ and $(\Upsilon(\Lambda'')^*)_1>M$, one has $\Upsilon(\Lambda'')^*\npreceq\Upsilon(\hat{\Lambda_1'})^*$. So $(\Lambda'',\Lambda_1')\notin \mathcal{G}$ and $(\Lambda'',\Lambda_1^{\prime t})\notin \mathcal{G}$. Then by Proposition \ref{van1},
\[
\langle \pi_{\rho,\Lambda'',\Lambda'}\otimes\omega_{m^* }^{\ee},I^{\sp_{2m^*}}(\tau\otimes \pi_{\rho_1,\Lambda_1,\Lambda_1'})\rangle_{\sp_{2m^*}(\Fq)}=0.
\]
Hence, by Proposition \ref{w2}, we have
\begin{equation}\label{bv1}
\begin{aligned}
&\langle \pi_{\rho,\Omega,\Omega'},I^{\o^{\ee\cdot\epsilon}_{2m+1}}((\chi\otimes\tau)\otimes\pi_{\rho_1,\Omega_1,\Omega_1',\epsilon''})\rangle_{\o^\epsilon_{2m}(\Fq)}\\
=&\langle \pi_{\rho,\Omega,\Omega'},\Theta^{\ee\cdot\epsilon}_{m^*,m}(I^{\sp_{2m^*}}(\tau\otimes \pi_{\rho_1,\Lambda_1,\Lambda_1'}))\rangle_{\o^\epsilon_{2m}(\Fq)}\\
=&\langle \Theta^\epsilon_{m,m^*}(\pi_{\rho,\Omega,\Omega'})\otimes\omega_{m^*}^{\ee},I^{\sp_{2m^*}}(\tau\otimes \pi_{\rho_1,\Lambda_1,\Lambda_1'})\rangle_{\sp_{2m^*}(\Fq)}\\
=&\langle \pi_{\rho,\Lambda,\Lambda'}\otimes\omega_{m^* }^{\ee},I^{\sp_{2m^*}}(\tau\otimes \pi_{\rho_1,\Lambda_1,\Lambda_1'})\rangle_{\sp_{2m^*}(\Fq)},\\
\end{aligned}
\end{equation}
which completes the proof by Proposition \ref{van1}.
\end{proof}

To finish the proof of Theorem \ref{9.1}, it remains to prove the following result.

\begin{proposition}\label{fnv}
Keep the assumptions in Theorem \ref{9.1}. Assume that $(\prll,\pi_{\rho_1,\Lambda_1,\Lambda_1'})$ is $\epsilon_0$-strongly relevant.
Assume that $n\ge m$. If there are $\widetilde{\Lambda_1'}\in\{\Lambda_1',\Lambda_1^{\prime t}\}$ and $\widetilde{\Lambda'}\in\{\Lambda',\Lambda^{\prime t}\}$ such that $(\Lambda,\widetilde{\Lambda_1'})$ and $(\Lambda_1,\widetilde{\Lambda'})\in \mathcal{G}$, then we have
\[
\langle  \prll\otimes \omega_n^{\ee}, I_{P}^{\sp_{2n}}(\tau\otimes\pi_{\rho_1,\Lambda_1,\Lambda_1'})\rangle _{\sp_{2n}(\Fq)}=m_\psi(\pw_{\rho} ,\pw_{\rho_1}).
\]

\end{proposition}

\begin{proof}
 As before, we suppress various Levi subgroups from the parabolic induction. We also prove the proposition by induction on
\[
r=|\Upsilon(\Lambda)^*|+|\Upsilon(\Lambda)_*|+|\Upsilon(\Lambda')^*|+|\Upsilon(\Lambda')_*|+|\Upsilon(\Lambda_1)^*|+|\Upsilon(\Lambda_1)_*|+|\Upsilon(\Lambda_1')^*|+|\Upsilon(\Lambda_1')_*|.
 \]
 For $r=0$, it is Theorem \ref{8.6}. Keep the notations
 \[
 \epsilon,k,h,k_1,h_1,\widetilde{\Lambda_1'},\lambda,\mu,\lambda',\mu',\Omega,\Omega',\Omega_1,\Omega_1',\Gamma,\Gamma',\widetilde{\Gamma'},\Gamma_1,\Gamma_1',n_1,n_1',n_2,n_2'
  \]
  in the proof of Proposition \ref{van1}. We will only prove the case (A.1) in the proof of Proposition \ref{van1} with the assumption either $\rm{def}(\Lambda_1')\ne 0$ or $\Lambda_1'=\begin{pmatrix}
-\\
-
\end{pmatrix}$, and $\lambda_1=\lambda_1'$ and $\epsilon_0=\ee$. The proof of the $\rm{def}(\Lambda_1')= 0$ and $\Lambda_1'\ne\begin{pmatrix}
-\\
-
\end{pmatrix}$ case is similar.

Since $\rm{def}(\Lambda)=2k+1$ and $|\rm{def}(\Lambda_1')|=2|h_1|=2k=\rm{def}(\Lambda)-1$, for any $\widehat{\Lambda_1'}\in \{\Lambda_1',\Lambda_1^{\prime t}\} $, we have
\[
\left\{
\begin{array}{ll}
(\Lambda,\widehat{\Lambda_1'})\in\cal{G}, &\textrm{ if }\widehat{\Lambda_1'}=\widetilde{\Lambda_1'};\\
(\Lambda,\widehat{\Lambda_1'})\notin\cal{G},&\textrm{ if }\widehat{\Lambda_1'}\ne\widetilde{\Lambda_1'}.
\end{array}
\right.
\]

Consider the see-saw diagram
\[
\setlength{\unitlength}{0.8cm}
\begin{picture}(20,5)
\thicklines
\put(6.5,4){$\sp_{2n}\times \sp_{2n}$}
\put(7.3,1){$\sp_{2n}$}
\put(12.3,4){$\o^{\ee\cdot\epsilon}_{2n_1+1}$}
\put(11.6,1){$\o^\epsilon_{2n_1}\times \o^{\ee}_1$}
\put(7.7,1.5){\line(0,1){2.1}}
\put(12.8,1.5){\line(0,1){2.1}}
\put(8,1.5){\line(2,1){4.2}}
\put(8,3.7){\line(2,-1){4.2}}
\end{picture}
\]

By Corollary \ref{ctheta2}, Proposition \ref{f1}, one has
\[
\begin{aligned}
&\langle \Theta^\epsilon_{n_1,n}(\pi_{\rho,\Omega,\Omega'})\otimes\omega_{n}^{\ee},I^{\sp_{2n}}(\tau\otimes \pi_{\rho_1,\Lambda_1,\Lambda_1'})\rangle_{\sp_{2n}(\Fq)}\\
=&\langle \pi_{\rho,\Lambda,\Lambda'}\otimes\omega_{n }^{\ee},I^{\sp_{2n}}(\tau\otimes \pi_{\rho_1,\Lambda_1,\Lambda_1'})\rangle_{\sp_{2n}(\Fq)}+\bigoplus_{\Lambda''}\langle \pi_{\rho,\Lambda'',\Lambda'}\otimes\omega_{n }^{\ee},I^{\sp_{2n}}(\tau\otimes \pi_{\rho_1,\Lambda_1,\Lambda_1'})\rangle_{\sp_{2n}(\Fq)}
\end{aligned}
\]
with $(\Upsilon(\Lambda'')^*)_1>\lambda_1=\lambda_1'$. By the proof of Proposition \ref{bv}, one has
\[
\langle \pi_{\rho,\Lambda'',\Lambda'}\otimes\omega_{n }^{\ee},I^{\sp_{2n}}(\tau\otimes \pi_{\rho_1,\Lambda_1,\Lambda_1'})\rangle_{\sp_{2n}(\Fq)}=0,
\]
and
\[
\begin{aligned}
&\langle \pi_{\rho,\Lambda,\Lambda'}\otimes\omega_{n }^{\ee},I^{\sp_{2n}}(\tau\otimes \pi_{\rho_1,\Lambda_1,\Lambda_1'})\rangle_{\sp_{2n}(\Fq)}=\langle \pi_{\rho,\Omega,\Omega'},I^{\o^{\ee\cdot\epsilon}_{2n_1+1}}((\chi\otimes\tau)\otimes\pi_{\rho_1,\Omega_1,\Omega_1',\epsilon_1})\rangle_{\o^\epsilon_{2n_1}(\Fq)}.\\
\end{aligned}
\]

Using the same see-saw arguments of Proposition \ref{van1}, and by similar arguments in the proof of Proposition \ref{bv}, we have
\[
\begin{aligned}
&\langle \pi_{\rho,\Lambda,\Lambda'}\otimes\omega_{n }^{\ee},I^{\sp_{2n}}(\tau\otimes \pi_{\rho_1,\Lambda_1,\Lambda_1'})\rangle_{\sp_{2n}(\Fq)}\\
=&\langle \pi_{\rho,\Omega,\Omega'},I^{\o^\epsilon_{2n_1+1}}((\chi\otimes\tau)\otimes\pi_{\rho_1,\Omega_1,\Omega_1',\epsilon_1})\rangle_{\o^\epsilon_{2n_1}(\Fq)}\\
=&
\left\{
\begin{array}{ll}
\langle \pi_{\rho,\Gamma,\Gamma'}\otimes\omega_{n_2' }^{\ee},I^{\sp_{2n_2'}}(\tau\otimes \pi_{\rho_1,\Gamma_1,\Gamma_1'})\rangle_{\sp_{2n_2'}(\Fq)}&\textrm{ if }\mu_1'=\mu_1;\\
\langle \pi_{\rho,\Gamma,\Gamma'}\otimes\omega_{n_2 }^{\ee},I^{\sp_{2n_2}}(\tau_1\otimes \pi_{\rho_1,\Gamma_1,\Gamma_1'})\rangle_{\sp_{2n_2}(\Fq)}&\textrm{ if }\mu_1'=\mu_1-1\textrm{ and }\ell\ne0;\\
\langle I^{\sp_{2(n_2+1)}}(\tau_1\otimes \pi_{\rho,\Gamma,\Gamma'})\otimes\omega_{n_2 +1 }^{\ee},\pi_{\rho_1,\Gamma_1,\Gamma_1'}\rangle_{\sp_{2(n_2+1)}(\Fq)}&\textrm{ if }\mu_1'=\mu_1-1\textrm{ and }\ell=0,\\
\end{array}\right.
\end{aligned}
\]
Note that $|\lambda^2|+|\mu^2|+|\lambda^{\prime 2}|+|\mu^{\prime 2}|=|\lambda|+|\mu|+|\lambda'|+|\mu'|$ if and only if $|\lambda|+|\mu|+|\lambda'|+|\mu'|=0$. By induction hypothesis, the right-hand side is equal to $m_\psi(\pw_{\rho} ,\pw_{\rho'})$. Then
\[
\langle  \prll\otimes \omega_n^{\ee}, I_{P}^{\sp_{2n}}(\tau\otimes\pi_{\rho_1,\Lambda_1,\Lambda_1'})\rangle _{\sp_{2n}(\Fq)}=m_\psi(\pw_{\rho} ,\pw_{\rho'})
\]
\end{proof}

The non-vanishing result for the Bessel case follows immediately from (\ref{bv1}) and Proposition \ref{fnv}.

\end{document}